\documentclass[oneside,11pt]{report}

\usepackage{amssymb}
\usepackage{graphicx}
\usepackage{fancybox}
\usepackage{fullpage}
\usepackage{amsthm,amsmath, mathtools, mathrsfs}
\usepackage{tikz, array, multirow, rotating}
\usepackage[T1]{fontenc}
\usepackage[hang,flushmargin]{footmisc}
\usepackage{hyperref}
\allowdisplaybreaks
\definecolor{darkblue}{rgb}{0, 0, .6}
\hypersetup{
	colorlinks=true,
	linkcolor=darkblue,
	anchorcolor=darkblue,
	citecolor=darkblue,
	urlcolor=darkblue,
	linktocpage=true,
}

\usetikzlibrary[positioning,arrows, patterns]

\tikzstyle{smver}=[rectangle, draw, thick, minimum width=.5cm, minimum
height=1cm]

\tikzstyle{smhor}=[rectangle, draw, thick, minimum width=1cm, minimum
height=.5cm]

\tikzstyle{smdver}=[rectangle, draw, dashed, thick, minimum width=.5cm, minimum
height=1cm, color=shade]

\tikzstyle{smdhor}=[rectangle, draw, dashed, thick, minimum width=1cm, minimum
height=.5cm, color=shade]

\definecolor{shade}{rgb}{0.8,0.8,0.8}

\tikzstyle{shver}=[rectangle, draw, thick, minimum width=.5cm, minimum
height=1cm, fill=shade]

\tikzstyle{shhor}=[rectangle, draw, thick, minimum width=1cm, minimum
height=.5cm, fill=shade]

\tikzstyle{bver}=[rectangle, draw, thick, minimum width=1cm, minimum
height=2cm]

\tikzstyle{bhor}=[rectangle, draw, thick, minimum width=2cm, minimum
height=1cm]

\tikzstyle{fix}=[rectangle, fill=shade, minimum width=.5cm, minimum
height=.5cm]

\tikzstyle{bfix}=[rectangle, fill=shade, minimum width=1cm, minimum
height=1cm]

\tikzstyle{r} =[rectangle, draw, dashed, minimum width=.5cm, minimum
height=.5cm]

\tikzstyle{redr} =[rectangle, draw, dashed, ,thick, color=red, minimum width=.5cm, minimum
height=.5cm]

\tikzstyle{bredr} =[rectangle, draw, dashed, ,thick, color=red, minimum width=1cm, minimum
height=1cm]

\newtheorem{theorem}{Theorem}[section]
\newtheorem*{theorem*}{Theorem}
\newtheorem{lemma}[theorem]{Lemma}

\newtheorem{corollary}[theorem]{Corollary}
\newtheorem{proposition}[theorem]{Proposition}

\theoremstyle{definition}
\newtheorem{definition}[theorem]{Definition}
\newtheorem{example}[theorem]{Example}
\newtheorem{remark}[theorem]{Remark}

\numberwithin{equation}{section}

\DeclareMathOperator{\shape}{Shape}
\DeclareMathOperator{\sign}{sign}
\DeclareMathOperator{\supp}{supp}

\newcommand{\expl}[2]{{}^{#2}{#1}} 

\begin{document}

\setlength{\parindent}{0pt}

\pagenumbering{Alph}
\begin{titlepage}
\ 
\vspace{5cm}

{\huge \textbf{Leading Coefficients of Kazhdan--Lusztig\\
Polynomials in Type $D$\let\thefootnote\relax\footnote{This is a revised version of the author's Ph.D. thesis, which was directed by Richard M. Green at the University of Colorado Boulder.  See \hyperref[comments]{Comments} page for a complete list of revisions.}
}}

\bigskip

Ph.D. Thesis, University of Colorado Boulder, 2013

\bigskip

\bigskip

\href{http://www.tysongern.com}{\Large Tyson C. Gern}\\
\href{mailto:tyson.gern@colorado.edu}{{tyson.gern@colorado.edu}}

\end{titlepage}

\pagenumbering{roman}
\pagestyle{plain}


\chapter*{Abstract}\normalsize
\addcontentsline{toc}{chapter}{Abstract}
  Kazhdan--Lusztig polynomials arise in the context of Hecke algebras
  associated to Coxeter groups. The computation of these polynomials
  is very difficult for examples of even moderate rank. In type $A$ it
  is known that the leading coefficient, $\mu(x,w)$ of a
  Kazhdan--Lusztig polynomial $P_{x,w}$ is either 0 or 1 when $x$ is
  fully commutative and $w$ is arbitrary. In type $D$ Coxeter groups
  there are certain ``bad'' elements that make $\mu$-value computation
  difficult.

  The Robinson--Schensted correspondence between the symmetric group
  and pairs of standard Young tableaux gives rise to a way to compute
  cells of Coxeter groups of type $A$. A lesser known correspondence
  exists for signed permutations and pairs of so-called domino
  tableaux, which allows us to compute cells in Coxeter groups of
  types $B$ and $D$.  I will use this correspondence in type $D$ to
  compute $\mu$-values involving bad elements. I will conclude by
  showing that $\mu(x,w)$ is 0 or 1 when $x$ is fully commutative in
  type $D$.

\chapter*{Comments}\label{comments}\normalsize
\addcontentsline{toc}{chapter}{Comments}

This is a revised version of the author's Ph.D. thesis, which was directed by Richard M. Green at the University of Colorado Boulder.  The numbering of definitions, theorems, remarks, and examples is identical in both versions. This version was typeset using the \texttt{report} document class instead of the University of Colorado \texttt{thesis} document class.  As a result, there were some modifications in formatting, all of which were cosmetic.

\chapter*{Acknowledgements}\normalsize
\addcontentsline{toc}{chapter}{Acknowledgements}

I would like to extend my thanks to my advisor Richard M. Green for
his help and support in producing this thesis. His patience,
diligence, and mathematical experience have been invaluable to me. I
am grateful to the members of my thesis defense committee members for
their helpful feedback and comments. I would also like to express my
gratitude to my family and friends for their support, and
encouragement.

\tableofcontents

\chapter{Coxeter groups}

\pagenumbering{arabic}

In their seminal paper~\cite{kazhdan1979representations}, Kazhdan and
Lusztig defined remarkable polynomials, $P_{x,w}$ indexed by elements
$x$ and $w$ of an arbitrary Coxeter group $W$. These polynomials are
called Kazhdan--Lusztig polynomials, and are important in algebra and
geometry. For example, they give rise to representations of both the
Coxeter group and its corresponding Hecke algebra. Unfortunately,
these polynomials are particularly difficult to compute, even for
relatively small Coxeter groups. A bound on the degree of $P_{x,w}$
is known, but it unknown when this bound is achieved, in general.  Of
particular importance are the coefficients $\mu(x,w)$ of the highest
possible degree term.  The polynomials $P_{x,w}$ and the $\mu$-values
are defined by recurrence relations, but there is no known algorithm
that allows for their efficient computation, even in groups of
relatively small rank.

For many years computational evidence suggested that the values
$\mu(x,w)$ were always 0 or 1 in Coxeter groups of type $A$. This
conjecture, known as the \textit{0-1 Conjecture}, was shown to be
false by McLarnan and Warrington in
\cite{mclarnan2003counterexamples}. However, empirical evidence
suggests that $\mu(x,w)\in \{0,1\}$ in many cases. For example, in
type $A_n$ it is known that $\mu(x,w)\in \{0,1\}$ if one of the
following holds:
\begin{enumerate}
\item $n \leq 8$ \cite{mclarnan2003counterexamples};
\item $a(x) < a(w)$ \cite{xi2005leading}, where $a$ is Lusztig's
  $a$-function, to be discussed in Section~\ref{afunction};
\item $x$ is fully commutative \cite{green2009leading}.
\end{enumerate}

In this thesis we work in type $D$. Coxeter groups of type $D$ have
so-called \textit{bad} elements whose descent sets have undesirable
properties.  These properties make computing $\mu(x,w)$ difficult when
$w$ is bad and $x$ is fully commutative in the sense of Stembridge
\cite{stembridge1996fully}. We compute $\mu$-values involving
these bad elements and use these calculations to prove our main result
in Theorem~\ref{mainresult}:
\begin{theorem*}
  Let $x,w\in W(D_n)$ be such that $x$ is fully commutative. Then
  $\mu(x,w)\in\{0,1\}$.
\end{theorem*}
We will only rely on computer calculations for two computations
of $\mu$-values in Coxeter groups of small rank.

\section{Basic properties}

We begin with a short overview of the basic properties of Coxeter
groups.  The following definitions are from
\cite{bjorner2005combinatorics} and \cite{humphreys1992reflection}.

\begin{definition}
  A \textbf{Coxeter system} is an ordered pair $(W,S)$ consisting of a
  \textbf{Coxeter group} $W$ generated by a set $S$
  with presentation
  \[
  \langle S\mid (st)^{m_{(s,t)}}=1, s,t \in S,
    m(s,t)\in\mathbb{N}\cup\{\infty\}\rangle,
  \]
  where $m(s,t)=1$ if $s=t$ and $m(s,t)=m(t,s)\geq 2$ if $s\neq t$. If
  there is no relation between a pair $s,t\in S$ we say
  $m(s,t)=\infty$. If $m(s,t)\leq 3$ for all $s,t\in S$ we say that
  $(W,S)$ is \textbf{simply laced}.
  
  If $S$ is finite then we say that $(W,S)$ is a
  Coxeter system of \textbf{rank} $|S|$.
\end{definition}

\begin{example}
  The dihedral group of order 8 is a Coxeter group with presentation
  \[
  \langle \{s_1,s_2\}\mid s_1^2 = s_2^2 = (s_1s_2)^4 = 1 \rangle.
  \]
\end{example}

\begin{definition}
  Let $s,t\in S$ be such that $s\neq t$. We call each relation
  $(st)^{m(s,t)} = 1$ a \textbf{braid relation}.  Note that each braid
  relation may be rewritten as
  \[
  \underbrace{sts\cdots}_{m(s,t)\text{ factors}} =
  \underbrace{tst\cdots}_{m(s,t)\text{ factors}}.
  \]
  In particular, if $m(s,t)=2$ then $st=ts$, so $s$ and $t$ commute.
  If $m(s,t)=2$ we call the relation a \textbf{short braid relation}.
  If $m(s,t)\geq 3$ we call the relation a \textbf{long braid
    relation}.
\end{definition}

We encode the information contained in the presentation of a Coxeter
system into a picture called a Coxeter graph. We will use $\mathbf{n}$
to denote the set
\[
\{1,2,3,4,\dots,n\}.
\]

\begin{definition}
  Let $(W,S)$ be a Coxeter system.  A \textbf{Coxeter graph} is a
  graph $\Gamma$ with vertex set $S$.  We join $s,t\in S$ by an edge
  labeled $m(s,t)$ whenever $m(s,t)\geq 3$.  As a convention we omit
  the label when $m(s,t)=3$.
\end{definition}

\begin{example}\label{typeaintro}
  Let $n\in\mathbb{N}$ and let $(W,S)$ be a Coxeter system with
  $S=\{s_1,s_2,\dots,s_n\}$, and Coxeter diagram shown below.
  
  
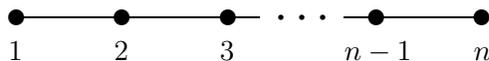
\begin{center} 
    \begin{picture}(176,26)
      \put(000,16){\circle*{6}}
      \put(040,16){\circle*{6}}
      \put(080,16){\circle*{6}}
      \put(136,16){\circle*{6}}
      \put(176,16){\circle*{6}}
      
      \put(000,16){\line(1,0){40}}
      \put(040,16){\line(1,0){40}}
      \put(080,16){\line(1,0){12}}
      \put(124,16){\line(1,0){12}}
      \put(136,16){\line(1,0){40}}
      
      \put(100,16){\circle*{2}}
      \put(108,16){\circle*{2}}
      \put(116,16){\circle*{2}}  
      
      \put(-3,00){1}
      \put(37,00){2}
      \put(77,00){3}
      \put(124,0){$n-1$}
      \put(173,0){$n$}    
    \end{picture}
  \end{center} 
  
  Such a Coxeter group is said to be of type $A_n$, and we write
  $W=W(A_n)$.  The symmetric group on $n+1$ elements, $S_{n+1}$ is a
  Coxeter group of type $A_n$ since if we let $W=S_{n+1}$,
  $s_i=(i,i+1)$ and $S=\{s_i\}_{i=1}^n$, then $(W,S)$ is a Coxeter
  system of type $A_n$ \cite[Proposition
  1.5.4]{bjorner2005combinatorics}.
\end{example}

\begin{example}\label{typedintro}
  Let $(W,S)$ be a Coxeter system with $S=\{s_1,s_2,\dots,s_n\}$, and
  with Coxeter diagram given below.
  
  
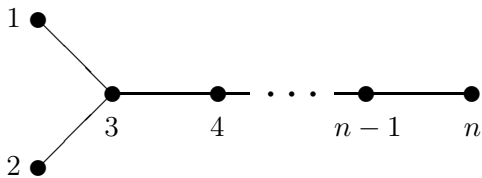
\begin{center} 
    \begin{picture}(164,56)
      \put(000,56){\circle*{6}}
      \put(000,00){\circle*{6}}
      \put(028,28){\circle*{6}}
      \put(068,28){\circle*{6}}
      \put(124,28){\circle*{6}}
      \put(164,28){\circle*{6}}
      
      \put(000,56){\line(1,-1){28}}
      \put(000,00){\line(1,01){28}}
      \put(028,28){\line(1,00){40}}
      \put(068,28){\line(1,00){12}}
      \put(112,28){\line(1,00){12}}
      \put(124,28){\line(1,00){40}}
      
      \put(088,28){\circle*{2}}
      \put(096,28){\circle*{2}}
      \put(104,28){\circle*{2}}  
      
      \put(-12, 53){1}
      \put(-12, -3){2}
      \put(25,  12){3}
      \put(65,  12){4}
      \put(112, 12){$n-1$}
      \put(161, 12){$n$}    
    \end{picture}
  \end{center} 
  
  Such a Coxeter group is said to be of type $D_n$, and we write
  $W=W(D_n)$. The wreath product $\mathbb{Z}_2\wr S_n$ consists of all
  bijections $\sigma$ of the set $\{i\mid\pm i\in{\bf n}\}$ such that
  $\sigma(-a)=-\sigma(a)$.  This is called the signed permutation
  group, and is isomorphic to the Coxeter group of type $B_n$
  \cite[Proposition 8.1.3]{bjorner2005combinatorics}. The group
  $W(D_n)$ is an index 2 subgroup of the signed permutation group,
  consisting of all elements with an even number of sign changes,
  under the embedding
  \[
  s_i\mapsto \begin{cases}
    (1,\,-2)(-1,\,2)&\mbox{if }i = 1;\\
    (i-1,\,i)(-(i-1),\,-i)&\mbox{if }i\geq 2,\end{cases}
  \]
  \cite[Proposition 8.2.3]{bjorner2005combinatorics}.
\end{example}

\begin{remark}\label{embedding}
  From Example~\ref{typeaintro} we see that $W(A_{n-1})$ consists of
  all permutations of $\bf{n}$.  Then we have a canonical inclusion
  \[
  \iota:W(A_{n-1})\to W(D_n)
  \]
  that sends a permutation in $W(A_{n-1})$ to the same permutation in
  $W(D_n)$. From Examples~\ref{typeaintro} and~\ref{typedintro} we see
  that $\iota(s_i)=s_{i+1}$.
\end{remark}

\begin{definition}
  Let $(W,S)$ be a Coxeter system.  Any element $w\in W$ can be
  written as a product of generators $w=s_1s_2\cdots s_r$, $s_i\in
  S$. 
  \begin{enumerate}
  \item If $r$ is minimal for all expressions of $w$, we call $r$ the
    \textbf{length} of $w$, denoted $\ell(w)$.
  \item Any expression of $w$ as a product of $\ell(w)$ generators is
    called a \textbf{reduced expression} for $w$.
  \item The set of all $s\in S$ that appear in a reduced expression of
    $w$ is called the \textbf{support} of $w$, denoted
    $\supp(w)$. Note that each element in a Coxeter group can have
    many different reduced expressions.  However, if $s\in S$ appears
    in a particular reduced expression for $w$, it must appear in each
    reduced expression for $w$ as a consequence of~\cite[Theorem
    3.3.1]{bjorner2005combinatorics}.  Thus, to determine $\supp(w)$
    we only need to consider a particular reduced expression for $w$,
    so $\supp(w)$ is well-defined.
  \item Let $v_i \in W$ for $1 \leq i \leq k$. We say that the
    product $v = v_1 v_2 \cdots v_k$ is \textbf{reduced} if $\ell(v) =
    \sum_{i=1}^k \ell(v_i)$.
  \end{enumerate}
\end{definition}

\begin{example}
  Let $W=W(A_4)$, let $w=s_1s_2s_3$, and let $x=s_1s_3s_1$.  Then the
  expression given for $w$ is reduced, and $\ell(w) = 3$.  However, we
  see that $x = s_1s_3s_1 = s_1s_1s_3 = s_3$, so the above expression
  for $x$ is not reduced.
\end{example}

\begin{proposition}
  Let $W$ be a finite Coxeter group.  Then there is a unique element,
  $w_0$, of maximal length in $W$.
\end{proposition}

\begin{proof}
  This is~\cite[Proposition 2.3.1]{bjorner2005combinatorics}.
\end{proof}

\begin{example}
  Let $W=W(A_2)$.  Then $w_0 = s_1s_2s_1 = s_2s_1s_2$.
\end{example}

We now distinguish between elements of a Coxeter groups to which long
braid relations may be applied. The following definition is due to
Stembridge \cite{stembridge1996fully}.

\begin{definition}
  Let $(W,S)$ be a simply laced Coxeter system.  We call an element
  $w\in W$ \textbf{complex} if there exist $w_1,w_2\in W$ and $s,t\in
  S$ with $m(s,t)=3$ such that $w = w_1 \cdot sts\cdot w_2= w_1 \cdot
  tst\cdot w_2$ reduced.  An element $w\in W$ that is not complex is
  called \textbf{fully commutative}. We denote the set of fully
  commutative elements by $W_c$.
\end{definition}

\begin{example}
  Let $W=W(D_6)$, let $x=s_1s_2s_4s_3s_4$ and let
  $y=s_1s_2s_6s_3s_5s_4$. Then $x$ is not fully commutative since we
  can apply a long braid relation to the product $s_4s_3s_4$.
  However, $y$ is fully commutative: we cannot apply any long
  braid relations to $y$ because there are no repeated generators.
\end{example}

\begin{definition}
  Let $(W,S)$ be a Coxeter system and let $w\in W$.  We define the
  \textbf{left descent set}, $\mathcal{L}(w)$, and the \textbf{right
    descent set}, $\mathcal{R}(w)$, as follows:
  \begin{align*}
    \mathcal{L}(w) &= \{s\in S\mid \ell(sw) < \ell(w)\};\\
    \mathcal{R}(w) &= \{s\in S\mid \ell(ws) < \ell(w)\}.
  \end{align*}
  A left or right descent set is \textbf{commutative} if it consists
  of mutually commuting generators.
\end{definition}

\begin{example}
  Let $W=W(D_4)$ and let $w=s_1s_4s_3s_2s_3$.  Then
  $\mathcal{L}(w)=\{s_1,s_2, s_4\}$ is commutative and
  $\mathcal{R}(w)=\{s_2, s_3\}$ is not commutative.
\end{example}

It is known that $s\in \mathcal{L}(w)$ if and only if $w$ has a
reduced expression beginning in $s$ \cite[Corollary
1.4.6]{bjorner2005combinatorics}. Similarly, $s\in \mathcal{R}(w)$ if
and only if $w$ has a reduced expression ending in $s$.

\begin{proposition}\label{descentlong}
  Let $(W,S)$ be a finite Coxeter system with $w\in W$. Then
  $\mathcal{L}(w) = S$ if and only if $w = w_0$. Similarly,
  $\mathcal{R}(w) = S$ if and only if $w = w_0$.
\end{proposition}

\begin{proof}
  This is \cite[Proposition 2.3.1 (ii)]{bjorner2005combinatorics}.
\end{proof}

\begin{definition}
  Let $(W,S)$ be a Coxeter system and let $I\subset S$.  Define
  $W_I$ to be the subgroup of $W$ generated by $I$ and define the set
  \[
  W^I=\{w\in W\mid \ell(ws)>\ell(w)\text{ for all }s\in I\}.
  \]
\end{definition}

\begin{proposition}
  Let $(W,S)$ be a Coxeter system and let $I\subset S$.  Then
  $(W_I,I)$ is a Coxeter system.
\end{proposition}

\begin{proof}
  This is \cite[Theorem 5.12(a)]{humphreys1992reflection}.
\end{proof}

The Coxeter group $W_I$ is called a \textbf{parabolic subgroup} of
$W$, with presentation
\[
\langle I\mid (st)^{m(s,t)} = 1, s,t\in I\rangle.
\]
The set $W^I$ is called the set of \textbf{distinguished coset
  representatives}, and consists of minimal representatives of left
cosets $wW_I$.

\begin{proposition}\label{factor}
  Let $(W,S)$ be a Coxeter system and let $I\subset S$.  For each
  $w\in W$ there exist unique elements $w^I\in W^I$ and $w_I\in W_I$
  such that $w=w^Iw_I$ reduced.
\end{proposition}

\begin{proof}
  This is \cite[Proposition 2.4.4]{bjorner2005combinatorics}.
\end{proof}

\begin{definition}
Let $(W,S)$ be a Coxeter system and let $I\subset S$. Write $w = w^Iw_I$ as in Proposition~\ref{factor}. We call this the \textbf{reduced
    decomposition} of $w$ with respect to $I$.
\end{definition}

\begin{lemma}\label{noncommutative}
  Let $(W,S)$ be a simply laced Coxeter system. Let $w\in W$ suppose
  that $s,t\in S$ are such that $s$ and $t$ do not commute.
  \begin{enumerate}
  \item If $s,t\in\mathcal{R}(w)$ then $w$ can be written $w = w'\cdot sts$
    reduced for some $w'\in W$.
  \item If $s,t\in\mathcal{L}(w)$ then $w$ can be written $w = sts\cdot w'$
    reduced for some $w'\in W$.
  \end{enumerate}
\end{lemma}

\begin{proof}
  Let $I = \{s,t\}$ and suppose that $s,t\in\mathcal{R}(w)$. Then by
  Proposition~\ref{factor} we can find a reduced decomposition
  $ws=x^Ix_I$ reduced such that $x^I\in W^I$ and $x_I\in W_I$.  Then
  $w = x^Ix_Is$ reduced, and since $x_I s \in W_I$ this is the unique
  reduced decomposition of $w$ according to Proposition~\ref{factor}.

  Write a reduced decomposition $w = w^I w_I$. Since
  $s,t\in\mathcal{R}(w)$ we can use the above argument to show $s,t\in
  R(w_I)$. Then by Proposition~\ref{descentlong} we have $w_I = sts$,
  so $w = w^I\cdot sts$ reduced.

  Now suppose $s,t\in\mathcal{L}(w)$. Then we can repeat the above
  argument with $w^{-1}$ and the result follows.
\end{proof}

\begin{corollary}\label{fccomm}
  Let $(W,S)$ be a simply laced Coxeter system. If $w\in W_c$ then $\mathcal{R}(w)$ and $\mathcal{L}(w)$ are both
  commutative.
\end{corollary}

\begin{proof}
  Let $w\in W$ be such that $\mathcal{R}(w)$ is not commutative. Then
  by Lemma~\ref{noncommutative} we can write $w = w'\cdot sts$ reduced
  for some $w'\in W$ and noncommuting generators $s,t\in S$, so
  $w\not\in W_c$.
\end{proof}

We now examine a partial ordering on a Coxeter group called the
\textbf{Bruhat order}.

\begin{definition}\label{bruhatdef}
  Let $(W,S)$ be a Coxeter system and let $w\in W$.  Pick a reduced
  expression $w=s_1s_2\cdots s_r$.  Then $x\leq w$ if and only if $x$
  is a subword of $w$; that is, $x=s_{i_1}s_{i_2}\cdots s_{i_k}$ where
  $1\leq i_1 < i_2 < \cdots < i_k \leq r$.
\end{definition}

Since the definition allows for the choice of any reduced expression for $w$, it
is not immediately clear that the Bruhat order is well defined.
However, the Bruhat order is indeed well defined; see~\cite[Theorem
2.2.2]{bjorner2005combinatorics} for a proof.

\begin{example}
  Let $W=W(D_4)$, let $w = s_1s_2s_3s_2$, let $x = s_1s_2$ and let $y
  = s_3s_4$. It is easy to see that the given expressions for $w$,
  $x$, and $y$ are reduced.  Then we have $x\leq w$ since $x$ is a subword of
  $w$, but $y\not\leq w$ since $s_4\in\supp(y)\setminus\supp(w)$, so
  no reduced expression for $y$ is a subword of $w$.
\end{example}

We now introduce star operations, which were developed by Kazhdan and
Lusztig in \cite{kazhdan1979representations}.

\begin{definition}\label{stardef}
  Let $(W,S)$ be a simply laced Coxeter system.  Let $s,t\in S$ such
  that $m(s,t) = 3$.  Define
  \begin{align*}
    D_\mathcal{L}(s,t) &= \left\{w\in W \;\big|\; |\mathcal{L}(w)\cap\{s,t\}| = 1\right\};\\
    D_\mathcal{R}(s,t) &= \left\{w\in W \;\big|\;
      |\mathcal{R}(w)\cap\{s,t\}| = 1\right\}.
  \end{align*}
  If $w\in D_\mathcal{L}(s,t)$, then exactly one of $sw$, $tw$, is in
  $D_\mathcal{L}(s,t)$.  We call the resulting element $\expl{w}{*}$,
  and define the \textbf{left star operation} with respect to
  $\{s,t\}$ to be the map ${}^*:W\rightarrow W:w\mapsto\expl{w}{*}$.
  We can define $w^*$ in a similar way using $D_\mathcal{R}(s,t)$,
  resulting in a \textbf{right star operation} with respect to
  $\{s,t\}$.  Note that each of these maps is an involution and
  partially defined.
\end{definition}

\begin{example}
  Let $x=s_3s_4s_5s_6s_5$ and let $w = s_4$. Note that the given
  expression for $x$ is reduced. Then $x,w\in D_\mathcal{L}(s_3,s_4)$
  so we can apply the operations $\expl{x}{*} = s_4s_5s_6s_5$ and
  $\expl{w}{*} = s_3s_4$. However, $x,w\not\in
  D_\mathcal{R}(s_5,s_6)$, so the right star operation with respect to
  $\{s_5,s_6\}$ is not defined on either $x$ or $w$.
\end{example}

\begin{definition}\label{starreddef}
If $w,y\in D_\mathcal{L}(s,t)$ are such that $\expl{w}{*} = y$ and
$\ell(w) - 1 = \ell(y)$ then we say that $w$ is \textbf{left star
  reducible} to $y$. If $w,y\in D_\mathcal{R}(s,t)$ are such that $w^* =
y$ and $\ell(w) - 1 = \ell(y)$ then we say that $w$ is \textbf{right
  star reducible} to $y$. If there is a sequence
\[
w = w_{(0)}, w_{(1)}, \dots, w_{(k-1)}, w_{(k)} = y
\]
such that $w_{(i)}$ is left or right star reducible to $w_{(i+1)}$
then we say that $w$ is \textbf{star reducible} to $y$.
\end{definition}

\begin{proposition}\label{starredcomm}
  Let $W$ be a Coxeter group of type $A$ or $D$ and let $w\in
  W_c$. Then $w$ is star reducible to a product of commuting
  generators.
\end{proposition}

\begin{proof}
  This is \cite[Theorem 6.3]{green2006star}.
\end{proof}

\begin{proposition}\label{starfc}
  Let $(W,S)$ be a simply laced Coxeter system. If $x\in W_c$ and
  $s,t\in S$ are such that $x\in D_\mathcal{L}(s,t)$, then
  $\expl{w}{*}\in W_c$.
\end{proposition}

\begin{proof}
  This is \cite[Proposition 2.10]{shi2005fully}.
\end{proof}

\section{Kazhdan--Lusztig theory}

We will now use a Coxeter system $(W,S)$ to construct an algebra
$\mathcal{H} = \mathcal{H}(W,S)$ called the \textbf{Hecke algebra}.
The following definitions can be found in
\cite{kazhdan1979representations}.

\begin{definition}\label{heckedef}
  Let $\mathcal{A}=\mathbb{Z}[q^{\frac{1}{2}},q^{-\frac{1}{2}}]$ be
  the ring of Laurent polynomials over $\mathbb{Z}$. Let
  $\mathcal{A}^{+} = \mathbb{Z}[q^{\frac{1}{2}}]$.  Then
  $\mathcal{H}(W,S)$ is the algebra over $\mathcal{A}$ with linear
  basis $\{T_w\mid w\in W\}$ with multiplication determined by the
  following relations:
  \begin{enumerate}
    \item $T_sT_w = T_{sw}$ if $\ell(sw) > \ell(w)$;
    \item $T_s^2 = (q-1)T_s + qT_1$.
  \end{enumerate}
\end{definition}

Using the above relations, we can compute that
\[
T_s^{-1} = q^{-1}T_s - (1-q^{-1})T_1.
\]
Let $w\in W$ have reduced expression $w=s_1s_2\cdots s_r$.  Then using
the first multiplication rule above we see that $T_w =
T_{s_1}T_{s_2}\cdots T_{s_r}$, so $T_w^{-1} = T_{s_r}^{-1}\cdots
T_{s_1}^{-1}$, thus each $T_w$ is invertible in $\mathcal{H}$.

We now define a ring homomorphism $\iota:\mathcal{H}\rightarrow
\mathcal{H}$ by $\iota\left(q^\frac{1}{2}\right)=q^{-\frac{1}{2}}$ and
$\iota(T_w)=T^{-1}_{w^{-1}}$. Note that $\iota$ is an involution.
This involution gives rise to an interesting basis for $\mathcal{H}$.

\begin{proposition}\label{cbasis}
  For each $w\in W$ we have a unique element $C_w\in\mathcal{H}$ with
  the following properties:
  \begin{enumerate}
  \item $\iota(C_w) = C_w$,
  \item $C_w = \displaystyle\sum_{x\leq w} (-1)^{\ell(w) + \ell(x)}
    q^{\frac{1}{2}(\ell(w)-\ell(x))} \iota\left(P_{x,w}\right) T_x$,
    where $P_{w,w} = 1$ and $P_{x,w}(q)\in \mathbb{Z}[q]$ has degree
    $\leq\frac{1}{2}(\ell(w)-\ell(x)-1)$ if $x < w$.
  \end{enumerate}
\end{proposition}

\begin{proof}
  This is \cite[Theorem 1.1]{kazhdan1979representations}.
\end{proof}

These $C_w$ form a basis for $\mathcal{H}$
\cite[6.1]{bjorner2005combinatorics} called the \textbf{Kazhdan--Lusztig basis}.

The polynomials $P_{x,w}$ in Proposition~\ref{cbasis} are called
\textbf{Kazhdan--Lusztig polynomials}, and are particularly difficult
to calculate. For example, the degree of a particular polynomial
$P_{x,w}$ is not even known in general.  We do, however, have an upper
bound for the degree of Kazhdan--Lusztig polynomials when $x < w$:
\[
\deg(P_{x,w})\leq \frac{1}{2}(\ell(w)-\ell(x)-1)
\]
\cite[(1.1.c)]{kazhdan1979representations}.

\begin{definition}\label{muvaldef}
  We denote by $\mu(x,w)$ the coefficient of the term of degree
  $\frac{1}{2}(\ell(w)-\ell(x)-1)$ in $P_{x,w}$. Note that if
  $\ell(w)\equiv \ell(x)\bmod 2$ then $\mu(x,w) = 0$ since
  $\frac{1}{2}(\ell(w)-\ell(x)-1)$ is not an integer; in particular,
  $\mu(w,w) = 0$ since $\ell(w)-\ell(w) = 0$ is even. If $\mu(x,w)\neq
  0$ we write $x\prec w$.
\end{definition}

As we will soon see, calculating $\mu$-values for Kazhdan--Lusztig
polynomials is very difficult, even for Coxeter groups of small rank.
In \cite{kazhdan1979representations} Kazhdan and Lusztig proved the
following helpful elementary properties of $\mu$-values.

\begin{proposition}\label{muproperties}
  \mbox{}
  \begin{enumerate}
  \item If $x,w\in D_\mathcal{L}(s,t)$ then $\mu(x,w) =
    \mu(\expl{x}{*},\expl{w}{*})$;
  \item If $x,w\in D_\mathcal{R}(s,t)$ then $\mu(x,w) = \mu(x^*,w^*)$;
  \item If there exists $s\in\mathcal{L}(w)\setminus\mathcal{L}(x)$ then either
    \begin{enumerate}
    \item $\mu(x,w) = 0$, or
    \item $x=sw$ and $\mu(x,w) = 1$;
    \end{enumerate}
  \item If there exists $s\in\mathcal{R}(w)\setminus\mathcal{R}(x)$ then either
    \begin{enumerate}
    \item $\mu(x,w) = 0$, or
    \item $x=ws$ and $\mu(x,w) = 1$.
    \end{enumerate}
  \end{enumerate}
\end{proposition}

\begin{proof}
  Parts (1) and (2) are \cite[Theorem
  4.2]{kazhdan1979representations}. Parts (3) and (4) are
  \cite[(2.3.e)]{kazhdan1979representations} and
  \cite[(2.3.f)]{kazhdan1979representations}, respectively.
\end{proof}

\begin{corollary}\label{mucomm}
  Let $x\in W_c$ and let $w$ be a product of mutually commuting
  generators. Then $\mu(x,w)\in\{0,1\}$.
\end{corollary}

\begin{proof}
  If $x\not < w$ then $\mu(x,w) = 0$ so we are done. If $x < w$ then
  $\mathcal{L}(x) \subsetneq \mathcal{L}(w)$, so there exists some
  $s\in\mathcal{L}(w) \setminus\mathcal{L}(x)$ and we are done by
  Proposition~\ref{muproperties} part (3).
\end{proof}

Understanding $\mu$-values helps us to calculate Kazhdan--Lusztig
polynomials in general due to the following recurrence relation.

\begin{proposition}\label{klpolrecursive}
  Let $(W,S)$ be a Coxeter system.  Let $x,w\in W$ and let $s\in S$
  be such that $s\in\mathcal{L}(w)$.  Then
  \[
  P_{x,w}=q^{1-c}P_{sx,sw}+q^cP_{x,sw}-\sum_{\substack{sz < z\\
      z\prec sw }}\mu(z,sw)q^{\frac{1}{2}(\ell(w)-\ell(z))}P_{x,z},
  \]
  where
  \[
  c =
    \begin{cases}
      1, &\text{ if }s\in\mathcal{L}(x); \\
      0, &\text{else}.
    \end{cases}
  \]
\end{proposition}

\begin{proof}
  This is \cite[(2.2.c)]{kazhdan1979representations}.
\end{proof}

Unfortunately, this is the only obvious way to compute
Kazhdan--Lusztig polynomials. If $W=W(D_n)$ we have $|W| = 2^{n-1}
n!$, which is very large even in groups of moderate rank. In order to
compute $P_{x,w}$ using the above recurrence relation, we must compute
intervals of the Bruhat order in $W$. In large groups such a
computation is very expensive in terms of either processor time or
memory, depending on the algorithm used.  As a result, using the above
recurrence relation to compute Kazhdan--Lusztig polynomials is
computationally infeasible in all but groups of small rank.

We can use Proposition~\ref{klpolrecursive} to deduce several simple
facts about Kazhdan--Lusztig polynomials.

\begin{proposition}\label{klpolreduce}\label{klpolsxsw}
  Let $x,w\in W$ be such that $x<w$ and let $s\in S$.
  \begin{enumerate}
  \item If $sw<w$ and $sx>x$ then $P_{x,w} = P_{sx,w}$.
  \item If $w < ws$ and $xs\not\leq w$ (thus $xs > x$) then
    $P_{x,w}=P_{xs,ws}$.
  \end{enumerate}
\end{proposition}

\begin{proof}
  Part (1) is \cite[(2.3.g)]{kazhdan1979representations}. Part (2) is
  \cite[Lemma 1.4.5(v)]{shi1986cells}.
\end{proof}

\begin{lemma}\label{wnu}
  Let $x,w\in W$ and let $u=s_1\cdots s_r$ be a product of mutually
  commuting generators such that $s_i\not\in\supp(x)\,\cup\,\supp(w)$ for
  each $1\leq i\leq r$.  Then $P_{x,w}=P_{xu, wu}$, and
  $\mu(x,w)=\mu(xu,wu)$.
\end{lemma}

\begin{proof}
  We will show $P_{x,w}=P_{xu, wu}$ by induction on $r=\ell(u)$.  If
  $r=1$ then we are done by Proposition~\ref{klpolsxsw} (2).  Suppose
  that the statement holds for all values less than $r$. Then by
  induction, we have $P_{x,w} = P_{xus_r,wus_r}$.  Since $u$ is a product of
  mutually commuting generators,
  $s_r\not\in\supp(xus_r)\cup\supp(wus_r)$, so by
  Proposition~\ref{klpolsxsw} (1) we have $P_{x,w} = P_{xus_r,wus_r} =
  P_{xu,wu}$.

  Then since
  \[
  \frac{1}{2}\left(\ell(wu)-\ell(xu)-1\right) =
  \frac{1}{2}\left(\ell(w) + r -\ell(x) - r-1\right) =
  \frac{1}{2}\left(\ell(w)-\ell(x)-1\right)
  \]
  we have $\mu(x,w)=\mu(xu,wu)$.
\end{proof}

Surprisingly little is known about $\mu$-values, even for finite
Coxeter groups. Previously, computer computations suggested that
$\mu(x,w)\in \{0,1\}$ in Coxeter systems of type $A$.  This was shown
to be egregiously false by McLarnan and Warrington
in~\cite{mclarnan2003counterexamples} using computer calculations.
Billey and Warrington have developed a more efficient recursive way to
compute Kazhdan--Lusztig polynomials in certain cases in type $A$
\cite[Lemma 39]{billey2003maximal}.

The group $W(\widetilde{A}_n)$ is an infinite Coxeter group which,
like $W(D_n)$, contains the symmetric group as a parabolic
subgroup. Since the 0-1 conjecture fails in type $A$, it must therefore
also fail in types $\widetilde{A}$ and $D$. However,
in~\cite{green2009leading} Green showed that $\mu(x,w)\in\{0,1\}$ for
$x,w\in W(\widetilde A_n)$ as long as $x$ is fully
commutative. This proof relies on the fact that Coxeter groups of type
$\widetilde A$ do not contain certain elements called ``bad
elements,'' which will be discussed in Chapter~\ref{badchapter}. Green
remarks that there may be many other types of Coxeter groups for which
$\mu(x,w)\in\{0,1\}$ if $x$ is fully
commutative~\cite[Introduction]{green2009leading}.  We will prove this
result in Theorem~\ref{mainresult} for Coxeter groups of type $D$.

We can partition a Coxeter group into sets called Kazhdan--Lusztig
cells, first defined in \cite{kazhdan1979representations}, which behave
nicely with regard to calculations involving $\mu$-values.

\begin{definition}
  Recall from Definition~\ref{muvaldef} that we write $x\prec w$ if
  $\mu(x,w)\neq 0$. Define $x\leq_L w$ if there is a (possibly trivial) chain
  \[
  x=x_0,x_1,\dots,x_r=w
  \]
  such that either $x_i\prec x_{i+1}$ or $x_{i+1}\prec x_i$ and
  $\mathcal{L}(x_i) \not\subset \mathcal{L}(x_{i+1})$. Define $x\sim_L
  w$ if and only if $x\leq_L w$ and $w\leq_L x$.  Then $\sim_L$ is an
  equivalence relation that partitions $W$ into \textbf{left
    Kazhdan--Lusztig cells}, or left cells.  There is an analogous
  definition for \textbf{right Kazhdan--Lusztig cells}.
\end{definition}

\begin{definition}
  We define $x\leq_{LR} w$ if there is a (possibly trivial) chain
  \[
  x=x_0,x_1,\dots,x_r=w
  \]
  such that either $x_i\leq_L x_{i+1}$ or $x_i\leq_R x_{i+1}$ for each
  $i<r$.  Define $x\sim_{LR} w$ if and only if $x\leq_{LR} w$ and
  $w\leq_{LR} x$. As above, $\sim_{LR}$ is an equivalence
  relation that partitions $W$ into \textbf{two-sided
    Kazhdan--Lusztig cells}, or two-sided cells.
\end{definition}

\section{Lusztig's \texorpdfstring{$a$}{a}-function}\label{afunction}

In \cite{lusztig1985cells}, Lusztig defined a function that behaves
nicely with respect to Kazhdan--Lusztig cells.  As we will see, this
function will help us to bound the degree of certain Kazhdan--Lusztig
polynomials.  We begin with a series of definitions and lemmas from
\cite{lusztig1985cells} leading to the definition of Lusztig's
$a$-function.  Although the $a$-function may be defined for affine
Coxeter groups, we will simplify our calculations by assuming that $W$
is a finite Coxeter group.

\begin{lemma}\label{qpol}
  We may define polynomials $Q_{x,w}$ for each $x\leq w$ using the
  following identity:
  \[
  \sum_{x\leq z\leq w} (-1)^{\ell(z)-\ell(x)} Q_{x,z}(q) P_{z,w}(q) =
  \begin{cases}
    1 &\text{ if } x = w;\\
    0 &\text{ if } x < w.
  \end{cases}
  \]
  Then $Q_{x,w}$ is a polynomial of degree $\leq
  \frac{1}{2}\left(\ell(x) - \ell(x) - 1\right)$ if $y < w$ and
  $Q_{w,w}=1$.
\end{lemma}

\begin{proof}
  This is \cite[(1.3.1)]{lusztig1985cells}. 
\end{proof}

The polynomials $Q_{x,w}$ are sometimes called \textbf{inverse
  Kazhdan--Lusztig polynomials}. Like Kazhdan--Lusztig polynomials,
they are difficult to compute.

\begin{definition}
  For $w\in W$ set $\widetilde{T}_w = q^{-\ell(w)/2} T_w$. We
  define
  \[
  D_x = \sum_{x\leq w} Q_{x,w}\left(q^{-1}\right)
  q^{\frac{1}{2}(\ell(w) - \ell(x))} \widetilde{T}_w.
  \]
\end{definition}

\begin{definition}
  Recall that
  $\mathcal{A}=\mathbb{Z}[q^{\frac{1}{2}},q^{-\frac{1}{2}}]$ is the
  ring of Laurent polynomials over $\mathbb{Z}$. Let
  $\tau:\mathcal{H}\rightarrow\mathcal{A}$ be the $\mathcal{A}$-linear
  map defined by $\tau\left(\sum_w \alpha_w \widetilde{T}_w\right) =
  \alpha_e$. (The map $\tau$ turns out to be a trace map.)
\end{definition}

\begin{definition}\label{afuntion}
  Let $w\in W$ and define the set
  \[
  \mathscr{S}_w = \left\{i\in\mathbb{N} \; \middle| \; q^{\frac{i}{2}}
    \tau\left(\widetilde{T}_x \widetilde{T}_y D_w\right)\in
    \mathcal{A}^{+} \text{ for all } x,y\in W \right\}.
  \]
  If $\mathscr{S}_w$ is nonempty we denote $a(w) =
  \min\left(\mathscr{S}_w\right)$, otherwise set $a(w)=\infty$.  Then
  we have a function
  \[
  a:W\rightarrow \mathbb{N}\cup\{\infty\}.
  \]
\end{definition}

We now observe some properties of $a$.

\begin{proposition}\label{abound}
  We have
  \[
  \deg(P_{e,w}) \leq \frac{1}{2}(\ell(w)-a(w)).
  \]
\end{proposition}

\begin{proof}
  This is \cite[1.3(a)]{lusztig1987cells}.
\end{proof}

It will later be very useful to compute $a(w)$ to find a bound for
$\deg(P_{e,w})$. However, from the definition, we can see that
calculating $a(w)$ can be very difficult. To make calculations easier,
we can use the fact that the $a$-function is known to be constant on
Kazhdan--Lusztig cells when $W$ is a Weyl group. Note that Coxeter
groups of types $A$ and $D$ are Weyl groups.

\begin{lemma}\label{avaluecells}
  Let $W$ be a finite Weyl group and let $x,w\in W$ be such that $x
  \sim_{LR} w$.  Then $a(x)=a(w)$.
\end{lemma}

\begin{proof}
  This is \cite[Theorem 5.4]{lusztig1985cells}.
\end{proof}

In \cite{lusztig2003hecke} Lusztig introduces the $a$-function for
Hecke algebras with unequal parameters. He develops a series of
conjectures about how the $a$-function relates to the Coxeter group
$W$, the structure of the Hecke algebra, and Kazhdan--Lusztig cells.
These conjectures are not known to hold for general Coxeter groups in
the unequal parameter case. However, these conjectures are known to
hold for finite Coxeter groups in the equal parameter case.

\begin{lemma}\label{alongest}
  Let $W$ be a finite Coxeter group with longest element $w_0$.  Then
  $a(w_0)=\ell(w_0)$.
\end{lemma}

\begin{proof}
  This is \cite[Proposition 13.8]{lusztig2003hecke}.
\end{proof}

\begin{lemma}\label{aparabolic}
  Let $(W,S)$ be a Coxeter system of type $D_n$ and let $I\subset S$.
  If $w\in W_I$ then $a(w)$ calculated in terms of $W_I$ is equal to
  $a(w)$ calculated in terms of $W$.
\end{lemma}

\begin{proof}
  This is \cite[Conjecture 14.2 P12]{lusztig2003hecke}. In
  \cite[15.1]{lusztig2003hecke} Lusztig proves that Conjecture
  14.2 P12 holds in our case.
\end{proof}

\begin{lemma}\label{aadditive}
  Let $(W,S)$ be a Coxeter system with Coxeter graph $\Gamma$ that can
  be decomposed into a disjoint union of connected components $\Gamma
  = \Gamma_1 \cup \Gamma_2$. Then every $w\in W$ has a unique
  expression $w = w_1w_2$ reduced where $w_1\in W(\Gamma_1)$ and
  $w_2\in W(\Gamma_2)$. Furthermore $a(w) = a(w_1) + a(w_2)$.
\end{lemma}

\begin{proof}
  This is \cite[Lemma 1.8 (1)]{shi2002coxeter}.
\end{proof}

\chapter{Bad elements}\label{badchapter}

There are elements called bad elements whose reduced expressions have
certain unfavorable properties which complicate computing $\mu(x,w)$ where
$x$ is fully commutative and $w$ is bad.  As we will see in
Section~\ref{abadelements}, there are no bad elements in Coxeter
groups of type $A$.  This fact was used in \cite{green2009leading} by
Green to show that $\mu(x,w)\in\{0,1\}$ when $x$ is fully commutative.
However in Section~\ref{dbadelements} we will see that Coxeter groups
of type $D$ do contain bad elements.  We conclude by finding a general
form for reduced expressions of bad elements in Section~\ref{badre}.

\section{Type \texorpdfstring{$A$}{A}}\label{abadelements}

\begin{definition}
  Let $W$ be a simply laced Coxeter group and let $w\in W$.  We say
  that $w$ is \textbf{bad} if $w$ is not a product of commuting
  generators and if $w$ has no reduced expressions beginning or ending
  in two noncommuting generators.  We say $w$ is \textbf{weakly bad}
  if $w$ has no reduced expressions beginning or ending in two
  noncommuting generators.
\end{definition}

\begin{example}
  Let $W = W(D_4)$ and consider the elements $x = s_1s_2s_3$, $y =
  s_1s_2s_4$, and $w = s_1s_2s_4s_3s_1s_2s_4$. We see that $x$ is not
  bad since it has a reduced expression ending in $s_2s_3$. Since $y$
  is a product of mutually commuting generators we see that $y$ is
  weakly bad, but not bad. However, if we compute all reduced
  expressions for $w$ we see that none of them begin or end in two
  noncommuting generators. We can easily see that $w$ is not a product
  of commuting generators, so $w$ is bad.
\end{example}

\begin{lemma}\label{badinverse}
  If $w$ is bad then so is $w^{-1}$.
\end{lemma}

\begin{proof}
  This is an immediate consequence of the symmetry of the definition.
\end{proof}

Recall from Example~\ref{typeaintro} that $W(A_n) \cong S_{n+1}$. For
each element $w\in W(A_n)$ we may use one-line notation to represent
$w$:
\[
w = \left(w(1), w(2), w(3), \dots ,w(n),w(n+1)\right).
\]
We can use this one-line notation to help find left and right descent sets.

\begin{proposition}\label{adescent}
  Let $w\in W(A_n)$.  Then
  \[
  \mathcal{R}(w) = \left\{s_i\in S\mid w(i)>w(i+1)\right\}
  \]
  is the right descent set of $w$ and
  \[
  \mathcal{L}(w) = \left\{s_i\in S\mid w^{-1}(i)>w^{-1}(i+1)\right\}
  \]
  is the left descent set of $w$.
\end{proposition}

\begin{proof}
  This is proven in \cite[Proposition 1.5.3]{bjorner2005combinatorics}.
\end{proof}

We will now use this correspondence between descent sets and one-line
notation to find and classify bad elements in terms of pattern
avoidance.

\begin{definition}
  Let $w\in W(A_n)$ and let $a$, $b$, and $c$ be positive integers.
  We say that $w$ \textbf{has the consecutive pattern} $abc$ if there
  is some $i\in {\bf n-1}$ such that $(w(i),w(i+1),w(i+2))$ is in the
  same relative order as $(a,b,c)$.  If $w$ does not have the
  consecutive pattern $abc$ then we say that $w$ \textbf{avoids the
    consecutive pattern} $abc$.
\end{definition}

\begin{example}
  Let $w\in W(A_5)$ have the one-line notation
  \[
  w = (5, 3, 2, 1, 6, 4).
  \]
  Then $w$ has the consecutive pattern 321 since $(w(1), w(2), w(3)) =
  (5,3,2)$ are in the same relative order as $(3,2,1)$. However, $w$
  avoids the consecutive pattern 123 since there is no $i$ such that
  $(w(i),w(i+1),w(i+2))$ is in the same relative order as $(1,2,3)$.
\end{example}

\begin{lemma}\label{arightavoid}\label{aleftavoid}
  Let $(W,S)$ be a Coxeter system of type $A_n$ and let $w\in W$. Then
  \begin{enumerate}
  \item $w$ has a reduced expression ending in two noncommuting
    generators if and only if $w$ has at least one of the consecutive
    patterns $321$, $231$, or $312$, and
  \item $w\in W$ has a reduced expression beginning in two
    noncommuting generators if and only if $w^{-1}$ has at least one
    of the consecutive patterns $321$, $231$, or $312$.
  \end{enumerate}

\end{lemma}

\begin{proof}
  Let $I=\{s_i,s_{i+1}\}$ and write $w=w^Iw_I$ as in
  Proposition~\ref{factor}. We first observe that if $w$ has a reduced
  expression ending in two noncommuting generators, $s_i,s_{i+1}$, in
  some order, then we have $w_I\in \{s_is_{i+1}s_i, s_is_{i+1},
  s_{i+1}s_i\}$.

  Suppose that $w$ has the consecutive pattern $321$.  Then there is
  some $i$ such that $w(i)>w(i+1)>w(i+2)$, so by
  Proposition~\ref{adescent} we have $s_i,s_{i+1}\in\mathcal{R}(w)$,
  thus $w$ has a reduced expression ending in $s_is_{i+1}s_i$ by
  Lemma~\ref{noncommutative}. Conversely, suppose that
  $w_I=s_is_{i+1}s_i$.  Then $s_i,s_{i+1}\in\mathcal{R}(w)$, so $w(i)
  > w(i+1) > w(i+2)$ by Proposition~\ref{adescent}, thus $(w(i),
  w(i+1), w(i+2))$ has the consecutive pattern 321.

  Next suppose that $w$ has the consecutive pattern $231$.  Then there
  is some $i$ such that $w(i+1)>w(i)>w(i+2)$, so $s_{i+1}\in
  \mathcal{R}(w)$ by Proposition~\ref{adescent}.  If we multiply on
  the right by $s_{i+1}$ then we get
  $ws_{i+1}(i+1)=w(i+2)<w(i)=ws_{i+1}(i)$, so
  $s_{i}\in\mathcal{R}(ws_{i+1})$.  Then $w$ has a reduced expression
  ending in $s_is_{i+1}$. Conversely, if $w_I=s_is_{i+1}$ then
  $w(i+2)< w(i+1)$ and $w(i) < w(i+1)$.  Furthermore, since $s_{i}\in
  \mathcal{R}(ws_{i+1})$ we have $w(i+2) = ws_{i+1}(i+1)<
  ws_{i+1}(i)=w(i)$, so $(w(i), w(i+1), w(i+2))$ has the consecutive
  pattern 231.

  Suppose that $w$ has the consecutive pattern $312$.  Then there is
  some $i$ such that $w(i)>w(i+2)>w(i+1)$. Then $s_i\in
  \mathcal{R}(w)$.  If we multiply on the right by $s_i$ then we get
  $ws_i(i+1)=w(i)>w(i+2)=ws_i(i+2)$, so $s_{i+1}\in\mathcal{R}(ws_i)$.
  Then $w$ has a reduced expression ending in $s_{i+1}s_i$. Conversely, if
  $w_I=s_{i+1}s_i$ then $w(i)> w(i+1)$ and $w(i+2) > w(i+1)$. Since
  $s_{i+1}\in \mathcal{R}(ws_i)$, we have $w(i+2) = ws_{i}(i+2) <
  ws_{i}(i+1)=w(i)$, so $(w(i), w(i+1), w(i+2))$ has the consecutive
  pattern 312.

  Finally, we know that $w$ has no reduced expressions beginning in two
  noncommuting generators if and only if $w^{-1}$ has no reduced
  expressions ending in two noncommuting generators, which by the
  above discussion occurs if and only if $w^{-1}$ avoids the
  consecutive patterns $321$, $231$, and $312$.
\end{proof}

The following theorem was originally proven by Green in the more
general case of a Coxeter group of type $\tilde{A}$.  We will state
the result in type $A$.

\begin{proposition}\label{anobad}
  Let $W=W(A_n)$.  Then there are no bad elements in $W$.
\end{proposition}

\begin{proof}
This is a consequence of \cite[Propositions 2.3, 2.4]{green2009leading}
\end{proof}

\begin{corollary}\label{dpositivenotbad}
If $w\in W(D_n)$ is such that all entries in $w$ are positive then $w$ is not bad.
\end{corollary}

\begin{proof}
Let $w\in W(D_n)$ be such that all entries in $w$ are positive. Recall the embedding $\iota$ from Remark~\ref{embedding}. Since all entries in $w$ are positive, we can find an element $w'\in W(A_{n-1})$ such that $\iota(w') = w$.  If $w'$ has a reduced expression
  \[
  w' = s_{i_1} s_{i_2} \cdots s_{i_r}
  \]
  in $W(A_{n-1})$ then $w$ has the reduced expression
  \[
  w = s_{i_1 + 1} s_{i_2 + 1} \cdots s_{i_r + 1}
  \]
  in $W(D_n)$. By Proposition~\ref{anobad} we see that $w'$ is not bad, thus $w$ is not bad.
\end{proof}

\begin{lemma}\label{acommuting}
  Let $w\in W(A_n)$ be such that both $w$ and $w^{-1}$ avoid the
  consecutive patterns $321$, $231$, and $312$. Then $w$ is a product
  of commuting generators.
\end{lemma}

\begin{proof}
  By Lemma~\ref{arightavoid} we know that $w$ has no reduced
  expressions beginning or ending in two noncommuting generators. By
  Proposition~\ref{anobad} we know that $w$ is not bad, so $w$ must be
  a product of commuting generators.
\end{proof}

\section{Type \texorpdfstring{$D$}{D}}\label{dbadelements}

Recall from Example~\ref{typedintro} that we can represent each
element $w\in W(D_n)$ as a member of the signed permutation group. We
write $w\in W$ using one-line notation
\[
w = \left(w(1), w(2), w(3), \dots ,w(n)\right),
\]
where we write a bar underneath a number in place of a negative sign
in order to simplify notation.
\begin{example}
  Let $w=s_2s_3s_1s_2s_4\in W(D_4)$.  Then we write
  \[
  w = (\underline{2},\underline{3},4,1).
  \]
\end{example}

As in type $A$, we can use the one-line notation of an element to find
its length.

\begin{proposition}~\label{dlength}
  Let $w\in W(D_n)$.  Then
  \[
  \ell(w) = |\{(i,j)\in{\bf n}\times{\bf n}\mid i<j, w(i)>w(j)\}|
          + |\{(i,j)\in{\bf n}\times{\bf n}\mid i<j, w(-i)>w(j)\}|.
  \]
\end{proposition}

\begin{proof}
  This is \cite[Proposition 8.2.1]{bjorner2005combinatorics}.
\end{proof}

\begin{proposition}\label{smultiply}
  Let $(W,S)$ be a Coxeter system of type $D_n$, and let $w\in W$ have
  signed permutation
  \[
  w = (w(1), \dots,w(n)).
  \]
  Suppose $s_i\in S$.  If $i\geq 2$ then multiplying $w$ by $s_i$ on
  the right has the effect of interchanging $w(i)$ and $w(i+1)$.
  Multiplying $w$ by $s_i$ on the left has the effect of interchanging
  the entries in $w$ whose absolute values are $i$ and $i+1$.

  If $i=1$ then multiplying $w$ by $s_1$ on the right has the effect
  of interchanging $w(1)$ and $w(2)$ and switching their signs.
  Multiplying $w$ by $s_1$ on the left has the effect of interchanging
  the entries in $w$ whose absolute values are $1$ and $2$ and
  changing their signs.
\end{proposition}

\begin{proof}
  This follows from the discussion in \cite[Sections 8.1 and
  A3.1]{bjorner2005combinatorics}.
\end{proof}

As in type $A$, in type $D$ we can easily find the descent sets
of an element written in one-line notation.

\begin{proposition}\label{ddescent}
  Let $w\in W(D_n)$.  Then
  \[
  \mathcal{R}(w) = \left\{s_i\in S\mid w(i-1)>w(i)\right\}
  \]
  and
  \[
  \mathcal{L}(w) = \left\{s_i\in S\mid w^{-1}(i-1)>w^{-1}(i)\right\},
  \]
  where $w(0)\overset{\rm{def}}{=}-w(2)$.
\end{proposition}

\begin{proof}
  This is \cite[Propositions 8.2.1 and 8.2.2]{bjorner2005combinatorics}.
\end{proof}

The following lemmas will help us to classify the signed permutations
of bad elements in Theorem~\ref{dbad}.  We will first introduce the
notion of \textit{signed pattern avoidance}, which is not known to be
found in other sources, to begin to describe how to find bad elements
in type $D$.

\begin{definition}
  Let $w\in W(D_n)$.  As in type $A$, we say that $w$ \textbf{avoids
    the consecutive pattern} $abc$ if there is no $i\in{\bf n-2}$ such
  that $(w(i),w(i+1),w(i+2))$ is in the same relative order as
  $(a,b,c)$.  We say that $w$ \textbf{avoids the signed consecutive
    pattern} $abc$ if there is no $i\in{\bf n-2}$ such that
  $(|w(i)|,|w(i+1)|,|w(i+2)|)$ is in the same relative order as
  $(|a|,|b|,|c|)$ and such that $\sign(a)=\sign(w(i))$,
  $\sign(b)=\sign(w(i+1))$, and $\sign(c)=\sign(w(i+2))$.
\end{definition}

\begin{example}
  Let $w\in W(D_6)$ have the following one line notation
  \[
  w = (4, \underline{3}, 1, 5, 6, \underline{2}).
  \]
  Then $w$ has the signed consecutive pattern $3\underline{2}1$ since
  $(|w(1)|, |w(2)|, |w(3)|)$ are in the same relative order as
  $(|3|,|-2|,|1|)$ and $\sign(3)=\sign(w(1))$, $\sign(-2)=\sign(w(2))$, and
  $\sign(1)=\sign(w(3))$. However, $w$ avoids the signed consecutive
  pattern $1\underline{2}3$.
\end{example}

\begin{lemma}\label{drightavoidbig}\label{dleftavoidbig}
  Let $s,t\in S$ such that $m(s,t) = 3$ and $s_1\not\in \{s,t\}$. Then 
  \begin{enumerate}
  \item $w$ has a reduced expression ending in $st$ or $ts$ if and
    only if $w$ has at least one of the consecutive patterns 321, 231,
    or 312, and
  \item $w$ has a reduced expression beginning in $st$ or $ts$ if and
    only if $w^{-1}$ has at least one of the consecutive patterns 321,
    231, or 312.
  \end{enumerate}
\end{lemma}

\begin{proof}
  Let $i \geq 2$, let $I=\{s_i,s_{i+1}\}$ and write $w=w^Iw_I$ as in
  Proposition~\ref{factor}. We first observe that if $w$ has a reduced
  expression ending in two noncommuting generators, $s_i,s_{i+1}$, in
  some order, then we have $w_I\in \{s_is_{i+1}s_i, s_is_{i+1},
  s_{i+1}s_i\}$.

  Suppose that $w$ has the consecutive pattern $321$.  Then there is
  some $i$ such that $w(i-1)>w(i)>w(i+1)$, so by
  Proposition~\ref{ddescent} we have $s_i,s_{i+1}\in\mathcal{R}(w)$,
  thus $w$ has a reduced expression ending in $s_is_{i+1}s_i$ by
  Lemma~\ref{noncommutative}. Conversely, suppose that
  $w_I=s_is_{i+1}s_i$.  Then $s_i,s_{i+1}\in\mathcal{R}(w)$, so $w(i-1)
  > w(i) > w(i+1)$ by Proposition~\ref{ddescent}, thus $(w(i-1), w(i), w(i+1))$ has the
  consecutive pattern 321.

  Next suppose that $w$ has the consecutive pattern $231$.  Then there is some $i$ such that
  $w(i)>w(i-1)>w(i+1)$, so $s_{i+1}\in \mathcal{R}(w)$ by
  Proposition~\ref{ddescent}.  If we
  multiply on the right by $s_{i+1}$ then we get
  $ws_{i+1}(i)=w(i+1) < w(i - 1)=ws_{i+1}(i - 1)$, so
  $s_{i}\in\mathcal{R}(ws_{i+1})$.  Then $w$ has a reduced expression
  ending in $s_is_{i+1}$. Conversely, if $w_I=s_is_{i+1}$ then
  $w(i+1)< w(i)$ and $w(i-1) < w(i)$.  Furthermore, since $s_{i}\in
  \mathcal{R}(ws_{i+1})$ we have $w(i+1) = ws_{i+1}(i)<
  ws_{i+1}(i-1)=w(i-1)$, so $(w(i-1), w(i), w(i+1))$ has the consecutive
  pattern 231.

  Suppose that $w$ has the consecutive pattern $312$.  Then there is
  some $i$ such that $w(i-1)>w(i+1)>w(i)$, so $s_i\in
  \mathcal{R}(w)$.  If we multiply on the right by $s_i$ then we get
  $ws_i(i) = w(i - 1) > w(i+1) = ws_i(i+1)$, so $s_{i+1}\in\mathcal{R}(ws_i)$.
  Then $w$ has a reduced expression ending in $s_{i+1}s_i$. Conversely, if
  $w_I=s_{i+1}s_i$ then $w(i-1)> w(i)$ and $w(i+1) > w(i)$. Since
  $s_{i+1}\in \mathcal{R}(ws_i)$, we have $w(i+1) = ws_{i}(i+1) <
  ws_{i}(i) = w(i-1)$, so $(w(i - 1), w(i), w(i + 1))$ has the consecutive
  pattern 312.

  Finally, we know that $w$ has no reduced expressions beginning in two
  noncommuting generators $s,t$ with $s_1\not\in\{s,t\}$ if and only if $w^{-1}$ has no reduced
  expressions ending in $st$ or $ts$. By the
  above discussion, this occurs if and only if $w^{-1}$ avoids the
  consecutive patterns $321$, $231$, and $312$.
\end{proof}

\begin{lemma}\label{drightavoidsmall}\label{dleftavoidsmall}
  Let $w\in W(D_n)$. Then
  \begin{enumerate}
  \item $w$ has a reduced expression ending in $s_1s_3$ or
  $s_3s_1$ if and only if $-w(1) > w(3)$, and
  \item $w$ has a reduced expression beginning in $s_1s_3$ or
  $s_3s_1$ if and only if $-w^{-1}(1) > w^{-1}(3)$.
  \end{enumerate}
\end{lemma}

\begin{proof}
  Let $w\in W$ be such that $-w(1) > w(3)$. Then we either have $-w(2)>w(1)$ or $-w(2) \leq
  w(1)$. 

  If $-w(2)>w(1)$ then
  $s_1\in\mathcal{R}(w)$. Multiplying on the right by $s_1$, we see
  that $ws_1(2)=-w(1)>w(3)=ws_1(3)$, so $s_3\in\mathcal{R}(ws_1)$.
  Then $w$ has a reduced expression ending in $s_3s_1$. 

  On the other hand, if $-w(2) \leq
  w(1)$ we must have $w(2) \geq -w(1) > w(3)$, so
  $s_3\in\mathcal{R}(w)$. Multiplying on the right by $s_3$, we see
  that $-ws_3(2)=-w(3)>w(1)=ws_3(1)$, so $s_1\in\mathcal{R}(ws_3)$.
  Then $w$ has a reduced expression ending in $s_1s_3$.

  Conversely, let $w\in W$ be such that $w$ has a reduced expression ending in
  $s_3s_1$ or $s_1 s_3$. If an expression for $w$ ends in $s_3s_1$
  then we have $s_1\in\mathcal{R}(w)$ and $s_3\in\mathcal{R}(ws_1)$,
  so $-w(1)=ws_1(2)>ws_1(3)=w(3)$. If an expression for $w$ ends in
  $s_1s_3$ then we have $s_3\in\mathcal{R}(w)$ and
  $s_1\in\mathcal{R}(ws_3)$, so $-w(1)=-ws_3(1)>ws_3(2)=w(3)$.

  Finally, we know that $w$ has no reduced expressions beginning in
  $s_1s_3$ or $s_3s_1$ if and only if $w^{-1}$ has no reduced
  expressions ending in $s_1s_3$ or $s_3s_1$. By the above discussion,
  this occurs if and only if $-w^{-1}(1) > w^{-1}(3)$.
\end{proof}

\begin{lemma}\label{dcommuting}
  Let $w\in W(D_n)$ be such that each entry in the one-line notation for
  $w$ is positive and both $w$ and $w^{-1}$ avoid the consecutive
  patterns $321$, $231$, and $312$. Then $w$ is a product of commuting
  generators.
\end{lemma}

\begin{proof}
  Recall the embedding $\iota$ from Remark~\ref{embedding}. Since all
  entries in $w$ are positive, we can find an element $w'\in
  W(A_{n-1})$ such that $\iota(w') = w$.  If $w'$ has a reduced
  expression
  \[
  w' = s_{i_1} s_{i_2} \cdots s_{i_r}
  \]
  in $W(A_{n-1})$ then $w$ has the reduced expression
  \[
  w = s_{i_1 + 1} s_{i_2 + 1} \cdots s_{i_r + 1}
  \]
  in $W(D_n)$. By Lemma~\ref{acommuting} the $s_{i_j}$ are mutually
  commuting generators, so the $s_{i_j +1}$ are also mutually
  commuting, thus $w$ is a product of commuting generators in
  $W(D_n)$.
\end{proof}

\begin{lemma}\label{permbound}\label{permposneg}
  Let $w\in W(D_n)$ be weakly bad and let $i\in {\bf n}$.  Then $w$
  satisfies the following conditions:
  \begin{enumerate}
  \item $w(j)>\min(\{w(i-1),w(i)\})$ for all $j>i$;
  \item $w(k)<\max(\{w(i-1),w(i)\})$ for all $k<i - 1$;
  \item if $w(i),w(i+1)>0$ then $w(j)>0$ for all $j\geq i$;
  \item if $w(i),w(i+1)<0$ then $w(j)<0$ for all $j\leq i + 1$.
  \end{enumerate}
\end{lemma}

\begin{proof}
  Suppose that there is some least $j>i$ such that
  $w(j)\leq\min(\{w(i-1),w(i)\})$. Note that since $j > i$ we cannot
  have $w(j) = w(i)$ or $w(j) = w(i-1)$, so
  $w(j)<\min(\{w(i-1),w(i)\})$.  Then $w(j-2) \geq
  \min(\{w(i-1),w(i)\}) > w(j)$ and $w(j-1) \geq \min(\{w(i-1),w(i)\})
  > w(j)$, so we see that $(w(j-2),w(j-1),w(j))$ must have the
  consecutive pattern $321$ or $231$, which is impossible by
  Lemma~\ref{drightavoidbig}, proving (1).

  Suppose that there is some greatest $k < i - 1$ such that
  $w(k)\geq\max(\{w(i-1),w(i)\})$. Note that since $k < i - 1$ we
  cannot have $w(k) = w(i)$ or $w(k) = w(i-1)$, so
  $w(k)>\max(\{w(i-1),w(i)\})$.  Then $w(k+1)\leq
  \max(\{w(i-1),w(i)\}) < w(k)$ and $w(k+2)\leq \max(\{w(i-1),w(i)\})
  < w(k)$, so we see that $(w(k),w(k+1),w(k+2))$ must have the
  consecutive pattern $321$ or $312$, which is impossible by
  Lemma~\ref{drightavoidbig}, proving (2).

  It is easy to see that assertion (1) implies (3) and (2) implies (4).
\end{proof}

\begin{example}
  Let $w\in W(D_7)$ have the one-line notation given below
  \[
  w = (2, \underline{3}, \underline{6}, 1, 4, \underline{5}, 7).
  \]
  Then $(w(1), w(2), w(3)) = (2, -3, -6)$ has the consecutive pattern
  321, and $w$ is not bad by Lemma~\ref{drightavoidbig}. Similarly,
  $(w(4), w(5), w(6)) = (1,4,-5)$ has the consecutive pattern 231, and
  $w$ is not bad by Lemma~\ref{drightavoidbig}.
\end{example}

\begin{lemma}\label{first3}
  Let $w\in W(D_n)$ be a bad element.  Then $(w(1),w(2),w(3))$ has
  one of the following consecutive signed patterns:
  \[
    1\underline{2}3,\, \underline{1}\underline{2}3,\,
    1\underline{3}2,\, \underline{1}\underline{3}2,\,
    2\underline{1}3,\, \underline{2}\underline{1}3.
  \]
  In particular, we have $|w(1)| < w(3)$.
\end{lemma}

\begin{proof}
  There are $2^3\cdot 3! = 48$ possible choices of signed consecutive
  patterns for $(w(1),w(2),w(3))$.
  \[
  \begin{array}{cccccccc}
    123 & \underline{1}23 & \underline{1}\underline{2}3 & 
    \framebox{\underline{1}2\underline{3}} & \doublebox{\underline{1}\underline{2}\underline{3}} & 
    1\underline{2}3 & \framebox{12\underline{3}} & \doublebox{1\underline{2}\underline{3}}\\

    132 & \underline{1}32 & \underline{1}\underline{3}2 & 
    \framebox{\underline{1}3\underline{2}} & \ovalbox{\underline{1}\underline{3}\underline{2}} & 
    1\underline{3}2 & \framebox{13\underline{2}} & \ovalbox{1\underline{3}\underline{2}}\\

    213 & \underline{2}13 & \underline{2}\underline{1}3 & 
    \framebox{\underline{2}1\underline{3}} & \framebox{\underline{2}\underline{1}\underline{3}} & 
    2\underline{1}3 & \doublebox{21\underline{3}} & \doublebox{2\underline{1}\underline{3}}\\

    \framebox{231} & \colorbox{shade}{\underline{2}31} & \colorbox{shade}{\underline{2}\underline{3}1} & 
    \colorbox{shade}{\underline{2}3\underline{1}} & \colorbox{shade}{\underline{2}\underline{3}\underline{1}} & 
    \ovalbox{2\underline{3}1} & \framebox{23\underline{1}} & \ovalbox{2\underline{3}\underline{1}}\\

    \ovalbox{312} & \colorbox{shade}{\underline{3}12} & \colorbox{shade}{\underline{3}\underline{1}2} & 
    \colorbox{shade}{\underline{3}1\underline{2}} & \colorbox{shade}{\underline{3}\underline{1}\underline{2}} & 
    \ovalbox{3\underline{1}2} & \doublebox{31\underline{2}} & \doublebox{3\underline{1}\underline{2}}\\

    \doublebox{321} & \colorbox{shade}{\underline{3}21} & \colorbox{shade}{\underline{3}\underline{2}1} & 
    \colorbox{shade}{\underline{3}2\underline{1}} & \colorbox{shade}{\underline{3}\underline{2}\underline{1}} & 
    \ovalbox{3\underline{2}1} & \doublebox{32\underline{1}} & \ovalbox{3\underline{2}\underline{1}}
  \end{array}
  \]

  We can use Lemma~\ref{drightavoidbig} to eliminate the possibilities
  that have the consecutive patterns \doublebox{321}, \framebox{231},
  or \ovalbox{312}, and Lemma~\ref{drightavoidsmall} to eliminate the
  possibilities in which \colorbox{shade}{$-w(1) > w(3)$}. This leaves
  us with 12 possible choices.
  \begin{align*}
    123 && \underline{1}23 && \underline{1}\underline{2}3 && 1\underline{2}3\\
    132 && \underline{1}32 && \underline{1}\underline{3}2 && 1\underline{3}2\\
    213 && \underline{2}13 && \underline{2}\underline{1}3 && 2\underline{1}3
  \end{align*}
  Then we can use Lemma~\ref{permposneg} to see that if $w(1) < 0$ and
  $w(2),w(3) > 0$, then $w(i) > 0$ for all $i\geq 2$, so $w\not\in
  W(D_n)$ since $w$ has an odd number of negative signs.  Furthermore,
  if $w(1),w(2),w(3) > 0$ then $w(i) > 0$ for each $i$ by
  Lemma~\ref{permposneg}, so $w$ is not bad by
  Corollary~\ref{dpositivenotbad}. This eliminates the first and second column
  of possibilities, leaving us with the desired choices.
\end{proof}

\begin{lemma}\label{alternating}
  Let $w\in W(D_n)$ be bad and define $l$ to be the maximum integer
  $i$ such that $w(i)$ is negative, or $0$ if no such integer
  exists. Then $l$ is even and we can write
  \[
  w = \begin{dcases*}
    (a_1, b_1, a_2, b_2, \dots , a_k , b_k, a_{k+1}, b_{k+1},\dots, a_{n/2},b_{n/2}) & if $n$ is even;\\
    (a_1, b_1, a_2, b_2, \dots , a_k , b_k, a_{k+1}, b_{k+1},\dots, a_{(n-1)/2},b_{(n-1)/2},a_{(n+1)/2}) & if $n$ is odd;\\
  \end{dcases*}
  \]
  where $k = l/2$, $(a_i)$ and $(b_i)$ are increasing
  sequences, $a_i > 0$ for $i \geq 2$, $b_i < 0$ if $i\leq k$, and
  $b_i > 0$ if $i > k$.
\end{lemma}

\begin{proof}
  Let $w\in W(D_n)$ be bad. Then by Lemma~\ref{first3} we see that
  $w(3)$ must be positive, so by
  Lemma~\ref{permposneg} we can write
  \[
  w = \begin{dcases*}
    (a_1, b_1, a_2, b_2, \dots , a_k , b_k, a_{k+1}, b_{k+1},\dots, a_{n/2},b_{n/2}) & if $n$ is even;\\
    (a_1, b_1, a_2, b_2, \dots , a_k , b_k, a_{k+1}, b_{k+1},\dots, a_{(n-1)/2},b_{(n-1)/2},a_{(n+1)/2}) & if $n$ is odd;\\
  \end{dcases*}
  \]
  where $a_i > 0$ for $i \geq 2$, $b_i < 0$ if $i\leq k$, and $b_i >
  0$ if $i >k$. Now we see that $(a_i)$ must be increasing, since if
  $a_{i+1} < a_{i}$, then $(a_i, b_i, a_{i+1})$ would have one of the
  consecutive patterns $321$, $231$, or $312$, contradicting
  Lemma~\ref{drightavoidbig}. Similarly, $(b_i)$ must be increasing,
  since if $b_{i+1} < b_{i}$, then $(b_i, a_{i+1}, b_{i+1})$ would
  have one of the consecutive patterns $321$, $231$, or $312$,
  contradicting Lemma~\ref{drightavoidbig}.
\end{proof}

\begin{lemma}\label{baseeven}
  Let $w\in W_n$ be bad, let $n$ be even, and write $w$ as in
  Lemma~\ref{alternating}. If $n = 2k$ we have
  \[
  w = \left((-1)^{n/2},\underline{n},3,\underline{n-2}, 5,\dots,
    \underline{4},n-1,\underline{2}\right),
  \]
  and thus $w = w^{-1}$.
\end{lemma}

\begin{proof}
  First suppose that $n\equiv 0\bmod 4$. Then $k$ is even, so $a_1 >
  0$, else $w$ would have an odd number of negative entries. By
  Lemma~\ref{badinverse} we know that $w^{-1}$ is also bad, so we can write
  \[
  w^{-1} = (a'_1, b'_1, a'_2, b'_2, \dots , a'_{k'} , b'_{k'},
  a'_{k'+1}, b'_{k'+1},\dots, a'_{n/2},b'_{n/2})
  \]
  as in Lemma~\ref{alternating}. Note that $w$ and $w^{-1}$ each must
  have $k$ negative entries, so either $k = k'$ or $k' = k -1$ and
  $a'_1 < 0$.

  Suppose, towards a contradiction, that $k' = k - 1$. Then 
  \[
  \{b_i\} = \{-1,-2,-4,-6,\dots,-(n-2)\}.
  \]
  These are the only choices for $b_i$ since these are the places of the negative entries in $w^{-1}$. In other words, if $i \in \{-1,-2,-4,-6,\dots,-(n-2)\}$ we know that $w^{-1}(i) < 0$, so the entry with absolute value $i$ in $w$ must be negative.  Since $(b_i)$ is increasing, we have
  $b_k = w(n) = -1$. But then $w^{-1}(1) = -n$, so $n = |w^{-1}(1)| > w^{-1}(3)$,
  which is impossible by Lemma~\ref{first3}. Thus, $k = k'$ and $a'_1
  > 0$.

  This means that $(b_i) = (b'_i) = (-n, -(n-2), \dots, -4, -2)$. Then
  since $(a_i)$ is increasing, we have $(a_i) = (a'_i) = (1,3,5,\dots,
  n-1)$, so $w$ has the desired form.

  Next suppose that $n\equiv 2\bmod 4$. Then $k$ is odd, so $a_1 <
  0$ else $w$ would have an odd number of negative entries. By
  Lemma~\ref{badinverse} we know that $w^{-1}$ is also bad, so we can write
  \[
  w^{-1} = (a'_1, b'_1, a'_2, b'_2, \dots , a'_{k'} , b'_{k'},
  a'_{k'+1}, b'_{k'+1},\dots, a'_{n/2},b'_{n/2})
  \]
  as in Lemma~\ref{alternating}. Note that $w$ and $w^{-1}$ must have
  $k+1$ negative entries. Now we must have $k = k'$, since there is no
  other way for $w^{-1}$ to have $k+1$ negative entries.

  This means that $\{b_i\}\cup \{a_1\} = \{b'_i\}\cup \{a'_1\} =
  \{-1,-2,-4,-6,\dots, -n\}$.  Then since $(a_i)$ is increasing we
  have $(a_i)_{i=2}^k = (a'_i)_{i=2}^k = (3,5,\dots, n-1)$. 

  By Lemma~\ref{first3} we must have $|a_1|=|w(1)| < w(3) = 3$, so
  either $a_1 = -1$ or $a_1 = -2$. If $a_1 = -2$ then we must have
  $w(n) = b_k = -1$ since $(b_i)$ is increasing. But this means that
  $w^{-1}(1) = -n$, so $n = |w^{-1}(1)| > w^{-1}(3)$, contradicting
  Lemma~\ref{first3}. Then we must have $a_1 = -1$, so $a'_1 =
  -1$. Then $(b_i) = (b'_i) = (-n, -(n-2), \dots, -4, -2)$, so $w$ has
  the desired form.

  Once $w$ is in the desired form we can easily see that $w = w^{-1}$.
\end{proof}

\begin{lemma}\label{baseodd}
  Let $w\in W_n$ be bad, let $n$ be odd, and write $w$ as in
  Lemma~\ref{alternating}. If $n - 2k = 1$ we have
  \[
  w = \left((-1)^{(n-1)/2},\underline{n-1},3,\underline{n-3}, 5,\dots,
    \underline{4},n-2,\underline{2}, n\right),
  \]
  and thus $w=w^{-1}$.
\end{lemma}

\begin{proof}
  By Lemma~\ref{badinverse} we know that $w^{-1}$ is also bad, so we
  can write
  \[
  w^{-1} = (a'_1, b'_1, a'_2, b'_2, \dots , a'_{k'} , b'_{k'},
  a'_{k'+1}, b'_{k'+1},\dots, a'_{(n-1)/2},b'_{(n-1)/2}, a'_{(n+1)/2})
  \]
  as in Lemma~\ref{alternating}. We see that $a'_{(n+1)/2} = w^{-1}(n)
  > 0$, so since all $b_i < 0$ there must be some $j$ such that $a_j =
  n$. Then since $(a_i)$ is increasing we must have $a_{(n+1)/2} =
  n$. Then we have
  \[
  w = (a_1, b_1, a_2, b_2, \dots , a_k , b_k, a_{k+1}, b_{k+1},\dots,
  a_{(n-1)/2},b_{(n-1)/2},n),
  \]
  and the result follows by applying Lemma~\ref{baseeven} to
  \[
  (a_1, b_1, a_2, b_2, \dots , a_k , b_k, a_{k+1}, b_{k+1},\dots,
  a_{(n-1)/2},b_{(n-1)/2}).
  \]
\end{proof}

\begin{lemma}\label{npos}
  Let $w\in W(D_n)$ be bad and write $w$ as in
  Lemma~\ref{alternating}. If $n-2k > 1$ then $w^{-1}(n-1)$ and
  $w^{-1}(n)$ are both positive.
\end{lemma}

\begin{proof}
  We see that $w(n-1)$ and $w(n)$ are clearly positive since $n-2k >
  1$. Write $w^{-1}$ as in Lemma~\ref{alternating}, and suppose that
  $b'_i < 0$ for $i \leq k'$. If either $w^{-1}(n-1)$ or $w^{-1}(n)$
  were negative then we would have $n-2k' \leq 1$. Then by
  Lemma~\ref{baseeven} or~\ref{baseodd} we would have $w = w^{-1}$,
  which cannot be true since $k\neq k'$, thus $w^{-1}(n-1)$ and
  $w^{-1}(n)$ are positive.
\end{proof}

\begin{lemma}\label{badform}
  Let $w\in W(D_n)$ be bad and write $w$ as in
  Lemma~\ref{alternating}. Then $a_1 = \pm 1$, $(a_i)_{i=2}^k =
  (3,5,7, \dots, 2k-1)$, and $(b_i)_{i=1}^k = (-2k,-(2k-2), \dots,
  -6, -4,-2)$.
\end{lemma}

\begin{proof}
  We will induct on $n - 2k$. If $n - 2k = 0$ then we are done by
  Lemma~\ref{baseeven} and if $n - 2k = 1$ then we are done by
  Lemma~\ref{baseodd}.

  If $w(n) = n$ then we have 
  \[
  w = \begin{dcases*}
    (a_1, b_1, a_2, b_2, \dots , a_k , b_k, a_{k+1}, b_{k+1},\dots, a_{n/2},n) & if $n$ is even;\\
    (a_1, b_1, a_2, b_2, \dots , a_k , b_k, a_{k+1}, b_{k+1},\dots, a_{(n-1)/2},b_{(n-1)/2},n) & if $n$ is odd,\\
  \end{dcases*}
  \]
  so we can apply the inductive hypothesis to
  \[
  \begin{array}{ll}
    (a_1, b_1, a_2, b_2, \dots , a_k , b_k, a_{k+1}, b_{k+1},\dots, a_{n/2}) & \text{if $n$ is even;}\\
    (a_1, b_1, a_2, b_2, \dots , a_k , b_k, a_{k+1}, b_{k+1},\dots, a_{(n-1)/2},b_{(n-1)/2}) & \text{if $n$ is odd,}\\
  \end{array}
  \]
  and obtain the desired result.

  Suppose $w(n) \neq n$. Recall that $n-2k > 1$, so by Lemma~\ref{npos}, we have $w^{-1}(n) >
  0$, and since $(a_i)$ and $(b_i)$ are increasing we must have
  $w(n-1) = n$.

  If $w(n) \neq n-1$ and $w(n) \neq n$ then by Lemma~\ref{npos} we have $w^{-1}(n-1) > 0$, so we
  can use the fact that $(a_i)$ and $(b_i)$ are increasing to show
  that $w(n-3) = n - 1$. Then $w^{-1}(n) > 0$ and $w^{-1}(n) =
  n-1$, so $w^{-1}$ must have consecutive pattern $312$ or $321$ where
  the 3 is at position $w(n)$, contradicting Lemma~\ref{drightavoidbig}. Thus we
  must have have $w(n) = n-1$, so
  \[
  w = \begin{dcases*}
    (a_1, b_1, a_2, b_2, \dots , a_k , b_k, a_{k+1}, b_{k+1},\dots, b_{n/2-1},n, n-1) & if $n$ is even;\\
    (a_1, b_1, a_2, b_2, \dots , a_k , b_k, a_{k+1}, b_{k+1},\dots, a_{(n-1)/2},n,n-1) & if $n$ is odd,\\
  \end{dcases*}
  \]
  and we can apply the inductive hypothesis to
  \[
  \begin{array}{ll}
    (a_1, b_1, a_2, b_2, \dots , a_k , b_k, a_{k+1}, b_{k+1},\dots, b_{n/2-1}) & \text{if $n$ is even;}\\
    (a_1, b_1, a_2, b_2, \dots , a_k , b_k, a_{k+1}, b_{k+1},\dots, a_{(n-1)/2}) & \text{if $n$ is odd,}\\
  \end{array}
  \]
  and obtain the desired result.
\end{proof}

\begin{theorem}\label{dbad}
  Let $w\in W(D_n)$ be bad and define
  \[
  w_m=
  \begin{dcases*}
    \left((-1)^{m/2},\underline{m},3,\underline{m-2}, 5,\dots,
      \underline{4},m-1,\underline{2}, m+1, m+2, \dots, n\right) & if $m$ is even;\\
    \left((-1)^{(m-1)/2},\underline{m-1},3,\underline{m-3}, 5,\dots ,
      \underline{4},m-2,\underline{2}, m, m+1, \dots, n\right) & if
    $m$ is odd.
  \end{dcases*}
  \]
  Then we must have $w=w_{m}u$ reduced for some
  $m\leq n$, where $u$ is a product of mutually commuting
  generators such that $\supp(u) \subset \{s_{m+2}, s_{m+3}, s_{m+4},
  \dots, s_n\}.$
\end{theorem}

\begin{proof}
  Let $w\in W(D_n)$ be bad. Then by Lemma~\ref{badform} we can write
  \[
  w = \begin{dcases*}
    (a_1, b_1, a_2, b_2, \dots , a_k , b_k, a_{k+1}, b_{k+1},\dots, a_{n/2},b_{n/2}) & if $n$ is even;\\
    (a_1, b_1, a_2, b_2, \dots , a_k , b_k, a_{k+1}, b_{k+1},\dots, a_{(n-1)/2},b_{(n-1)/2},a_{(n+1)/2}) & if $n$ is odd;\\
  \end{dcases*}
  \]
  where $a_1 = \pm 1$, $(a_i)_{i=2}^k = (1,3,5,7, \dots, 2k-1)$, and
  $(b_i)_{i=1}^k = (-2k,-(2k-2), \dots, -6, -4,-2)$. Then for $j > k$
  we have $a_j > 2k$ and $b_j > 2k$, so we can write
  \[
  w = w_{2k}\cdot u
  \]
  where 
  \[
  u = \begin{dcases*}
    (1, 2, 3, 4,\dots, 2k-1, 2k, a_{k+1}, b_{k+1},\dots, a_{n/2},b_{n/2}) & if $n$ is even;\\
    (1, 2, 3, 4,\dots, 2k-1, 2k, a_{k+1}, b_{k+1},\dots, a_{(n-1)/2},b_{(n-1)/2},a_{(n+1)/2}) & if $n$ is odd.\\
  \end{dcases*}
  \]
  It follows that $u$ is a product of commuting generators by
  Lemma~\ref{dcommuting} because the original $(a_i)$ and $(b_i)$ sequences were increasing. We see that $s_1,s_2, \dots,
  s_{2k+1}\not\in\supp(u)$, so $\supp(u) \subset \{s_{2k+2}, s_{2k+3},
  \dots, s_n\}$.
\end{proof}

\begin{corollary}\label{dmap}
  If $n$ is even we have
  \[
  w_n(i) = \begin{dcases*}
    (-1)^{n/2} & if $i = 1$;\\
    i & if $i > 1$ and $i$ is odd;\\
    -(n+2-i) & if $i$ is even.\\
  \end{dcases*}
  \]
  If $n$ is odd we have
  \[
  w_n(i) = \begin{dcases*}
    (-1)^{(n-1)/2} & if $i = 1$;\\
    i & if $i > 1$ and $i$ is odd;\\
    -(n+1-i) & if $i$ is even.\\
  \end{dcases*}
  \]
\end{corollary}

\begin{proof}
  This is immediate from Theorem~\ref{dbad}
\end{proof}

Eventually, we will find a reduced expression for $w_n$.  In order to
ensure that the expression that we find is reduced, we must first find
the length of $w_n$.

\begin{lemma}\label{badlength}
  We have
  \[
  \ell(w_n) =
  \begin{dcases*}
    \frac{3n^2}{8}+\frac{n}{4} & if $n$ is even;\\
    \frac{3(n-1)^2}{8}+\frac{n-1}{4} & if $n$ is odd.
  \end{dcases*}
  \]
\end{lemma}

\begin{proof}
  Let $W=W(D_n)$ and consider $w_n\in W$.  Suppose until further
  notice that $n$ is even. By Theorem~\ref{dbad} we can write
  \[
  w_n=(a_1,b_1,a_2,b_2,\dots,a_n,b_n),
  \]
  where
  \[
  \begin{array}{cccc}
    a_1=\pm 1, & a_i = 2i-1\,(i\geq 2), &\text{and}& b_i = -(n+2-2i). 
  \end{array}
  \]
  Using Proposition~\ref{dlength} we see that
  \begin{align*}
    \ell(w_n) & = \sum_{i=1}^{\frac{n}{2}}\left|\{j\mid i < j, a_i>a_j\}\right|
    + \sum_{i=1}^{\frac{n}{2}}\left|\{j\mid i < j, -a_i>a_j\}\right|\\
    & + \sum_{i=1}^{\frac{n}{2}}\left|\{j\mid i < j, b_i>b_j\}\right|
    + \sum_{i=1}^{\frac{n}{2}}\left|\{j\mid i < j, b_i>a_j\}\right|\\
    & + \sum_{i=1}^{\frac{n}{2}}\left|\{j\mid i\leq j, a_i>b_j\}\right|
    + \sum_{i=1}^{\frac{n}{2}}\left|\{j\mid i\leq j, -a_i>b_j\}\right|\\
    & + \sum_{i=1}^{\frac{n}{2}}\left|\{j\mid i < j, -b_i>b_j\}\right|
    + \sum_{i=1}^{\frac{n}{2}}\left|\{j\mid i < j, -b_i>a_j\}\right|.
  \end{align*}
  Now since $(a_i)_{i=1}^\frac{n}{2}$ and $(b_i)_{i=1}^\frac{n}{2}$
  are both increasing sequences with $a_i\geq -1$ and $b_i\leq -2$ for
  all $i$, we see that the first four terms above are all equal to
  zero, so
  \begin{align*}
    \ell(w_n) & = \sum_{i=1}^{\frac{n}{2}}\left|\{j\mid i\leq j, a_i>b_j\}\right|
    + \sum_{i=1}^{\frac{n}{2}}\left|\{j\mid i\leq j, -a_i>b_j\}\right|\\
    & + \sum_{i=1}^{\frac{n}{2}}\left|\{j\mid i < j, -b_i>b_j\}\right|
    + \sum_{i=1}^{\frac{n}{2}}\left|\{j\mid i < j, -b_i>a_j\}\right|.
  \end{align*}
  Now since $a_i>b_j$ for all $i,j\in\mathbb{N}$, we see that
  $a_i>b_j$ holds for all $j\geq i$, so
  \[
  \left|\{j\mid i\leq j, a_i>b_j\}\right|=\frac{n}{2}+1-i.
  \]
  Similarly, since all $b_i<0$
  and since $(b_i)_{i=1}^{\frac{n}{2}}$ is increasing we see that
  $-b_i>b_j$ holds for all $j>i$, so
  \[
  \left|\{j\mid i\leq j, -b_i>b_j\}\right|=\frac{n}{2}-i.
  \]
  Then using the above expressions for $a_i$ and $b_i$ we see that
  $i\leq j$ and $-a_i > b_j$ if and only if $i\leq j \leq\frac{n}{2}+1-i$,
  so
  \[
  \left|\{j\mid i\leq j, -a_i>b_j\}\right| =
  \begin{cases}
    \frac{n}{2}+2-2i & \text{if }i\leq \frac{n+2}{4};\\
    0 & \text{else}.
  \end{cases}
  \]
  Finally we can again use the expressions for $a_i$ and $b_i$ to show
  $i < j$ and $-b_i > a_j$ if and only if $i< j \leq\frac{n}{2}+1-i$,
  so
  \[
  \left|\{j\mid i\leq j, -b_i>a_j\}\right| =
  \begin{cases}
    \frac{n}{2}+1-2i & \text{if }i\leq \frac{n+2}{4};\\
    0 & \text{else}.
  \end{cases}
  \]
  Then we have
  \[
  \ell(w_n) =
  \begin{dcases}
    \sum_{i=1}^{\frac{n}{2}}\left(\left(\frac{n}{2}+1-i\right)
    +\left(\frac{n}{2}-i\right)\right)
    +\sum_{i=1}^{\frac{n}{4}}\left(\left(\frac{n}{2}+2-2i\right)
    +\left(\frac{n}{2}+1-2i\right)\right)
    & \text{if } n\equiv 0\bmod 4;\\
    \sum_{i=1}^{\frac{n}{2}}\left(\left(\frac{n}{2}+1-i\right)
    +\left(\frac{n}{2}-i\right)\right)
    +\sum_{i=1}^{\frac{n+2}{4}}\left(\left(\frac{n}{2}+2-2i\right)
    +\left(\frac{n}{2}+1-2i\right)\right)
    & \text{if } n\equiv 2\bmod 4.
  \end{dcases}
  \]
  If $n\equiv 0\bmod 4$ then we have
  \begin{align*}
  \ell(w_n) &= \sum_{i=1}^{\frac{n}{2}}\left(\left(\frac{n}{2}+1-i\right)
    +\left(\frac{n}{2}-i\right)\right)
    +\sum_{i=1}^{\frac{n}{4}}\left(\left(\frac{n}{2}+2-2i\right)
    +\left(\frac{n}{2}+1-2i\right)\right)\\
    &= \sum_{i=1}^{\frac{n}{2}}\left( n + 1 - 2i\right)
    +\sum_{i=1}^{\frac{n}{4}}\left( n + 3 - 4i\right)\\
    &= \frac{n^2}{2} + \frac{n}{2} - 2\sum_{i=1}^{\frac{n}{2}} i + \frac{n^2}{4} + \frac{3n}{4} - 4\sum_{i=1}^{\frac{n}{4}} i\\
    &= \frac{3n^2}{4} + \frac{5n}{4} - 2\left(\frac{\frac{n}{2}\left(\frac{n}{2} + 1\right)}{2}\right) - 4\left(\frac{\frac{n}{4}\left(\frac{n}{4} + 1\right)}{2}\right)\\
    &= \frac{3n^2}{8} + \frac{n}{4}.
  \end{align*}

  If $n\equiv 2\bmod 4$ then we can use a nearly identical argument to show that
  \[
  \ell(w_n)=\frac{3n^2}{8}+\frac{n}{4}.
  \]

  Now suppose that $n$ is odd.  Then
  \[
  w_n=(a_1,b_1,a_2,b_2,\dots,a_{n-1},b_{n-1},n),
  \]
  where the $a_i$ and $b_i$ are as above.  We see that
  $-a_i,a_i,-b_i,b_i < n$ for each $i$, so
  \[
  \ell(w_n)=\ell(w_{n-1})=\frac{3(n-1)^2}{8}+\frac{n-1}{4}.
  \]
\end{proof}

\section{Properties of bad elements}\label{badre}

We will now find a reduced expression for the bad elements $w_n$.  To
do this it will be convenient to use interval notation.

\begin{definition}\label{interval}
  For $2\leq i\leq j$, denote the element $s_i s_{i+1}\cdots s_{j-1}
  s_j$ by $[i,j]$.  For $i\geq 3$, denote $s_1 s_3 s_4\cdots s_i$ by
  $[1,i]$ and for $j\geq 2$ denote $s_1 s_2 s_3\cdots s_j$ by $[0,j]$. If $0\leq j < i$ and $i\geq 2$ define $[j,i] = [i,j]^{-1}$.
  Finally for $i\leq -3$ and $j\geq 3$, denote $s_i s_{i-1}
  s_{i-2}\cdots s_4s_3s_1s_2s_3s_4\cdots s_j$ by $[-i,j]$.  
\end{definition}

The following two lemmas help describe how these intervals act as
signed permutations.

\begin{lemma}\label{intperm}
  Let $i,j,k\in\mathbb{N}$ be such that $j\geq i\geq 2$.  Then as a
  signed permutation, we have
  \[
  [j,i] = (1,2,\dots, i-2, j, i-1, i, \dots, j-2, j-1, j+1, j+2,
  \dots, n).
  \]
\end{lemma}

\begin{proof}
  We will prove the lemma using induction on $j-i$.  If $j=i$ the
  lemma is true by Proposition~\ref{smultiply}.  Now assume that, as a
  signed permutation, we have
  \[
  [j-1,i] = (1,2,\dots, i-2, j-1, i-1, i, \dots, j-3, j-2, j, j+1,
  \dots, n).
  \]
  Then by Proposition~\ref{smultiply} multiplying on the left by $s_j$
  has the effect of interchanging the entries with values $j$ and
  $j-1$, so we have
  \[
  [j,i]=s_j[j-1,i] = (1,2,\dots, i-2, j, i-1, i, \dots, j-2, j-1, j+1,
  j+2, \dots, n).
  \]  
\end{proof}

\begin{lemma}
  Let $j,k\in\mathbb{N}$ be such that $j\geq 2$.  Then as a signed
  permutation, we have
  \[
  [j,0] = \begin{cases}
    (\underline{1}, \underline{2}, 3, 4, \dots, n) &\text{if } j = 2;\\
    (\underline{1}, \underline{j}, 2, 3, \dots , j-2, j-1, j+1, 
  j+2, \dots, n) &\text{else}.
  \end{cases}
  \]

\end{lemma}

\begin{proof}
  We will prove the lemma using induction on $j$.  If $j=2$ the
  lemma is true by Proposition~\ref{smultiply}.  Now assume that, as a
  signed permutation, we have
  \[
  [j-1,0] = (\underline{1}, \underline{j-1}, 2, 3, \dots , j-3, j-2, j,
  j+1, \dots, n).
  \]
  Then by Proposition~\ref{smultiply} multiplying on the left by $s_j$
  has the effect of interchanging the entries with values $j$ and
  $j-1$, so we have
  \[
  [j,0] = s_j[j-1,0] = (\underline{1}, \underline{j}, 2, 3, \dots ,
  j-2, j-1, j+1, j+2, \dots, n).
  \]  
\end{proof}

Now we are ready to write down a reduced expression for $w_n$.  We
begin by finding a (not necessarily reduced) expression for $w_n$ in
terms of generators of the Coxeter group.

\begin{lemma}\label{equalbad}
  Using the above notation, we have
  \[
  w_n =
  \begin{dcases*}
    [2,0][4,0]\cdots[n-2,0][n,0][n-k,n-2k]\cdots[n-1,n-2][n,n]
    & if $n$ is even;\\
    [2,0][4,0]\cdots[m-2,0][m,0][m-k,m-2k]\cdots[m-1,m-2][m,m]
    & if $n$ is odd;
  \end{dcases*}
  \]
  where $m=n-1$ and
  \[
  k =
  \begin{dcases*}
    \frac{n}{2}-2 & if $n$ is even;\\
    \frac{n-1}{2}-2 & if $n$ is odd.\;
  \end{dcases*}
  \]
\end{lemma}

\begin{proof}
  Suppose $n$ is even. Define
  \[
  w'_n = [2,0][4,0]\cdots[n-2,0][n,0][n-k,n-2k]\cdots[n-1,n-2][n,n].
  \]
  There are $\frac{n}{2}$ intervals in $w'_n$ that end in 0, and all
  other intervals fix 1, so $w(1)=\left(-1\right)^{n/2}$.  

  Now suppose that $i\in{\bf n}\setminus\{1\}$ is odd.  Then using
  Lemma~\ref{intperm} we have
  \begin{align*}
    w'_n(i) & = [2,0][4,0]\cdots[n-2,0][n,0][n-k,n-2k]\cdots[n-1,n-2][n,n](i)\\
    & = [2,0][4,0]\cdots[n-2,0][n,0][n-k,n-2k]\cdots
    \left[\frac{n+i-1}{2},i-1\right]\left[\frac{n+i+1}{2},i+1\right](i) \\
    & = [2,0][4,0]\cdots[n-2,0][n,0][n-k,n-2k]\cdots
    \left[\frac{n+i-1}{2},i-1\right]\left(\frac{n+i+1}{2}\right)\\
    & = [2,0][4,0]\cdots[n-4,0][n-2,0][n,0]\left(\frac{n+i+1}{2}\right)\\
    & = [2,0][4,0]\cdots[n-4,0][n-2,0]\left(\frac{n+i+1}{2}-1\right)\\
    & = [2,0][4,0]\cdots[n-4,0]\left(\frac{n+i+1}{2}-2\right)\\
    & \hspace{.2cm}\vdots \\
    & = [2,0][4,0]\cdots[i-1,0][i+1,0]\left(i+1\right)\\
    & = [2,0][4,0]\cdots[i-1,0]\left(i\right)\\
    & = i.
  \end{align*}

  If $i = 2$ then 
  \begin{align*}
    w'_n(2) & = [2,0][4,0]\cdots[n-2,0][n,0][n-k,n-2k]\cdots[n-1,n-2][n,n](2)\\
    & = [2,0][4,0]\cdots[n-2,0][n,0]\left(2\right) \\
    & = [2,0][4,0]\cdots[n-2,0]\left(-n\right) \\
    & \hspace{.2cm}\vdots \\
    & = -n.
  \end{align*}

  Next suppose that $i > 2$ is even.  Then
  \begin{align*}
    w'_n(i) & = [2,0][4,0]\cdots[n-2,0][n,0][n-k,n-2k]\cdots[n-1,n-2][n,n](i)\\
    & = [2,0][4,0]\cdots[n-2,0][n,0][n-k,n-2k]\cdots
    \left[\frac{n+i-2}{2},i-2\right]\left[\frac{n+i}{2},i\right](i) \\
    & = [2,0][4,0]\cdots[n-2,0][n,0][n-k,n-2k]\cdots
    \left[\frac{n+i-2}{2},i-2\right](i-1) \\
    & \hspace{.2cm}\vdots \\
    & = [2,0][4,0]\cdots[n-2,0][n,0]\left(\frac{i+2}{2}\right) \\
    & \hspace{.2cm}\vdots \\
    & = [2,0][4,0]\cdots[n-i,0][n+2-i,0](2) \\
    & = [2,0][4,0]\cdots[n-i,0](-(n+2-i)) \\
    & = -(n+2-i).
  \end{align*}
  Then $w'_n(i)=w_n(i)$ for all $i\in{\bf n}$ by Corollary~\ref{dmap},
  so $w'_n = w_n$.

  If $n$ is odd we see that $w_n = w_{n-1}$ in $W(D_n)$, so they must
  have the same reduced expression.
\end{proof}

\begin{lemma}\label{reducedbad}
The expression for $w_n$ in Lemma~\ref{equalbad} is reduced.  
\end{lemma}

\begin{proof}
  Suppose that $n$ is even. By Lemma~\ref{badlength} we have
  $\ell(w_n)\leq\frac{3}{8}n^2+\frac{1}{4}n$.  Let $r$ be the number
  of generators in the given expression for $w_n$. Then we have
  \begin{align*}
    r &= \sum_{i=1}^{\frac{n}{2}} 2i
    + \sum_{i=1}^{\frac{n}{2}-1} i\\
    &= 2\sum_{i=1}^{\frac{n}{2}} i
    + \sum_{i=1}^{\frac{n}{2}} i-\frac{n}{2}\\
    &= \frac{3}{2}\cdot\frac{n}{2}\left(\frac{n}{2}+1\right)-\frac{n}{2}\\
    &= \frac{3}{8}n^2+\frac{1}{4}n.  
  \end{align*}
  Then $\ell(w_n) = \frac{3}{8}n^2+\frac{1}{4}n$ and the above
  expression for $w_n$ is reduced. We can use an identical argument to
  show that the expression for $w_n$ is reduced if $n$ is odd.
\end{proof}

Recall that a bad element is an element that is not a product of
commuting generators and that has no reduced expressions beginning or
ending in two noncommuting generators.

\begin{theorem}~\label{longbad} 
  Let $W=W(D_n)$. Then there is a unique longest bad element, $w_n\in
  W$.  Every other bad element in $W$ is of the form $w_k\cdot u$
  where $k<n$ and where $u$ is a product of mutually commuting
  generators not in $\supp(w_k)$. Furthermore, if $n$ is odd then
  $w_n=w_{n-1}$.
\end{theorem}

\begin{proof}
  Let $n\in\mathbb{N}$ and let $w\in W$ be bad.  Then by
  Theorem~\ref{dbad} we must have $w=w_k\cdot u$ reduced where $k\leq
  n$ and $u$ is a product of mutually commuting generators that are
  not in $\supp(w_k)$. Write $i = n-k$. Since $\supp(u) \subset \{s_{k+2}, s_{k+3},
  \dots, s_n\}$ we see $\ell(u) < (i/2)$.  Then $\ell(w) = \ell(w_k) +
  \ell(u)$ so by Lemma~\ref{badlength}
  \begin{align*}
    \ell(w) &\leq \frac{3}{8}(n-i)^2+\frac{1}{4}(n-i) + \frac{i}{2}\\
         &\leq \frac{3}{8}n^2 + \frac{1}{4}n + \frac{1}{8}i\cdot(-6n+3i+2).
  \end{align*}
  Since $n \geq 4$ and $n \geq i$ we see that $(-6n+3i+2) < 0$, so 
  $l(w) \leq (3/8)n^2+(1/4)n = l(w_n)$. 
\end{proof}

\begin{corollary}\label{badnotfc}
  If $w\in W(D_n)$ is bad then $w\not\in W_c$.
\end{corollary}

\begin{proof}
  By Theorem~\ref{longbad} we can write $w = w_m \cdot u$ reduced
  where $4\leq m\leq n$ and $u$ is a product of mutually commuting
  generators that are not in $\supp(w_m)$. Since $w_m = w_{m-1}$ when
  $m$ is odd we may assume that $m$ is even. By Lemma~\ref{equalbad}
  we have
  \[
  w_m = [2,0][4,0]\cdots[m-2,0][m,0][m-k,m-2k]\cdots[m-1,m-2][m,m]
  \]
  where $k = \frac{m}{2}-2$. We can write $[m-k,m-2k] =
  s_{m-k}[m-k-1,m-2k]$ and commute $s_{m-k}$ to the left to obtain
  \[
  w_m = [2,0]\cdots[m,m-k+1]\cdot s_{m-k}s_{m-k-1}s_{m-k}\cdot [m-k-2,
  0][m-k-1, m-2k]\cdots[m,m],
  \]
  so $w = w'\cdot s_{m-k}s_{m-k-1}s_{m-k}\cdot w''u$ reduced for some
  $w',w''\in W$, and thus $w\not\in W_c$.
\end{proof}

In \cite{green2009leading}, Green uses the fact that Coxeter groups of
type $\widetilde A$, and therefore of type $A$, have no bad elements
to show that $\mu(x,w)\in\{0,1\}$ if $x$ is fully commutative.  In
type $D$ we can use the proof of \cite[Theorem 3.1]{green2009leading}
to show that $\mu(x,w)\in\{0,1\}$ if $x$ is fully commutative and $w$
is not bad.  However, the case where $w$ is bad is much harder.  We
will show that we can reduce the case where $w$ is bad to calculating
$\mu(x_n,w_n)$, where $w_n$ is the longest bad element in $W(D_n)$ and
where $x_n = \prod_{s\in\mathcal{L}(w)}s$.  We see that since $w_n =
w_{n+1}$ for even $n$ we only need to calculate $\mu(x_n, w_n)$ for
even $n$.  Furthermore, if $n\equiv 2\text{ or }4\bmod 8$, then
$\ell(w_n)-\ell(x_n)$ is even, so $\mu(x_n, w_n) = 0$.  However, if
$n\equiv 0\text{ or }6\bmod 8$ the calculation is much harder.

As we will see, Lusztig's $a$-function gives us a way to bound the
degree of $P_{e,w}$ for an element $w\in W$.  With this in mind, we
first look at $P_{e,w_n}$.

\begin{definition}
  Let $n\geq 4$.  Then define
  \[
  x_n =
  \begin{cases}
    s_1s_2s_4s_6\cdots s_{n-2}s_n &\mbox{if $n$ is even;}\\
    s_1s_2s_4s_6\cdots s_{n-3}s_{n-1} &\mbox{if $n$ is odd.}
  \end{cases}
  \]
  Note that $x_n$ is a product of mutually commuting generators.
\end{definition}

\begin{lemma}\label{pew}
  We have $P_{x_n,w_n} = P_{e, w_n}$.
\end{lemma}

\begin{proof}
  We repeatedly apply Proposition~\ref{klpolreduce} (1), starting with $x=e$,
  $w=w_n$, and taking $s$ from the set $\mathcal{L}(x_n)=\{s_1, s_2,
  s_4, s_6,\dots , s_{n-2}, s_n\}$.
\end{proof}

This tells us that $\mu(x_n,w_n)$ is equal to the coefficient of
$q^{\frac{1}{2}(\ell(w_n)-\ell(x_n)-1)}$ in $P_{e,w_n}$.  Using
Proposition~\ref{abound}, if we calculate $a(w_n)$ we can find a bound
for the degree of $P_{e,w_n}$.  This will allow us to calculate
$\mu(x_n,w_n)$ in certain cases.

In order to simplify the calculation of $a(w_n)$, we will next find an
element $u_n$ in the same two-sided cell as $w_n$.  The calculation of
$a(u_n)$ will be made easy by lemmas~\ref{alongest}
and~\ref{aparabolic}.  To find such an element we will use domino
tableaux in order to calculate the Kazhdan--Lusztig cells of $W(D_n)$.

\begin{lemma}\label{starredbad}
  Recall the definition of star reducible from
  Definition~\ref{starreddef}. Let $w\in W$ be such that $w\not\in
  W_c$. Then $w$ is star reducible to either
  \begin{enumerate}
  \item a bad element; or
  \item an element $x$ such that either $\mathcal{L}(x)$ or
    $\mathcal{R}(x)$ is not commutative.
  \end{enumerate}
\end{lemma}

\begin{proof}
  We will proceed by induction on $\ell(w)$. We see that since
  $w\not\in W_c$ we must have $\ell(w) \geq 3$ and we know that $w$ is
  not a product of commuting generators. If $\ell(w) = 3$ then $w=sts$
  where $s$ and $t$ are noncommuting generators, so we are
  done. Suppose that $\ell(w) = r$. If $w$ is not bad and both
  $\mathcal{L}(w)$ and $\mathcal{R}(w)$ are commutative then we have
  either $w\in st\cdot w'$ reduced or $w = w'\cdot ts$ reduced for
  some pair of noncommuting generators $s$ and $t$ and some $w'\in
  W$. Without loss of generality suppose $w = st\cdot w'$
  reduced. Then $\expl{w}{*} = t\cdot w'$, so $\ell(\expl{w}{*}) = r -
  1$, and by induction $\expl{w}{*}$ is star reducible to either a bad
  element or an element $x$ such that either $\mathcal{L}(x)$ or
  $\mathcal{R}(x)$ is not commutative, thus $w$ is star reducible to
  either a bad element or an element $x$ such that either
  $\mathcal{L}(x)$ or $\mathcal{R}(x)$ is not commutative.
\end{proof}

\begin{corollary}\label{starred}
  Let $w\in W$. Then $w$ is star reducible to
  either
  \begin{enumerate}
  \item a product of mutually commuting generators,
  \item a bad element, or
  \item an element $x$ such that either $\mathcal{L}(x)$ or
    $\mathcal{R}(x)$ is not commutative.
  \end{enumerate}
\end{corollary}

\begin{proof}
  This is immediate from Proposition~\ref{starredcomm} and
  Lemma~\ref{starredbad}.
\end{proof}

\chapter{Kazhdan--Lusztig cells and domino tableaux}\label{klcellsdomtab}

If $W=W(A_n)$, we can use the Robinson--Schensted algorithm to
describe the left and right Kazhdan--Lusztig cells.  Each element of
of $w$ may be viewed as a permutation of $\{1,2,\dots,n+1\}$.  We can
use the Robinson--Schensted algorithm to construct an $(n+1)$-tableau,
$T(w)$, corresponding to $w$ as in \cite[4.1]{fulton1997young}.  Then
$x,w\in W$ are in the same left (respectively, right) cell if and only
if $T(x)=T(w)$ (respectively, $T(x^{-1})=T(w^{-1})$) \cite[Theorem
1.7.2]{shi1986cells}.

If $W=W(D_n)$, we can use an algorithm similar to the
Robinson--Schensted algorithm to construct tableaux with
dominoes. This algorithm was developed by D. Garfinkle in
\cite{garfinkleclassI}. We assign each element $w\in W$ a so-called
domino tableau, $T_L(w)$.  Unlike in type $A$, two elements of a
single cell $W$ may have two different domino tableaux.  However, we
define the notion of cycles of domino tableaux, and use the notion of
moving a tableau through a cycle to create a new domino tableau.  Then
we introduce an equivalence class, $\approx$, on tableaux such that
$T(w)\approx T(y)$ if and only if we can obtain $T_L(w)$ from $T_L(y)$
by moving through cycles.  Once we do this, we will see that the
partition of $W$ induced by $\approx$ is exactly the left cells of
$W$.  We will then use these domino tableaux to find an element,
$u_n$, in the same two-sided cell as $w_n$ with a relatively easy to
calculate $a$-value.

In Section~\ref{domdef} we begin with an account of Garfinkle's
definitions corresponding to domino tableaux as seen in
\cite{garfinkleclassI}. In Section~\ref{domtab} we define the
algorithm developed by Garfinkle in \cite{garfinkleclassI} used to
assign a pair of tableaux to an element of $W$. In
Section~\ref{domcyc} we follow Garfinkle's methods in
\cite{garfinkleclassI} to define the notion of cycles of a domino
tableau and use these cycles to define the equivalence relation
$\approx$.  

In Chapter~\ref{dombad} we will apply Garfinkle's notion of cycles to
tableaux of our bad elements.  This allows us to use $\approx$ to find
an new element $u_n$ whose $a$-value is relatively easy to calculate.
We conclude by calculating $a(w_n)$.

\section{Definitions}\label{domdef}

Let $\mathbb{N}$ be the set of natural numbers starting with 1 and let
$\mathbb{N}_0=\mathbb{N}\cup\{0\}$.  For $n\in\mathbb{N}$ define ${\bf
  n}=\{i\in\mathbb{N}\mid i\leq n\}$. Consider the set $\{S_{i,j}\mid
i,j\in\mathbb{N}_0\}$ of positions in the quadrant, where $i$ denotes
the row, increasing left to right, and $j$ denotes the column,
increasing top to bottom.

\begin{definition}
  Let $\mathcal{F}=\{S_{i,j}\}_{i,j\in \mathbb{N}}$, and let
  $\mathcal{F}^0=\{S_{i,j}\}_{i,j\in\mathbb{N}_0}$. We call the elements of
  $\mathcal{F}$ and $\mathcal{F}^0$ \textbf{squares}.
\end{definition}

\begin{example}\label{box}
  Let $J=\{S_{1,1}, S_{1,2}, S_{1,3}, S_{1,4}, S_{2,1}, S_{2,2}\}$.  We can
  visualize $J$ as a box diagram:
  \begin{center}
    \begin{tikzpicture}[node distance=0 cm,outer sep = 0pt]
      \tikzstyle{sq}=[rectangle, draw, thick,
      minimum width=.5cm, minimum height=.5cm]
      \node[sq] (11) at (   1,  1) {};
      \node[sq] (12) [right = of 11] {};
      \node[sq] (13) [right = of 12] {};
      \node[sq] (14) [right = of 13] {};
      \node[sq] (21) [below = of 11] {};
      \node[sq] (22) [right = of 21] {};
      \node at (0,.75) {$J=$};
      \node at (3,.75) {.};
    \end{tikzpicture}
  \end{center}
\end{example}

\begin{definition}
  A subset $J\subseteq\mathcal{F}$ is a \textbf{Young diagram} if it
  satisfies all of the following conditions:
  \begin{enumerate}
  \item $J$ is finite;
  \item for each $i$ there exists $i_j$ such that $S_{i,k}\in J$ if
    and only if $0\leq k\leq i_j$;
  \item for each $j$ there exists $j_i$ such that $S_{j,k}\in J$ if
    and only if $0\leq k\leq j_i$.
  \end{enumerate}
\end{definition}

\begin{example}
  In Example~\ref{box}, above, $J$ is a Young diagram.  However, if
  we let
  \begin{center}
    \begin{tikzpicture}[node distance=0 cm,outer sep = 0pt]
      \tikzstyle{sq}=[rectangle, draw, thick,
      minimum width=.5cm, minimum height=.5cm]
      \node[sq] (11) at        (1,1) {};
      \node[sq] (12) [right = of 11] {};
      \node[sq] (13) [right = of 12] {};
      \node[sq] (14) [right = of 13] {};
      \node[sq] (21) [below = of 11] {};
      \node[sq] (22) at        (1.5,0) {};
      \node at (0,.5) {$J'=$};
    \end{tikzpicture}
  \end{center}
  then we see that $J'$ is not a Young diagram.
\end{example}

\begin{definition}
  Let $J\subset\mathcal{F}$ be a Young diagram.  Then define
  \begin{enumerate}
  \item $\rho_i(J)  =\max\{0,\max\{j\mid S_{i,j}\in J\}\}$, and
  \item $\kappa_j(J)=\max\{0,\max\{i\mid S_{i,j}\in J\}\}$.
  \end{enumerate}
\end{definition}

\begin{remark}
  We see that $\rho_i(J)$ is the number of boxes in the $i$th row of
  $J$, and $\kappa_j(J)$ is the number of boxes in the $j$th column of
  $J$.
\end{remark}

\begin{example}
  If $J$ is as in Example~\ref{box}, then $\rho_1(J)=4$
  and $\kappa_1(J)=2$.
\end{example}

\begin{definition}
  A subset $D\subset\mathcal{F}$ is a called a \textbf{domino} if
  $D=\{S_{i,j},S_{i,j+1}\}$ or $D=\{S_{i,j},S_{i+1,j}\}$ for some
  $i,j\in\mathbb{N}$.
\end{definition}

\begin{example}
  Let $J$ be as in Example~\ref{box}. The pairs $\{S_{1,1},S_{2,1}\}$,
  $\{S_{1,2},S_{2,2}\}$, and $\{S_{1,3},S_{1,4}\}$ are all dominoes.
  With these pairings in mind we can write $J$ as a disjoint union of
  dominoes:
  \begin{center}
    \begin{tikzpicture}[node distance=0 cm,outer sep = 0pt]
      \tikzstyle{ver}=[rectangle, draw, thick,
      minimum width=.5cm, minimum height=1cm]
      \tikzstyle{hor}=[rectangle, draw, thick,
      minimum width=1cm, minimum height=.5cm]
      \node[ver] (1) at (   1,  1) {};
      \node[ver] (2) [right  = of 1] {};
      \node[hor] (3) at (2.25, 1.25) {};
      \node at (0,1) {$J=$};
      \node at (3,1) {.};
    \end{tikzpicture}
  \end{center} 
\end{example}

\begin{definition}
  Let $M\subseteq\mathbb{N}$ be finite.  Define projections
  \[
  \pi_1 : \mathcal{F}\times M\rightarrow\mathcal{F} : (f,m)\mapsto f
  \]
  and
  \[
  \pi_2 : \mathcal{F}\times M\rightarrow M : (f,m)\mapsto m.
  \]
  Let $T\subseteq \mathcal{F}\times M$ satisfy the following
  conditions:
  \begin{enumerate}
  \item $\pi_1|_T$ is injective;
  \item $\pi_1(T)$ is a Young diagram;
  \item $(\pi_2)^{-1}(k)$ is a domino for all $k\in M$.
  \item Suppose $(S_{i,j},k)\in T$.  Then if $(S_{i,j+1},k_1)\in T$ or
    $(S_{i+1,j},k_2)\in T$ we have $k\leq k_1,k_2$.
\end{enumerate}
Then we call $T$ a \textbf{domino tableau}.  Let $\mathcal{T}(M)$ be
the set of all domino tableaux.  
\end{definition}

\begin{example}\label{domino}
  Using the above definition we see that each element
  $T\in\mathcal{T}(M)$ is a disjoint union of labeled dominoes in the
  shape of a Young diagram, such that the labels increase along rows
  and columns.  For example, let $M={\bf 5}$, and let
  \[
  T=\{(S_{1,1},1), (S_{1,2},3), (S_{1,3},3), (S_{2,1},1), (S_{2,2},4),
  (S_{2,3},5), (S_{3,1},2), (S_{3,2},4), (S_{3,3},5), (S_{4,1},2)\}.
  \]
  Then $T\in \mathcal{T}(M)$ and we write
  \begin{center}
    \begin{tikzpicture}[node distance=0 cm,outer sep = 0pt]
        \tikzstyle{ver}=[rectangle, draw, thick,
        minimum width=.5cm, minimum height=1cm]
        \tikzstyle{hor}=[rectangle, draw, thick,
        minimum width=1cm, minimum height=.5cm]
        \node[ver] (1) at (   1,  1.5) {1};
        \node[ver] (2) [below  = of 1] {2};
        \node[hor] (3) at (1.75, 1.75) {3};
        \node[ver] (4) at ( 1.5,    1) {4};
        \node[ver] (5) [right  = of 4] {5};
        \node at (0,1) {$T=$};
        \node at (2.5,1) {.};
    \end{tikzpicture}
  \end{center}
\end{example}

\begin{definition}
  Let $M\subset \mathbb{N}$, let $T\in\mathcal{T}(M)$ and let
  $k\in M$.
  \begin{enumerate}
  \item The \textbf{domino with label $k$ in $T$} is given by
    $D(T,k)=(\pi_2)^{-1}(k)$.
  \item The \textbf{position of domino with label $k$ in $T$} is given
    by $P(T,k)=\pi_1(D(T,k))$.
  \item The \textbf{shape} of $T$ is $\shape(T)=\pi_1(T)$.
  \item Let $i,j\in\mathbb{N}$.  Then define a map
    \[
    N:\mathcal{T}(M)\times\mathcal{F}^0\rightarrow M\cup \{0,\infty\}
    \]
    by
    \[
    N(T,S_{i,j})=
    \begin{cases}
      k &\mbox{if } (S_{i,j},k)\in T;\\
      0 &\mbox{if } i=0\mbox{ or }j=0;\\
      \infty &\mbox{else}.
    \end{cases}
    \]

  \end{enumerate}

\end{definition}

\begin{example}
  Let $T$ be as in Example~\ref{domino}.  We see that
  $D(T,3)=\{(S_{1,2},3),(S_{1,3},3)\}$ and
  $P(T,3)=\{S_{1,2},S_{1,3}\}$.  Also, $\shape(T)$ is the unlabeled
  tableau with the same outline as $T$.  Thus
  \begin{center}
    \begin{tikzpicture}[node distance=0 cm,outer sep = 0pt]
      \tikzstyle{sq}=[rectangle, draw, thick,
      minimum width=.5cm, minimum height=.5cm]
      \node[sq] (11) at (   1.5,  2) {};
      \node[sq] (12) [right = of 11] {};
      \node[sq] (13) [right = of 12] {};
      \node[sq] (21) [below = of 11] {};
      \node[sq] (22) [below = of 12] {};
      \node[sq] (23) [below = of 13] {};
      \node[sq] (31) [below = of 21] {};
      \node[sq] (32) [below = of 22] {};
      \node[sq] (33) [below = of 23] {};
      \node[sq] (41) [below = of 31] {};
      \node at (0,1.25) {$\shape(T)=$};
      \node at (3,1.25) {.};
    \end{tikzpicture}
  \end{center}

  Finally, $N(T,S_{i,j})$ is the label
  of block $S_{i,j}$ in $T$, if it exists.  Otherwise, we say that
  $N(T,S_{i,j})=0$ if either $i=0$ or $j=0$, and $N(T,S_{i,j})=\infty$
  otherwise.  Then, $N(T,S_{2,3})=5$, $N(T,S_{0,2})=0$, and
  $N(T,S_{4,2})=\infty$.
\end{example}

\begin{definition}
  Let $M=\{e_1,e_2,\dots,e_k\}$ and let $T\in\mathcal{T}(M)$.  Then we
  define
  \[
  T(j)=T\setminus\left(\bigcup_{e_i> j} D(T,e_{i})\right).
  \]
\end{definition}

\begin{example}
  We note that $T(j)$ is obtained from $T$ by removing all dominoes
  with label strictly greater than $j$. Then using
  Example~\ref{domino} we see that
    \begin{center}
      \begin{tikzpicture}[node distance=0 cm,outer sep = 0pt]
        \tikzstyle{ver}=[rectangle, draw, thick,
        minimum width=.5cm, minimum height=1cm]
        \tikzstyle{hor}=[rectangle, draw, thick,
        minimum width=1cm, minimum height=.5cm]
        \node[ver] (1) at (   1,  1.5) {1};
        \node[ver] (2) [below  = of 1] {2};
        \node[hor] (3) at (1.75, 1.75) {3};
        \node at (0,1) {$T(3)=$};
        \node at (2.5,1) {.};
      \end{tikzpicture}
    \end{center}
\end{example}

We are eventually working towards a way to add a domino to a tableau
to create a new tableau.  To do this, we will next define a way to
shuffle a domino that overlaps with a particular tableau into a
position that allows it to fit into the tableau without overlapping.

\begin{definition}\label{shuffle}
  Let $J\subset\mathcal{F}$ be a Young diagram and let $P=\{S_{i,j},
  S_{i,j+1}\}$ (respectively, $P=\{S_{i,j},
  S_{i,j+1}\}$) be a domino.  We define $A(J,P)$ in the following
    cases:
  \begin{enumerate}
  \item If $j=\rho_i(J)+1$ (respectively, $i=\kappa_j(J)+1$) then $A(J,P)=P$.
  \item If $j=\rho_i(J)-1$ (respectively, $i=\kappa_j(J)-1$) and
    $\rho_{i+1}(J)< j$ (respectively, $\kappa_{j+1}(J)< i$) then
    $A(J,P)=\{S_{i+1,r}S_{i+1,r+1}\}$ (respectively,
    $A(J,P)=\{S_{r,j+1}S_{r+1,j+1}\}$) where $r=\rho_{i+1}(J)+1$
    (respectively, $r=\kappa_{j+1}(J)+1$).
  \item If $j=\rho_i(J)$ (respectively, $i=\kappa_j(J)$) and $\rho_{i+1}(J)=j$
    (respectively, $\kappa_{j+1}(J)=i$) then $A(J,P)=\{S_{i,j+1},S_{i+1,j+1}\}$
    (respectively, $A(J,P)=\{S_{i+1,j},S_{i+1,j+1}\}$).
  \end{enumerate}
\end{definition}

\begin{example}
  Let $J$ be given as below, and let $P=\{S_{1,4},S_{1,5}\}$,
  $Q=\{S_{1,2},S_{1,3}\}$, and $R=\{S_{2,1},S_{2,2}\}$:
  \begin{center}
    \begin{tikzpicture}[node distance=0 cm,outer sep = 0pt]
      \tikzstyle{ver}=[rectangle, draw, thick,
      minimum width=.5cm, minimum height=1cm]
      \tikzstyle{hor}=[rectangle, draw, thick,
      minimum width=1cm, minimum height=.5cm]
      \node[ver] (1) at (   1,   1) {};
      \node[hor] (2) at (1.75,1.25) {};
      \node[ver] (3) [below = of 1] {};
      \node at (0,.5) {$J=$};
      \node at (2.5,.5) {.};
    \end{tikzpicture}
  \end{center}
  Then
  \begin{center}
    \begin{tikzpicture}[node distance=0 cm,outer sep = 0pt]
      \tikzstyle{ver}=[rectangle, draw, thick,
      minimum width=.75cm, minimum height=1.5cm]
      \tikzstyle{hor}=[rectangle, draw, thick,
      minimum width=1.5cm, minimum height=.75cm]
      \node[ver] (1) at (   2,   3) {};
      \node[hor] (2) at (3.125,3.375) {};
      \node[ver] (3) [below = of 1] {};
      \node[hor, color=red] (P) [right = of 2] {$P$};
      \node at (0,2.25) {$J\cup P=$};
      \node at (6,2.25) {,};
    \end{tikzpicture}
  \end{center}
  so we are in case (1) of Definition~\ref{shuffle}.  Thus
  \begin{center}
    \begin{tikzpicture}[node distance=0 cm,outer sep = 0pt]
      \tikzstyle{ver}=[rectangle, draw, thick,
      minimum width=.75cm, minimum height=1.5cm]
      \tikzstyle{hor}=[rectangle, draw, thick,
      minimum width=1.5cm, minimum height=.75cm]
      \node[ver] (1) at (   2,   3) {};
      \node[hor] (2) at (3.125,3.375) {};
      \node[ver] (3) [below = of 1] {};
      \node[hor] (P) [right = of 2] {$A(J,P)$};
      \node at (0,2.25) {$J\cup A(J,P)=$};
      \node at (6,2.25) {.};
    \end{tikzpicture}
  \end{center}
  We also see that
  \begin{center}
    \begin{tikzpicture}[node distance=0 cm,outer sep = 0pt]
      \tikzstyle{ver}=[rectangle, draw, thick,
      minimum width=.75cm, minimum height=1.5cm]
      \tikzstyle{hor}=[rectangle, draw, thick,
      minimum width=1.5cm, minimum height=.75cm]
      \node[ver] (1) at (   2,   3) {};
      \node[hor] (2) at (3.125,3.375) {};
      \node[ver] (3) [below = of 1] {};
      \node[hor, color=red] (P) at (3.125,3.375) {$Q$};
      \node at (0,2.25) {$J\cup Q=$};
    \end{tikzpicture}
  \end{center}
  so we are in case (2) of Definition~\ref{shuffle}.  Thus
  \begin{center}
    \begin{tikzpicture}[node distance=0 cm,outer sep = 0pt]
      \tikzstyle{ver}=[rectangle, draw, thick,
      minimum width=.75cm, minimum height=1.5cm]
      \tikzstyle{hor}=[rectangle, draw, thick,
      minimum width=1.5cm, minimum height=.75cm]
      \node[ver] (1) at (   2,   3) {};
      \node[hor] (2) at (3.125,3.375) {};
      \node[ver] (3) [below = of 1] {};
      \node[hor] (P) [below = of 2] {\hspace{-.01in}$A(J,Q)$};
      \node at (0,2.25) {$J\cup A(J,Q)=$};
      \node at (4.5,2.25) {.};
    \end{tikzpicture}
  \end{center}
  Note that the tableau $J\cup A(J,Q)=$ would not be well-defined if
  we had $S_{2,2}\in J$.  However this would violate the condition
  that $\rho_{i+1}(J)<j$ for $i=1$ and $j=2$, so the resulting
  tableau is guaranteed to be well-defined.
  
  Similarly,
  \begin{center}
    \begin{tikzpicture}[node distance=0 cm,outer sep = 0pt]
        \tikzstyle{ver}=[rectangle, draw, thick,
        minimum width=.75cm, minimum height=1.5cm]
        \tikzstyle{hor}=[rectangle, draw, thick,
        minimum width=1.5cm, minimum height=.75cm]
        \node[ver] (1) at (   2,   3) {};
        \node[hor] (2) at (3.125,3.375) {};
        \node[ver] (3) [below = of 1] {};
        \node[hor, color=red] (P) at (2.375,2.625) {\hspace{.25in}$R$};
        \node at (0,2.25) {$J\cup R=$};
     \end{tikzpicture}
  \end{center}
  so we are in case (3) of Definition~\ref{shuffle}.  Thus
  \begin{center}
    \begin{tikzpicture}[node distance=0 cm,outer sep = 0pt]
      \tikzstyle{ver}=[rectangle, draw, thick,
      minimum width=.75cm, minimum height=1.5cm]
      \tikzstyle{hor}=[rectangle, draw, thick,
      minimum width=1.5cm, minimum height=.75cm]
      \node[ver] (1) at (   2,   3) {};
      \node[hor] (2) at (3.125,3.375) {};
      \node[ver] (3) [below = of 1] {};
      \node[ver] (P) at (2.75,2.25) {\begin{sideways}
          $A(J,R)$ \end{sideways}};
      \node at (0,2.25) {$J\cup A(J,R)=$};
      \node at (4.5,2.25) {.};
    \end{tikzpicture}
  \end{center}
  As before, the tableau $J\cup A(J,Q)=$ would not be well-defined if
  we had $S_{3,2}\in J$.  However this would violate the condition
  that $\rho_{i}(J)=j$ for $i=2$ and $j=1$, so the resulting tableau
  is well-defined.
\end{example}

\section{Constructing tableaux}\label{domtab}
We will now define an algorithm that assigns a domino tableau to each
element in $W(D_n)$.  Let $W=W(D_n)$ and let $w\in W$.  We will
construct a domino tableau from $w$ using the signed permutation
representation of $w$.  Using a method
similar to row insertion in standard tableaux, we add labeled dominoes
to a tableau.  If $w^{-1}(k)$ is positive we initially add a
horizontal domino with label $k$.  If $w^{-1}(k)$ is negative then we
initially add a vertical domino with label $k$.  When dominoes
overlap, we employ a shuffling technique that allows us to create a
valid domino tableau.  We begin by defining a way to write each $w$ as
a set of ordered triples,
\[
\left\{(k,\,|w(k)|\,,\sign(w(k)))\mid k\in {\bf n}\right\}.
\]

Let $M_1,M_2\subset\mathbb{N}$ be finite with $|M_1|=|M_2|$ and
let $p_1$, $p_2$, and $p_3$ be the projections of $M_1\times
M_2\times\{\pm 1\}$ onto the first, second, and third coordinate,
respectively.  Let
\[
\mathcal{W}(M_1,M_2) = \left\{ w\subset M_1\times M_2\times\{\pm 1\} \;\middle|\;
\begin{array}{c} p_1|_w\text{ and }p_2|_w\text{ are bijections and}\\
\left|\left(p_3^{-1}(-1)\cap w\right)\right|\equiv 0\bmod 2\end{array}\right\}.
\]

Let $W=W(D_n)$.  Then we can realize each element of $W$ as a signed
permutation.  Define the map
\[
\delta:W\rightarrow\mathcal{W}({\bf n},{\bf n})
\]
by $\delta(w)= \left\{\left(i, \,\vert w(i) \vert\, , \,\sign(w(i)) \right)
  \;\middle\vert\; i\in{\bf n}\right\}$.

\begin{example}
  Let $w=w_4=(1,\underline{4},3,\underline{2})\in W(D_4)$.  Then
  \[
  \delta(w)=\{(1,1,1),(2,4,-1),(3,3,1),(4,2,-1)\}.
  \]
\end{example}

\begin{definition}\label{Aset}
  Let $M\subset\mathbb{N}$ be finite.  Define
  \[
  \mathcal{A}(M)=\left\{(T,v,\epsilon)\mid v\in M,
    T\in\mathcal{T}(M\setminus\{v\}), \epsilon\in\{\pm 1\}\right\}.
  \]
\end{definition}

We can think of an element of $\mathcal{A}(M)$ as a domino tableau on
the set $M\setminus\{v\}$ paired with a horizontally aligned domino
with label $v$ if $\epsilon=1$, or a vertically aligned domino with
label $v$ if $\epsilon=-1$.  We proceed by defining a map from
$\mathcal{A}(M)$ to $\mathcal{T}(M)$.  This will allow us to add the
domino with label $v$ to the tableau to form a new tableau using all
elements of $M$.  We will use the idea of shuffling introduced in
Definition~\ref{shuffle}.

\begin{definition}\label{alphamap}
  We define a map $\alpha:\mathcal{A}(M)\to\mathcal{T}(M)$
  inductively.  Let $M=\{e_1,\cdots,e_n\}\subset \mathbb{N}$ be such
  that $e_1<e_2<\cdots <e_n$ and let $v=e_j$.  Suppose that $\alpha$
  is defined for all $M'$ with $|M'|<n$.  Let
  $(T,v,\epsilon)\in\mathcal{A}(M)$. We have two cases.
  \begin{itemize}
  \item[\textbf{Case 1.}] Suppose $v=e_n$.  Then
    \[
    \alpha(T,v,\epsilon)=
    \begin{cases}
      T\cup \{(S_{1,\rho_1(T)+1},e_n),(S_{1,\rho_1(T)+2},e_n)\},
      &\mbox{if } \epsilon=1,\\
      T\cup
      \{(S_{\kappa_1(T)+1,1},e_n),(S_{\kappa_1(T)+2,1},e_n)\},
      &\mbox{if } \epsilon=-1.\\
    \end{cases}
    \]
  \item[\textbf{Case 2.}] Suppose $v<e_n$.  Then let
    $T'=\alpha(T(e_{n-1}),v,\epsilon)$, which is defined by the
    inductive hypothesis.  Then define
    \[
    \alpha(T,v,\epsilon)=T'\cup \{(S,e_n)\mid S\in A(T',P(T,e_n))\}.
    \]
  \end{itemize}
\end{definition}

\begin{lemma}
  The map $\alpha$ is well defined on all of $\mathcal{A}(M)$, and
  maps into $\mathcal{T}(M)$.
\end{lemma}

\begin{proof}
  This is a consequence of \cite[Proposition 1.3.4,
  Definition 1.2.5]{garfinkleclassI}.
\end{proof}

\begin{example}
  Let $M={\bf 7}$ and consider $T\in\mathcal{T}(M\setminus\{7\})$ as defined below
  \begin{center}
    \begin{tikzpicture}[node distance=0 cm,outer sep = 0pt]
      \tikzstyle{ver}=[rectangle, draw, thick,
      minimum width=.5cm, minimum height=1cm]
      \tikzstyle{hor}=[rectangle, draw, thick,
      minimum width=1cm, minimum height=.5cm]
      \node[ver] (1) at (   1,  1.5) {1};
      \node[ver] (2) [below  = of 1] {2};
      \node[hor] (3) at (1.75, 1.75) {3};
      \node[ver] (4) at ( 1.5,    1) {4};
      \node[hor] (5) [right  = of 3] {5};
      \node[ver] (6) [right  = of 4] {6};
      \node at (0,1) {$T=$};
      \node at (3.5,1) {.};
    \end{tikzpicture}
  \end{center}
  Then $(T,7,1),(T,7,-1)\in\mathcal{A}(M)$, and since $7=\max(M)$ we can calculate
  $\alpha(T,7,1)$ and $\alpha(T,7,-1)$ using case (1) of
  Definition~\ref{alphamap}.
  \begin{center}
    \begin{tabular}{cc}
      \begin{tikzpicture}[node distance=0 cm,outer sep = 0pt]
        \tikzstyle{ver}=[rectangle, draw, thick,
        minimum width=.5cm, minimum height=1cm]
        \tikzstyle{hor}=[rectangle, draw, thick,
        minimum width=1cm, minimum height=.5cm]
        \node[ver] (1) at (   1,  1.5) {1};
        \node[ver] (2) [below  = of 1] {2};
        \node[hor] (3) at (1.75, 1.75) {3};
        \node[ver] (4) at ( 1.5,    1) {4};
        \node[hor] (5) [right  = of 3] {5};
        \node[ver] (6) [right  = of 4] {6};
        \node[hor] (7) [right  = of 5] {7};
        \node at (-.5,1.5) {$\alpha(T,7,1)=$};
        \node at (4.5,1.5) {,};
      \end{tikzpicture}
      &
      \begin{tikzpicture}[node distance=0 cm,outer sep = 0pt]
        \tikzstyle{ver}=[rectangle, draw, thick,
        minimum width=.5cm, minimum height=1cm]
        \tikzstyle{hor}=[rectangle, draw, thick,
        minimum width=1cm, minimum height=.5cm]
        \node[ver] (1) at (   1,  1.5) {1};
        \node[ver] (2) [below  = of 1] {2};
        \node[hor] (3) at (1.75, 1.75) {3};
        \node[ver] (4) at ( 1.5,    1) {4};
        \node[hor] (5) [right  = of 3] {5};
        \node[ver] (6) [right  = of 4] {6};
        \node[ver] (7) [below  = of 2] {7};
        \node at (-.5,.5) {$\alpha(T,7,-1)=$};
        \node at (3.5,.5) {.};
      \end{tikzpicture}
    \end{tabular}
  \end{center}
  Note that we obtained the tableaux above by simply adding a domino
  with label $7$ to the first row or column of $T$.
\end{example}

\begin{example}
  Let $M={\bf 8}$ and consider $T\in\mathcal{T}(M\setminus\{7\})$, below
  \begin{center}
    \begin{tikzpicture}[node distance=0 cm,outer sep = 0pt]
      \tikzstyle{ver}=[rectangle, draw, thick, minimum width=.5cm,
      minimum height=1cm]         
      \tikzstyle{hor}=[rectangle, draw, thick, minimum width=1cm,
      minimum height=.5cm]
      \node[ver] (1) at (   1,  1.5) {1};
      \node[ver] (2) [below  = of 1] {2};
      \node[hor] (3) at (1.75, 1.75) {3};
      \node[ver] (4) at ( 1.5,    1) {4};
      \node[hor] (5) [right  = of 3] {5};
      \node[ver] (6) [right  = of 4] {6};
      \node[hor] (8) [right  = of 5] {8};
      \node at (0,1) {$T=$};
      \node at (4.5,1) {.};
    \end{tikzpicture}
  \end{center}
  As above, we see that $(T,7,1),(T,7,-1)\in\mathcal{A}(M)$. Now since
  $7< \max(M)$, we will use case (2) of Definition~\ref{alphamap} to
  find $\alpha(T,7,1)$. First we will remove all dominoes with labels
  strictly greater than 7 and add a horizontal domino with label 7 at
  the end of the first row.  Next, we shuffle in the removed dominoes.
  We first use case (1) of Definition~\ref{alphamap} to see that
  \begin{center}
    \begin{tikzpicture}[node distance=0 cm,outer sep = 0pt]
      \tikzstyle{ver}=[rectangle, draw, thick, minimum width=.5cm,
      minimum height=1cm]         
      \tikzstyle{hor}=[rectangle, draw, thick, minimum width=1cm,
      minimum height=.5cm]
      \node[ver] (1) at (   1,  1.5) {1};
      \node[ver] (2) [below  = of 1] {2};
      \node[hor] (3) at (1.75, 1.75) {3};
      \node[ver] (4) at ( 1.5,    1) {4};
      \node[hor] (5) [right  = of 3] {5};
      \node[ver] (6) [right  = of 4] {6};
      \node[hor] (7) [right  = of 5] {7};
      \node at (-1.25,1) {$T'=\alpha(T(6),7,1)=$};
      \node at (4.5,1) {.};
    \end{tikzpicture}
  \end{center}
  Then $P(T,8)=\{S_{1,6},S_{1,7}\}$, and
  $A(T',P(T,8))=\{S_{2,4},S_{2,5}\}$, so
  \begin{center}
    \begin{tikzpicture}[node distance=0 cm,outer sep = 0pt]
      \tikzstyle{ver}=[rectangle, draw, thick, minimum width=.5cm,
      minimum height=1cm]         
      \tikzstyle{hor}=[rectangle, draw, thick, minimum width=1cm,
      minimum height=.5cm]
      \node[ver] (1) at (   1,  1.5) {1};
      \node[ver] (2) [below  = of 1] {2};
      \node[hor] (3) at (1.75, 1.75) {3};
      \node[ver] (4) at ( 1.5,    1) {4};
      \node[hor] (5) [right  = of 3] {5};
      \node[ver] (6) [right  = of 4] {6};
      \node[hor] (7) [right  = of 5] {7};
      \node[hor] (8) [below  = of 5] {8};
      \node at (-.5,1) {$\alpha(T',7,1)=$};
      \node at (4.5,1) {.};
    \end{tikzpicture}
  \end{center}

\end{example}

\begin{definition}
  Let $M_1,M_2\subset\mathbb{N}$ with $|M_1|=|M_2|=m$.  Let
  $u=\max(M_2)$, and let $w\in\mathcal{W}(M_1,M_2)$.  Then there
  exists $v\in M_1$ and $\epsilon\in\{\pm 1\}$ such that
  $(v,u,\epsilon)\in w$, and we define
  \[
  w_{(m)}=w\setminus\{(v,u,\epsilon)\}.
  \]
\end{definition}

\begin{example}
  Let $w=\{(1,1,1),(2,4,-1),(3,3,1),(4,2,-1)\}$.  Then
  \[
  w_{(4)}=\{(1,1,1),(3,3,1),(4,2,-1)\}.
  \]
\end{example}

Now we are ready to assign a tableau, $T(w)$ to each element, $w\in
W=W(D_n)$. We will start by assuming that $T\left(w_{(n)}\right)$ is
defined by induction, and then use $\alpha$ to add a domino with label
$n$ to obtain $T(w)$.

\begin{definition}\label{dominotableaux}
  Let $M_1,M_2\subset\mathbb{N}$ be such that $|M_1|=|M_2|=m$.  We
  will define a map,
  \[
  \widehat{T}(M_1,M_2):\mathcal{W}(M_1,M_2) \rightarrow\mathcal{T}(M_1),
  \]
  by induction.  Suppose that $\widehat{T}(M'_1,M'_2)$ is defined when
  $|M'_1|=|M'_2|<m$.  Let $w\in \mathcal{W}(M_1,M_2)$ and let
  $(v,u,\epsilon)=w\setminus w_{(m)}$.  Then
  \[
  \widehat{T}(w)=
  \alpha\left(\widehat{T}\left(M_1\setminus\{v\},M_2\setminus\{u\}\right)
    \left(w_{(m)}\right),v,\epsilon\right).
  \]
  Let $w\in W(D_n)$, then define $T_L(w)=\widehat{T}({\bf n},{\bf n})(\delta(w))$.

  We further define $T_R(w)=T_L(w^{-1})\in\mathcal{T}(M_2)$.
\end{definition}

\begin{remark}\label{build}
  Given $w\in W(D_n)$ we can calculate $T_L(w)$ in the following way.
  Write
  \[
  \delta(w)=\{(w^{-1}(i),i,\epsilon_i)\mid i\in{\bf n}\}.
  \]
  Suppose that we have constructed a tableau, $T^{j-1}_L(w)$, with
  dominoes $w^{-1}(1), \dots ,w^{-1}(j-1)$.  We can obtain a new
  tableau that includes domino $w^{-1}(j)$ by setting
  $T^j_L(w)=\alpha\left(T^{j-1}_L(w),w^{-1}(j),\epsilon_j\right)$.  To
  construct $T^j_L(w)$, first write down $T^{j-1}_L(w)(w^{-1}(i))$.
  Now add $w^{-1}(i)$ as a vertical domino at the end of the first
  column if $\epsilon_i=-1$ and a horizontal domino at the end of the
  first row if $\epsilon_i=1$.  Finally, shuffle in the remaining
  $w^{-1}(k)$-labeled dominoes in increasing order in the
  $A\left(T^{j-1}_L(w)(k-1), P\left(T^{j-1},w^{-1}(k)\right)\right)$
  position.
\end{remark}

\begin{example}\label{buildT}
  Let $w=w_6$.  Then as a signed permutation we have
  \[
  w=\left(\underline{1},\underline{6},3,\underline{4},
    5,\underline{2}\right),
  \]
  so
  \[
  \delta(w)=\{(1,1,-1),(6,2,-1),(3,3,1),(4,4,-1),(5,5,1),(2,6,-1)\}.
  \]
  We construct $T_L(w)$ using the process outlined in
  Remark~\ref{build}.

  Since $w^{-1}(1)=-1$, we begin by adding a vertical domino with
  label 1 to obtain $T^1_L(w)$.  Then since $w^{-1}(2)=-6$ and since 6
  is larger than all labels in $T^1_L(w)$ we add a vertical domino
  with label 6 at the end of the first column to obtain $T^2_L(w)$.

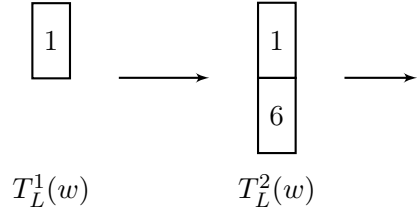
  \begin{center}
    \begin{tikzpicture}[node distance=3cm,>=latex']
      \node (a)
      {
        \begin{tikzpicture}[node distance=0 cm,outer sep = 0pt]
          \node[smver] (1) at (   1,  1.5)   {\rm{1}};
          \node[rectangle, thick, minimum width=.5cm,
          minimum height=1cm] [below = of 1] {}; 
        \end{tikzpicture}
      };
      \node[below of=a,node distance=1.5cm] {$T^1_L(w)$};
      \node[right of=a] (b)
      {
        \begin{tikzpicture}[node distance=0 cm,outer sep = 0pt]
          \node[smver] (1) at (   1,  1.5)   {\rm{1}};
          \node[smver] (2) [below  = of 1]   {\rm{6}};
        \end{tikzpicture}
      };
      \node[below of=b,node distance=1.5cm] {$T^2_L(w)$};
      \node[right of=b, node distance=2cm] (c) {};
      \draw[shorten >=0.5cm,shorten <=0.5cm,->,thick] (a)--(b);
      \draw[shorten <=0.5cm,->,thick] (b)--(c);
    \end{tikzpicture}
  \end{center}
  
  Now $w^{-1}(3)=3$ so we must next add a horizontal domino with label
  3.  However, we must first remove all dominoes with labels larger
  than 3.  Once we have done this we are able to add a horizontal
  domino with label 3 to the end of the first row.  We place the
  domino with label 6 back in its original position because there is
  no overlap.
  
  
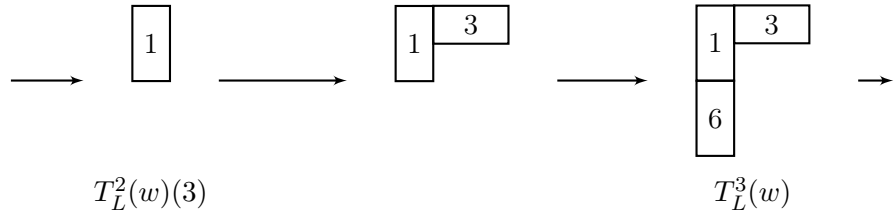
\begin{center}
    \begin{tikzpicture}[node distance=4cm,>=latex']
      \node (a){};
      \node[right of=a, node distance=2cm] (b)
      {
        \begin{tikzpicture}[node distance=0 cm,outer sep = 0pt]
          \node[smver] (1) at (   1,  1.5)   {\rm{1}};
          \node[rectangle, thick, minimum width=.5cm,
          minimum height=1cm] [below = of 1] {}; 
        \end{tikzpicture}
      };
      \node[below of=b,node distance=1.5cm] {$T^2_L(w)(3)$};
      \node[right of=b] (c)
      {
        \begin{tikzpicture}[node distance=0 cm,outer sep = 0pt]
          \node[smver] (1) at (   1,  1.5)   {\rm{1}};
          \node[smhor] (3) at (1.75, 1.75)   {\rm{3}};
          \node[rectangle, thick, minimum width=.5cm,
          minimum height=1cm] [below = of 1] {}; 
        \end{tikzpicture}
      };
      \node[right of=c] (d)
      {
        \begin{tikzpicture}[node distance=0 cm,outer sep = 0pt]
          \node[smver] (1) at (   1,  1.5)   {\rm{1}};
          \node[smver] (2) [below  = of 1]   {\rm{6}};
          \node[smhor] (3) at (1.75, 1.75)   {\rm{3}};
        \end{tikzpicture}
      };
      \node[below of=d,node distance=1.5cm] {$T^3_L(w)$};
      \node[right of=d, node distance=2cm] (e) {};
      \draw[shorten >=0.5cm,->,thick] (a)--(b);
      \draw[shorten >=0.5cm,shorten <=0.5cm,->,thick] (b)--(c);
      \draw[shorten >=0.5cm,shorten <=0.5cm,->,thick] (c)--(d);
      \draw[shorten <=0.5cm,->,thick] (d)--(e);
    \end{tikzpicture}
  \end{center}
  
  Next, we see that $w^{-1}(4)=-4$, so we will add a vertical domino
  with label 4.  As before, we must first remove all dominoes with
  label greater than 4, and then we are free to add a vertical
  domino with label 4 to the end of the first column.  However, when
  we try to replace the domino with label 6 it now overlaps with the
  domino with label 4.  As a result, we must use the $A$ map from
  Definition~\ref{shuffle} to add a domino with label 6. We see that
  since the domino with label 6 overlaps completely with the tableau,
  we are in case (2) of the definition, so the $A$ map has the effect
  of bumping the domino with label 6 to the right.
  
  
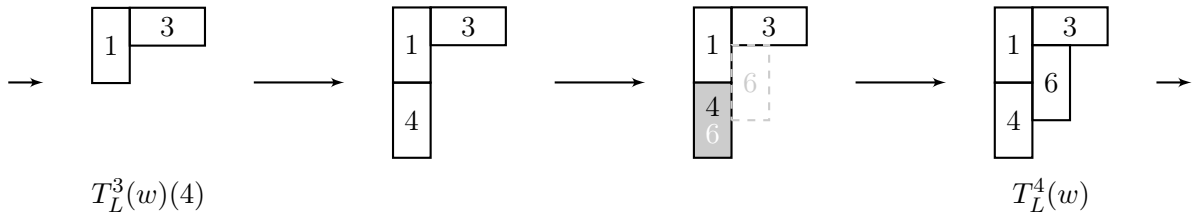
\begin{center}
    \begin{tikzpicture}[node distance=4cm,>=latex']
      \node (a) {};
      \node[right of=a, node distance=2cm] (b)
      {
        \begin{tikzpicture}[node distance=0 cm,outer sep = 0pt]
          \node[smver] (1) at (   1,  1.5)   {\rm{1}};
          \node[smhor] (3) at (1.75, 1.75)   {\rm{3}};
          \node[rectangle, thick, minimum width=.5cm,
          minimum height=1cm] [below = of 1] {}; 
        \end{tikzpicture}
      };
      \node[below of=b,node distance=1.5cm] {$T^3_L(w)(4)$};
      \node[right of=b] (c)
      {
        \begin{tikzpicture}[node distance=0 cm,outer sep = 0pt]
          \node[smver] (1) at (   1,  1.5)   {\rm{1}};
          \node[smver] (2) [below  = of 1]   {\rm{4}};
          \node[smhor] (3) at (1.75, 1.75)   {\rm{3}};
        \end{tikzpicture}
      };
      \node[right of=c] (d)
      {
        \begin{tikzpicture}[node distance=0 cm,outer sep = 0pt]
          \node[smver] (1) at (   1,  1.5)   {\rm{1}};
          \node[shver] (2) [below  = of 1]   {$\shortstack{4\\ \color{white}6}$};
          \node[smhor] (3) at (1.75, 1.75)   {\rm{3}};
          \node[smdver] (4) at ( 1.5,    1)   {\rm{6}};
        \end{tikzpicture}
      };
      \node[right of=d] (e)
      {
        \begin{tikzpicture}[node distance=0 cm,outer sep = 0pt]
          \node[smver] (1) at (   1,  1.5)   {\rm{1}};
          \node[smver] (2) [below  = of 1]   {\rm{4}};
          \node[smhor] (3) at (1.75, 1.75)   {\rm{3}};
          \node[smver] (4) at ( 1.5,    1)   {\rm{6}};
        \end{tikzpicture}
      };
      \node[below of=e,node distance=1.5cm] {$T^4_L(w)$};
      \node[right of=e, node distance=2cm] (f) {};
      \draw[shorten >=0.5cm,->,thick] (a)--(b);
      \draw[shorten >=0.5cm,shorten <=0.5cm,->,thick] (b)--(c);
      \draw[shorten >=0.5cm,shorten <=0.5cm,->,thick] (c)--(d);
      \draw[shorten >=0.5cm,shorten <=0.5cm,->,thick] (d)--(e);
      \draw[shorten <=0.5cm,->,thick] (e)--(f);
    \end{tikzpicture}
  \end{center}
  
  Now $w^{-1}(5)=5$, so we will next add a horizontal domino with
  label 5.  We remove all dominoes with label greater than 5, then add
  a domino with label 5 at the end of the first row.  We are then able
  to replace the remaining dominoes without overlap.
  
  
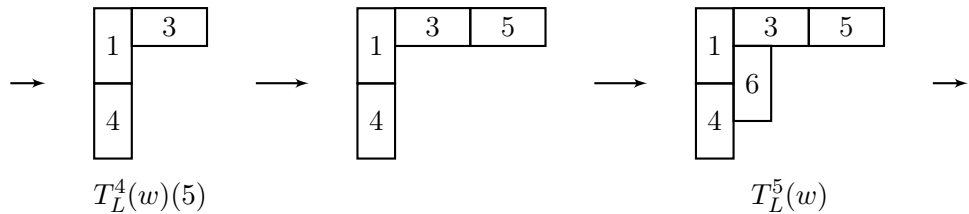
\begin{center}
    \begin{tikzpicture}[node distance=4cm,>=latex']
      \node (a) {};
      \node[right of=a, node distance=2cm] (b)
      {
        \begin{tikzpicture}[node distance=0 cm,outer sep = 0pt]
          \node[smver] (1) at (   1,  1.5)   {\rm{1}};
          \node[smver] (2) [below  = of 1]   {\rm{4}};
          \node[smhor] (3) at (1.75, 1.75)   {\rm{3}};
        \end{tikzpicture}
      };
      \node[below of=b,node distance=1.5cm] {$T^4_L(w)(5)$};
      \node[right of=b] (c)
      {
        \begin{tikzpicture}[node distance=0 cm,outer sep = 0pt]
          \node[smver] (1) at (   1,  1.5)   {\rm{1}};
          \node[smver] (2) [below  = of 1]   {\rm{4}};
          \node[smhor] (3) at (1.75, 1.75)   {\rm{3}};
          \node[smhor] (5) [right  = of 3]   {\rm{5}};
        \end{tikzpicture}
        
      };
      \node[right of=c, node distance=4.5cm] (d)
      {
        \begin{tikzpicture}[node distance=0 cm,outer sep = 0pt]
          \node[smver] (1) at (   1,  1.5)   {\rm{1}};
          \node[smver] (2) [below  = of 1]   {\rm{4}};
          \node[smhor] (3) at (1.75, 1.75)   {\rm{3}};
          \node[smver] (4) at ( 1.5,    1)   {\rm{6}};
          \node[smhor] (5) [right  = of 3]   {\rm{5}};
        \end{tikzpicture}
      };
      \node[below of=d,node distance=1.5cm] {$T^5_L(w)$};
      \node[right of=d, node distance=2.5cm] (e) {};
      \draw[shorten >=0.5cm,->,thick] (a)--(b);
      \draw[shorten >=0.5cm,shorten <=0.5cm,->,thick] (b)--(c);
      \draw[shorten >=0.5cm,shorten <=0.5cm,->,thick] (c)--(d);
      \draw[shorten <=0.5cm,->,thick] (d)--(e);
    \end{tikzpicture}
  \end{center}
  
  Finally, $w^{-1}(6)=-2$, so the last domino that we will add is a
  vertical domino with label 2.  We remove all dominoes with label
  greater than 2 and add a domino with label 2 at the end of the first
  column.  Now we must add in the dominoes that we removed.  The domino
  with label 3 does not overlap with any other dominoes, so we may add
  it in its former position.  The former position of the domino with
  label 4 is now fully occupied by another domino, so we must use the
  $A$ map from Definition~\ref{shuffle}.  As above, we are in case (2)
  of the definition, so the domino with label 4 gets bumped to the
  right.
  
  
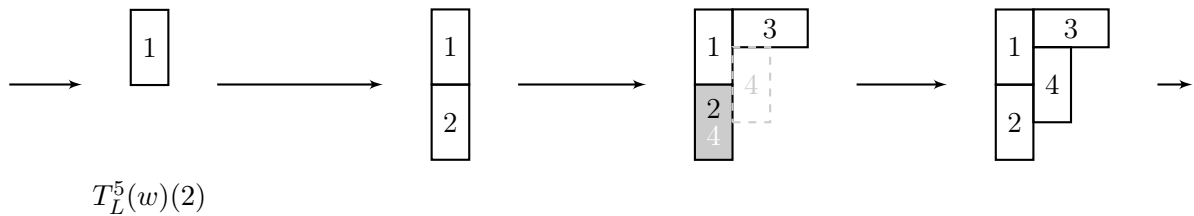
\begin{center}
    \begin{tikzpicture}[node distance=4cm,>=latex']
      \node (a) {};
      \node[right of=a, node distance=2cm] (b)
      {
        \begin{tikzpicture}[node distance=0 cm,outer sep = 0pt]
          \node[smver] (1) at (   1,  1.5)   {\rm{1}};
          \node[rectangle, thick, minimum width=.5cm,
          minimum height=1cm] [below = of 1] {};  
        \end{tikzpicture}
      };
      \node[below of=b,node distance=1.5cm] {$T^5_L(w)(2)$};
      \node[right of=b] (c)
      {
        \begin{tikzpicture}[node distance=0 cm,outer sep = 0pt]
          \node[smver] (1) at (   1,  1.5)   {\rm{1}};
          \node[smver] (2) [below  = of 1]   {\rm{2}};
        \end{tikzpicture}
      };
      \node[right of=c] (d)
      {
        \begin{tikzpicture}[node distance=0 cm,outer sep = 0pt]
          \node[smver] (1) at (   1,  1.5)   {\rm{1}};
          \node[shver] (2) [below  = of 1]   {$\shortstack{2\\ \color{white}4}$};
          \node[smhor] (3) at (1.75, 1.75)   {\rm{3}};
          \node[smdver] (4) at ( 1.5,    1)   {\rm{4}};
        \end{tikzpicture}
      };
      \node[right of=d] (e)
      {
        \begin{tikzpicture}[node distance=0 cm,outer sep = 0pt]
          \node[smver] (1) at (   1,  1.5)   {\rm{1}};
          \node[smver] (2) [below  = of 1]   {\rm{2}};
          \node[smhor] (3) at (1.75, 1.75)   {\rm{3}};
          \node[smver] (4) at ( 1.5,    1)   {\rm{4}};
        \end{tikzpicture}
      };
      \node[right of=e, node distance=2cm] (f) {};
      \draw[shorten >=0.5cm,->,thick] (a)--(b);
      \draw[shorten >=0.5cm,shorten <=0.5cm,->,thick] (b)--(c);
      \draw[shorten >=0.5cm,shorten <=0.5cm,->,thick] (c)--(d);
      \draw[shorten >=0.5cm,shorten <=0.5cm,->,thick] (d)--(e);
      \draw[shorten <=0.5cm,->,thick] (e)--(f);
    \end{tikzpicture}
  \end{center}
  
  When we try to replace the domino with label 6 we see that its
  former position is fully occupied by the domino with label 4.  We
  again use the $A$ map from Definition~\ref{shuffle}, which has the
  effect of bumping the domino with label 6 to the right.
  
  
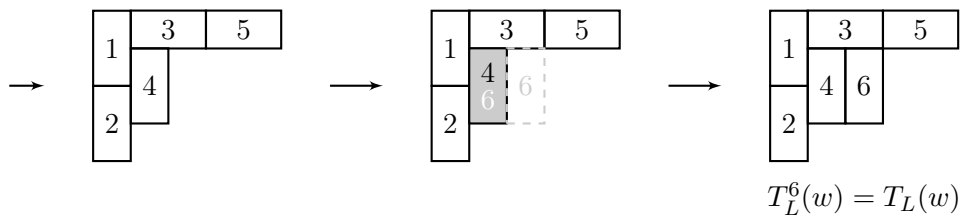
\begin{center}
    \begin{tikzpicture}[node distance=4cm,>=latex']
      \node (a) {};
      \node[right of=a, node distance=2.5cm] (b)
      {
        \begin{tikzpicture}[node distance=0 cm,outer sep = 0pt]
          \node[smver] (1) at (   1,  1.5)   {\rm{1}};
          \node[smver] (2) [below  = of 1]   {\rm{2}};
          \node[smhor] (3) at (1.75, 1.75)   {\rm{3}};
          \node[smver] (4) at ( 1.5,    1)   {\rm{4}};
          \node[smhor] (5) [right  = of 3]   {\rm{5}};
        \end{tikzpicture}
      };
      \node[right of=b, node distance=4.5cm] (c)
      {
        \begin{tikzpicture}[node distance=0 cm,outer sep = 0pt]
          \node[smver] (1) at (   1,  1.5)   {\rm{1}};
          \node[smver] (2) [below  = of 1]   {\rm{2}};
          \node[smhor] (3) at (1.75, 1.75)   {\rm{3}};
          \node[shver] (4) at ( 1.5,    1)  {$\shortstack{4\\ \color{white}6}$};
          \node[smhor] (5) [right  = of 3]   {\rm{5}};
          \node[smdver] (6) [right  = of 4]   {\color{shade}6};
        \end{tikzpicture}
      };
      \node[right of=c, node distance=4.5cm] (d)
      {
        \begin{tikzpicture}[node distance=0 cm,outer sep = 0pt]
          \node[smver] (1) at (   1,  1.5)   {\rm{1}};
          \node[smver] (2) [below  = of 1]   {\rm{2}};
          \node[smhor] (3) at (1.75, 1.75)   {\rm{3}};
          \node[smver] (4) at ( 1.5,    1)   {\rm{4}};
          \node[smhor] (5) [right  = of 3]   {\rm{5}};
          \node[smver] (6) [right  = of 4]   {\rm{6}};
        \end{tikzpicture}
      };
      \node[below of=d,node distance=1.5cm] {$T^6_L(w)=T_L(w)$};
      \draw[shorten >=0.5cm,->,thick] (a)--(b);
      \draw[shorten >=0.5cm,shorten <=0.5cm,->,thick] (b)--(c);
      \draw[shorten >=0.5cm,shorten <=0.5cm,->,thick] (c)--(d);
    \end{tikzpicture}
  \end{center}

  At this point we have added all dominoes to the tableau, so
  $T^6_L(w)=T_L(w)$.

\end{example}

\begin{proposition}
  Let $W=W(D_n)$. Then the map
  \[
  W\rightarrow \mathcal{T}\left({\bf n}\right) \times
  \mathcal{T}\left({\bf n}\right) :w\mapsto (T_L(w),T_R(w))
  \]
  is an injection.
\end{proposition}

\begin{proof}
  This is proved in \cite[Theorem 1.2.13]{garfinkleclassI}.
\end{proof}

\section{Cycles of tableaux}\label{domcyc}

The partition of $W=W(D_n)$ into sets with the same left tableau is
finer than the partition of $W$ into left
cells~\cite{mcgovern1996left}. However, following the work of
Garfinkle in \cite{garfinkleclassI}, we can use the notion of cycles
to define an equivalence relation on tableaux that corresponds to the
partition of $W$ into left cells. We begin with some preliminary
definitions.

\begin{definition}
  Let $S_{i,j}\in\mathcal{F}$.  If $i+j$ is even then we say that the
  square $S_{i,j}$ is \textbf{fixed}. If $S_{i,j}\in\mathcal{F}$ is
  not fixed then we say that $S_{i,j}$ is \textbf{variable}.
\end{definition}

\begin{example}
  Let $T$ be as in Example~\ref{domino}.  Then the fixed squares of
  $T$ are those that are shaded below
  \begin{center}
    \begin{tikzpicture}[node distance=0 cm,outer sep = 0pt]
      \definecolor{shade}{rgb}{0.8,0.8,0.8}
      \tikzstyle{ver}=[rectangle, draw, thick,
      minimum width=.5cm, minimum height=1cm]
      \tikzstyle{hor}=[rectangle, draw, thick,
      minimum width=1cm, minimum height=.5cm]
      \tikzstyle{fix}=[rectangle, fill=shade, minimum width=.5cm,
      minimum height=.5cm]
      \node[fix] (a) at (   1, 1.75) {};
      \node[fix] (b) at (   1,  .75) {};
      \node[fix] (c) at (   2, 1.75) {};
      \node[fix] (d) at (   2,  .75) {};
      \node[fix] (e) at ( 1.5, 1.25) {};
      \node[ver] (1) at (   1,  1.5) {1};
      \node[ver] (2) [below  = of 1] {2};
      \node[hor] (3) at (1.75, 1.75) {3};
      \node[ver] (4) at ( 1.5,    1) {4};
      \node[ver] (5) [right  = of 4] {5};
      \node at (0,1) {$T=$};
      \node at (2.5,1) {.};
    \end{tikzpicture}
  \end{center}
  
\end{example}

It is easy to see that if $T\in\mathcal{T}(M)$ and $k\in M$ then
$P(T,k)$ contains exactly one fixed square.  We will now introduce a
way to move dominoes within a tableau in such a way that the fixed
squares are not affected.

\begin{definition}\label{Pprime}
  Let $M\subset\mathbb{N}$ be finite and let $T\in\mathcal{T}(M)$.
  Pick $k\in M$.  Let $S_{i,j}$ be the fixed square in $P(T,k)$, and
  find $l,m\in\mathbb{N}$ such that $P(T,k)=\{S_{i,j},S_{l,m}\}$.  Let
  \[
  r=
  \begin{cases}
    N(T,S_{i-1,j+1})&\mbox{if } l>i\mbox{ or }m<j;\\
    N(T,S_{i+1,j-1})&\mbox{if } l<i\mbox{ or }m>j.\\
  \end{cases}
  \]
  Then we define a new domino, $P'(T,k)$, as follows.  If $l>i$ or
  $m<j$ then
  \[
  P'(T,k)=
  \begin{cases}
    \{S_{i,j},S_{i-1,j}\}&\mbox{if } r>k;\\
    \{S_{i,j},S_{i,j+1}\}&\mbox{if } r<k.\\
  \end{cases}
  \]
  If $l<i$ or $m>j$ then
  \[
  P'(T,k)=
  \begin{cases}
    \{S_{i,j},S_{i+1,j}\}&\mbox{if } r<k;\\
    \{S_{i,j},S_{i,j-1}\}&\mbox{if } r>k.\\
  \end{cases}
  \]
  Note that $P(T,k)\cap P'(T,k)=S_{i,j}$ is a fixed square.  Define
  $D'(T,k)=\{(S,k)\mid S\in P'(T,k)\}$.
\end{definition}

\begin{remark}\label{computer}
  Let $M\subset\mathbb{N}$ be finite, let $k\in M$, let
  $T\in\mathcal{T}(M)$, and let $r$ be defined as above.  Then
  the following table summarizes the relationship between $P(T,k)$ and
  $P'(T,k)$. The shaded squares in the table below correspond to fixed
  squares.

  \begin{center}
    \begin{tabular}{|c|c|c|c|}
      \hline
      \multirow{2}{*}{Position} &\multirow{2}{*}{$P(T,k)$} & \multicolumn{2}{|c|}{$P'(T,k)$}\\
      \cline{3-4}
      && $r>k$ & $r<k$ \\
      \hline\hline
      \multirow{2}{*}
      {
          $l>i$ or $m<j$
      }
      &
      \begin{tikzpicture}[node distance=0 cm,outer sep = 0pt]
        \definecolor{shade}{rgb}{0.8,0.8,0.8}
        \tikzstyle{ver}=[rectangle, draw, thick, minimum width=.5cm,
        minimum height=1cm]         
        \tikzstyle{hor}=[rectangle, draw, thick, minimum width=1cm,
        minimum height=.5cm]
        \tikzstyle{fix}=[rectangle, fill=shade, minimum width=.5cm,
        minimum height=.5cm]
        \tikzstyle{r}  =[rectangle, draw, dashed, minimum width=.5cm,
        minimum height=.5cm]
        \node[r]   (r) at (  1, 1.5) {$r$};
        \node[fix] (d) at ( .5,   1)    {};
        \node[ver] (2) at ( .5, .75)    {};
        \node at (1,2) {}; 
      \end{tikzpicture}
      &
      \multirow{2}{*}
      {
        \begin{tikzpicture}[node distance=0 cm,outer sep = 0pt]
          \definecolor{shade}{rgb}{0.8,0.8,0.8}
          \tikzstyle{ver}=[rectangle, draw, thick, minimum width=.5cm,
          minimum height=1cm]         
          \tikzstyle{hor}=[rectangle, draw, thick, minimum width=1cm,
          minimum height=.5cm]
          \tikzstyle{fix}=[rectangle, fill=shade, minimum width=.5cm,
          minimum height=.5cm]
          \tikzstyle{r}  =[rectangle, draw, dashed, minimum width=.5cm,
          minimum height=.5cm]
          \node[fix] (d) at ( .5,  .5)    {};
          \node[ver] (2) at ( .5, .75)    {};
        \end{tikzpicture}
      }
      &
      \multirow{2}{*}
      {
        \begin{tikzpicture}[node distance=0 cm,outer sep = 0pt]
          \definecolor{shade}{rgb}{0.8,0.8,0.8}
          \tikzstyle{ver}=[rectangle, draw, thick, minimum width=.5cm,
          minimum height=1cm]         
          \tikzstyle{hor}=[rectangle, draw, thick, minimum width=1cm,
          minimum height=.5cm]
          \tikzstyle{fix}=[rectangle, fill=shade, minimum width=.5cm,
          minimum height=.5cm]
          \tikzstyle{r}  =[rectangle, draw, dashed, minimum width=.5cm,
          minimum height=.5cm]
          \node[fix] (d) at (   0,   1)    {};
          \node[hor] (2) at ( .25,   1)    {};
      \end{tikzpicture}
      }
      \\ \cline{2-2} &
      \begin{tikzpicture}[node distance=0 cm,outer sep = 0pt]
        \definecolor{shade}{rgb}{0.8,0.8,0.8}
        \tikzstyle{ver}=[rectangle, draw, thick, minimum width=.5cm,
        minimum height=1cm]         
        \tikzstyle{hor}=[rectangle, draw, thick, minimum width=1cm,
        minimum height=.5cm]
        \tikzstyle{fix}=[rectangle, fill=shade, minimum width=.5cm,
        minimum height=.5cm]
        \tikzstyle{r}  =[rectangle, draw, dashed, minimum width=.5cm,
        minimum height=.5cm]
        \node[r]   (r) at (   1, 1.5) {$r$};
        \node[fix] (d) at (  .5,   1)    {};
        \node[hor] (2) at ( .25,   1)    {};
        \node at (1,2) {}; 
      \end{tikzpicture}
      & & \\
      \hline
      \multirow{2}{*}
      {
          $l<i$ or $m>j$
      }
      &
      \begin{tikzpicture}[node distance=0 cm,outer sep = 0pt]
        \definecolor{shade}{rgb}{0.8,0.8,0.8}
        \tikzstyle{ver}=[rectangle, draw, thick, minimum width=.5cm,
        minimum height=1cm]         
        \tikzstyle{hor}=[rectangle, draw, thick, minimum width=1cm,
        minimum height=.5cm]
        \tikzstyle{fix}=[rectangle, fill=shade, minimum width=.5cm,
        minimum height=.5cm]
        \tikzstyle{r}  =[rectangle, draw, dashed, minimum width=.5cm,
        minimum height=.5cm]
        \node[r]   (r) at ( .25,   .5) {$r$};
        \node[fix] (d) at ( .75,    1)    {};
        \node[ver] (2) at ( .75, 1.25)    {};
        \node at (1,2) {}; 
      \end{tikzpicture}
      &
      \multirow{2}{*}
      {
        \begin{tikzpicture}[node distance=0 cm,outer sep = 0pt]
          \definecolor{shade}{rgb}{0.8,0.8,0.8}
          \tikzstyle{ver}=[rectangle, draw, thick, minimum width=.5cm,
          minimum height=1cm]         
          \tikzstyle{hor}=[rectangle, draw, thick, minimum width=1cm,
          minimum height=.5cm]
          \tikzstyle{fix}=[rectangle, fill=shade, minimum width=.5cm,
          minimum height=.5cm]
          \tikzstyle{r}  =[rectangle, draw, dashed, minimum width=.5cm,
          minimum height=.5cm]
          \node[fix] (d) at (  .5,   1)    {};
          \node[hor] (2) at ( .25,   1)    {};
      \end{tikzpicture}
      }
      &
      \multirow{2}{*}
      {
        \begin{tikzpicture}[node distance=0 cm,outer sep = 0pt]
          \definecolor{shade}{rgb}{0.8,0.8,0.8}
          \tikzstyle{ver}=[rectangle, draw, thick, minimum width=.5cm,
          minimum height=1cm]         
          \tikzstyle{hor}=[rectangle, draw, thick, minimum width=1cm,
          minimum height=.5cm]
          \tikzstyle{fix}=[rectangle, fill=shade, minimum width=.5cm,
          minimum height=.5cm]
          \tikzstyle{r}  =[rectangle, draw, dashed, minimum width=.5cm,
          minimum height=.5cm]
          \node[fix] (d) at ( .5,   1)    {};
          \node[ver] (2) at ( .5, .75)    {};
        \end{tikzpicture}
      }
      \\ \cline{2-2} &
      \begin{tikzpicture}[node distance=0 cm,outer sep = 0pt]
        \definecolor{shade}{rgb}{0.8,0.8,0.8}
        \tikzstyle{ver}=[rectangle, draw, thick, minimum width=.5cm,
        minimum height=1cm]         
        \tikzstyle{hor}=[rectangle, draw, thick, minimum width=1cm,
        minimum height=.5cm]
        \tikzstyle{fix}=[rectangle, fill=shade, minimum width=.5cm,
        minimum height=.5cm]
        \tikzstyle{r}  =[rectangle, draw, dashed, minimum width=.5cm,
        minimum height=.5cm]
        \node[r]   (r) at ( .25, .5) {$r$};
        \node[fix] (d) at ( .75,  1)    {};
        \node[hor] (2) at (   1,  1)    {};
        \node at (1,1.5) {}; 
      \end{tikzpicture}
      & & \\
      \hline
    \end{tabular}
  \end{center}
\end{remark}

\begin{example}\label{primes}
  Let $T$ be defined as in Example~\ref{buildT}.  The shaded squares
  in the diagram below correspond to fixed squares.
  
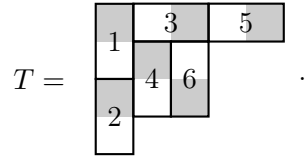
\begin{center}
    \begin{tikzpicture}[node distance=0 cm,outer sep = 0pt]
      \definecolor{shade}{rgb}{0.8,0.8,0.8}
      \tikzstyle{ver}=[rectangle, draw, thick, minimum width=.5cm,
      minimum height=1cm]         
      \tikzstyle{hor}=[rectangle, draw, thick, minimum width=1cm,
      minimum height=.5cm]
      \node[fix] (a) at (   1, 1.75) {};
      \node[fix] (b) at (   1,  .75) {};
      \node[fix] (c) at (   2, 1.75) {};
      \node[fix] (d) at (   2,  .75) {};
      \node[fix] (e) at ( 1.5, 1.25) {};
      \node[fix] (f) at (   3, 1.75) {};
      \node[ver] (1) at (   1,  1.5)   {\rm{1}};
      \node[ver] (2) [below  = of 1]   {\rm{2}};
      \node[hor] (3) at (1.75, 1.75)   {\rm{3}};
      \node[ver] (4) at ( 1.5,    1)   {\rm{4}};
      \node[hor] (5) [right  = of 3]   {\rm{5}};
      \node[ver] (6) [right  = of 4]   {\rm{6}};
      \node at (0,1) {$T=$};
      \node at (3.5,1) {.};
    \end{tikzpicture}
  \end{center}

  The squares containing the $r$-values associated to each $k\in{\bf
    6}$ are outlined below,
  
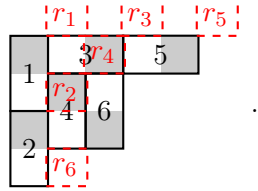
\begin{center}
    \begin{tikzpicture}[node distance=0 cm,outer sep = 0pt]
      \node[fix] (a) at (   1, 1.75) {};
      \node[fix] (b) at (   1,  .75) {};
      \node[fix] (c) at (   2, 1.75) {};
      \node[fix] (d) at (   2,  .75) {};
      \node[fix] (e) at ( 1.5, 1.25) {};
      \node[fix] (f) at (   3, 1.75) {};
      \node[smver] (1) at (   1,  1.5)   {\rm{1}};
      \node[smver] (2) [below  = of 1]   {\rm{2}};
      \node[smhor] (3) at (1.75, 1.75)   {\rm{3}};
      \node[smver] (4) at ( 1.5,    1)   {\rm{4}};
      \node[smhor] (5) [right  = of 3]   {\rm{5}};
      \node[smver] (6) [right  = of 4]   {\rm{6}};
      \node[redr] at (1.5, 2.25) {\hspace{-.02in}$r_1$};
      \node[redr] at (2.5, 2.25) {\hspace{-.02in}$r_3$};
      \node[redr] at (3.5, 2.25) {\hspace{-.02in}$r_5$};
      \node[redr] at (1.5, 1.25) {\hspace{-.02in}$r_2$};
      \node[redr] at (  2, 1.75) {\hspace{-.02in}$r_4$};
      \node[redr] at (1.5, 0.25) {\hspace{-.02in}$r_6$};
      \node at (4,1) {.};
    \end{tikzpicture}
  \end{center}
  
  We list the $r$-values in the following table:
  \begin{center}
    \begin{tabular}{|c||c|c|c|c|c|c|}
      \hline
      $k$ & 1 & 2 & 3 & 4 & 5 & 6 \\
      \hline
      $r_k$ & 0 & 4 & 0 & 3 & 0 & $\infty$ \\
      \hline
    \end{tabular}
  \end{center}

  Now we can use these values to help calculate $P'(T,k)$ for each
  $k\in {\bf 6}$.  For example, if we consider $P(T,1)$ we see that
  $r=0$, so $r<k$ and thus $P'(T,1)=\{S_{1,1},S_{1,2}\}$. Similarly,
  considering $P(T,6)$ we see that $r=\infty$, so $r>k$, and thus
  $P'(T,6)=\{S_{3,3},S_{3,2}\}$.  We can repeat this process to create
  a new tableau
  
  \begin{center}
    \begin{tikzpicture}[node distance=0 cm,outer sep = 0pt]
      \definecolor{shade}{rgb}{0.8,0.8,0.8}
      \tikzstyle{ver}=[rectangle, draw, thick, minimum width=.5cm,
      minimum height=1cm]         
      \tikzstyle{hor}=[rectangle, draw, thick, minimum width=1cm,
      minimum height=.5cm]
      \tikzstyle{fix}=[rectangle, fill=shade, minimum width=.5cm,
      minimum height=.5cm]
      \node[fix] (a) at ( 1.5, 1.75) {};
      \node[fix] (b) at ( 1.5,  .75) {};
      \node[fix] (c) at ( 2.5, 1.75) {};
      \node[fix] (d) at ( 2.5,  .75) {};
      \node[fix] (e) at (   2, 1.25) {};
      \node[fix] (f) at ( 3.5, 1.75) {};
      \node[hor] (1) at (1.75,1.75)    {\rm{1}};
      \node[ver] (2) at ( 1.5,   1)    {\rm{2}};
      \node[hor] (3) at (2.75,1.75)    {\rm{3}};
      \node[hor] (4) at (2.25,1.25)    {\rm{4}};
      \node[hor] (5) [right  = of 3]   {\rm{5}};
      \node[hor] (6) [below  = of 4]   {\rm{6}};
      \node at (-1.25,1.25) {$T'=\{D'(T,k)\mid k\in {\bf 6}\}=$};
      \node at (4.5,1.25) {.};
    \end{tikzpicture}
  \end{center}
  Remarkably, by \cite[Proposition 1.5.27]{garfinkleclassI} if
  $T\in\mathcal{T}(M)$ then we are guaranteed
  $T'\in\mathcal{T}(M)$. Observe that each fixed square has the same
  label in both $T$ and $T'$.
\end{example}

We can now  these new $P'$ dominoes to define an equivalence
relation, $\sim$, on $M$.

\begin{definition}
  Let $M\subset \mathbb{N}$ be finite, let $T\in\mathcal{T}(M)$, and
  let $a,b\in M$.  Then $\sim$ is the equivalence relation generated
  by $a\sim b$ if $P(T,a)\cap P'(T,b)$ is nonempty.

  The equivalence relation $\sim$ partitions $M$ into sets called
  \textbf{cycles}.  We call a cycle $C$ \textbf{closed} if
  $N(T,P(T,k)\setminus S_{i,j})\in C$ for all $k\in C$, where
  $S_{i,j}$ is the fixed square of $P(T,k)$. If a cycle is not closed
  we call the cycle \textbf{open}.
\end{definition}

\begin{example}
  Let $T$ be as in Example~\ref{primes}.  Then since each of
  $P'(T,2)\cap P(T,1)$, $P'(T,1)\cap P(T,3)$, and $P'(T,3)\cap P(T,5)$
  is nonempty, we know that $C_1=\{1,2,3,5\}$ is a cycle.  However
  since the variable square in $P'(T,5)$ is not in $T$, we have
  $N(T,P'(T,5)\setminus P_{1,6})=\infty\not\in C_1$, so $C_1$ is an
  open cycle.

  Similarly, $P(T,4)\cap P'(T,6)$ is nonempty, so $C_2=\{4,6\}$ is a
  cycle.  In this case, the variable squares in $P'(T,4)$ and
  $P'(T,6)$ both overlap with $T$. We have $N(T,P'(T,4)\setminus
  P_{2,3})=6\in C_2$ and $N(T,P'(T,6)\setminus P_{3,2})=4\in C_2$, so
  $C_2$ is a closed cycle.
\end{example}

\begin{definition}
  Let $T\in\mathcal{T}(M)$ be a domino tableau, and let $C$ be a
  cycle.  Define
  \[
  E(T,C)=\left(T\setminus \{D(T,k)\mid k\in C\}\right)\cup\{D'(T,k)\mid k\in C\}.
  \]

  Moreover, if $C_1,C_2,\dots,C_n\subset M$ are cycles, then we define
  \[
  E(T,C_1,C_2,\dots,C_n)=E(\cdots E(E(T,C_1),C_2),\cdots,C_n).
  \]
\end{definition}

\begin{example}
  Again, let $T$ be as in Example~\ref{primes}.  Let $C=\{1,2,3,5\}$.
  To construct $E(T,C)$ we replace $D(T,k)$ with $D'(T,k)$ for each
  $k\in C$. Then
  \begin{center}
    \begin{tikzpicture}[node distance=0 cm,outer sep = 0pt]
      \tikzstyle{ver}=[rectangle, draw, thick, minimum width=.5cm,
      minimum height=1cm]         
      \tikzstyle{hor}=[rectangle, draw, thick, minimum width=1cm,
      minimum height=.5cm]
      \node[hor] (1) at (1.75,1.75)    {\rm{1}};
      \node[ver] (2) at ( 1.5,   1)    {\rm{2}};
      \node[hor] (3) at (2.75,1.75)    {\rm{3}};
      \node[ver] (4) [right  = of 2]   {\rm{4}};
      \node[hor] (5) [right  = of 3]   {\rm{5}};
      \node[ver] (6) [right  = of 4]   {\rm{6}};
      \node at (0,1.25) {$E(T,C)=$};
      \node at (4.5,1.25) {.};
    \end{tikzpicture}
  \end{center}
\end{example}

\begin{proposition}
  If $T\in\mathcal{T}(M)$ and $C\subset M$ is a cycle then
  $E(T,C)\in\mathcal{T}(M)$. In addition, $E(T,C_1,C_2)=E(T,C_2,C_1)$.
\end{proposition}

\begin{proof}
  This follows from \cite[Proposition 1.5.27, Proposition 1.5.31]{garfinkleclassI}.
\end{proof}

\begin{definition}
  Let $T, T'\in\mathcal{T}(M)$. If $T'=E(T,C_1,C_2,\dots,C_n)$ for
  cycles $C_i$ of $T$ then we say that we can move from $T$ to $T'$
  through the sequence of cycles $C_1, C_2, \dots , C_n$.

  We define $T\approx T'$ if and only if we can move from $T$ to
  $T'$ through a (possibly empty) sequence of open cycles.
\end{definition}

\begin{remark}
  The relation $\approx$ is an equivalence relation on elements of
  $\mathcal{T}(M)$.  It is immediately apparent that $\approx$ is
  reflexive and transitive, and symmetry follows from \cite[Proposition
  1.5.28]{garfinkleclassI}.
\end{remark}

\begin{example}
  Let $T$ and $T'$ be given below
  \begin{center}
    \begin{tabular}{cc}
    \begin{tikzpicture}[node distance=0 cm,outer sep = 0pt]
      \tikzstyle{ver}=[rectangle, draw, thick, minimum width=.5cm,
      minimum height=1cm]         
      \tikzstyle{hor}=[rectangle, draw, thick, minimum width=1cm,
      minimum height=.5cm]
      \node[hor] (1) at ( .75, 1.75)   {\rm{1}};
      \node[ver] (2) at (  .5,    1)   {\rm{2}};
      \node[hor] (3) at (1.75, 1.75)   {\rm{3}};
      \node[ver] (4) [right  = of 2]   {\rm{4}};
      \node at (-.25,1.25) {$T = $};
      \node at (2.5,1.25) {,};
    \end{tikzpicture}
    &
    \begin{tikzpicture}[node distance=0 cm,outer sep = 0pt]
      \tikzstyle{ver}=[rectangle, draw, thick, minimum width=.5cm,
      minimum height=1cm]         
      \tikzstyle{hor}=[rectangle, draw, thick, minimum width=1cm,
      minimum height=.5cm]
      \node[ver] (1) at (   1,  1.5)   {\rm{1}};
      \node[ver] (2) [below = of 1]    {\rm{2}};
      \node[hor] (3) at (1.75, 1.75)   {\rm{3}};
      \node[ver] (4) at ( 1.5,    1)   {\rm{4}};
      \node at ( 0,1) {$T'=$};
      \node at ( 2.5,1) {.};
    \end{tikzpicture}
  \end{tabular}
  \end{center}
  Then it can be shown that $C=\{1,2,3\}$ is an open cycle in $T$, and
  if we move $T$ through $C$ we obtain
  \begin{center}
    \begin{tikzpicture}[node distance=0 cm,outer sep = 0pt]
      \tikzstyle{ver}=[rectangle, draw, thick, minimum width=.5cm,
      minimum height=1cm]         
      \tikzstyle{hor}=[rectangle, draw, thick, minimum width=1cm,
      minimum height=.5cm]
      \node[ver] (1) at (   1,  1.5)   {\rm{1}};
      \node[ver] (2) [below = of 1]    {\rm{2}};
      \node[hor] (3) at (1.75, 1.75)   {\rm{3}};
      \node[ver] (4) at ( 1.5,    1)   {\rm{4}};
      \node at (-.5,1) {$E(T,C) = $};
      \node at ( 3,1) {$=T'$,};
    \end{tikzpicture}
  \end{center}
  so $T\approx T'$.
\end{example}

Remarkably, these domino tableau can help us calculate the cells of a
Coxeter group.  As we will see in the following theorem, the partition
of a Coxeter group $W$ into cells corresponds to the partition of
$\mathcal{T}(M)$ generated by $\sim$.  Two elements $x,w\in W$ are
in the same cell if and only if we can move $T_L(x)$ through open
cycles to obtain $T_L(w)$.

\begin{theorem}\label{leftcells}
  Let $x,w\in W$.  Then $x\sim_L w$ if and only if $T_L(x)\approx
  T_L(w)$.
\end{theorem}

\begin{proof}
  See the discussion in \cite[Section 3]{mcgovern1996left}. 
\end{proof}

In \cite{garfinkleclassI}, \cite{garfinkleclassII} and
\cite{garfinkleclassIII} Garfinkle proves a version of
Theorem~\ref{leftcells} for Coxeter groups of type $B$ using the
following method.  Let $W=W(B_n)$ be a Coxeter group with simple root
system $\Pi$.  For each adjacent $\alpha,\beta\in\Pi$, Garfinkle
defines operators $T_{\alpha\beta}$.  These operators are defined
both on certain subsets of $W$ and on the corresponding type $B$
domino tableaux.  Applying a sequence of $T_{\alpha\beta}$ operators
to an element of $W$ is equivalent to moving the corresponding domino
tableau through a sequence of open cycles \cite[Theorem 3.2.2 and
Proposition 3.2.3]{garfinkleclassIII}.

As defined in \cite[Definition 3.4.1]{garfinkleclassIII}, let ${\bf T}
= \{T_{\alpha\beta} | \alpha,\beta\in\Pi \text{ are adjacent}\}$ and
let $\{T_i\}_{i=0}^k\subset {\bf T}$. In \cite[Theorem
3.5.11]{garfinkleclassIII} Garfinkle proved that two elements $x,w\in
W$ lie in the same left cell if and only if the following two conditions hold:
\begin{enumerate}
\item $T_k(T_{k-1}(\cdots T_{0}(x)\cdots))$ is defined if and only if
  $T_{k}(T_{k-1}(\cdots T_{0}(w)\cdots))$ is defined;
\item the resulting elements must have the same generalized
  $\tau$-invariant.
\end{enumerate}

Most of the proof in type $D$ follows as in type $B$.  However, in
type $D$ we do not have to worry about defining the $T_{\alpha\beta}$
operator when $\alpha$ and $\beta$ have different lengths, but the
branch node introduces complications.  We have to define a new
operator, $T_D$, that corresponds to the four simple roots in the
Dynkin diagram that form a system of type $D_4$ \cite[Discussion
preceding Lemma 3.1]{mcgovern2000triangularity}.  The definition of
$T_D$ is given in \cite[Theorem 2.15]{garfinkle1992structure} and
\cite[Discussion preceding Lemma 3.1]{mcgovern2000triangularity}.

The set of operators $\{T_{\alpha\beta}|\alpha,\beta\in\Pi\text{
  are adjacent}\}\cup\{T_D\}$ then preserve left cells. As in type $B$,
applying a sequence of operators from ${\bf T}$ to an element of $W$
is equivalent to moving the corresponding domino tableau through a
sequence of open cycles \cite{garfinkleclassIV},
\cite[4.1]{garfinkle1992structure}.

\begin{corollary}\label{rightcells}
  Let $x,w\in W$. Then $x\sim_R w$ if and only if $T_R(x)\approx T_R(w)$.
\end{corollary}

\begin{proof}
  By definition $x\sim_R w$ if and only if $x^{-1}\sim_L w^{-1}$,
  which by Theorem~\ref{leftcells} happens if and only if
  $T_L(x^{-1})\approx T_L(w^{-1})$, so $T_R(x)\approx T_R(w)$ by
  Definition~\ref{dominotableaux}.
\end{proof}

\chapter{Calculating \texorpdfstring{$a$}{a}-values of bad elements}\label{dombad}

\section{Constructing \texorpdfstring{$T_L(w_n)$}{TL(wn)}}

We can use Theorem~\ref{leftcells} to better understand the two sided
cells of bad elements in type $D$ by computing their domino tableaux.
Since the unique longest bad element, $w_n$, in $W(D_n)$ is an
involution, we have $T_L(w_n)=T_R(w_n)$, so it suffices to calculate
$T_L(w_n)$.  Furthermore, by Lemma~\ref{equalbad}, $w_n$ and
$w_{n+1}$ have the same reduced expressions for even $n$, so we will
only consider the case where $n$ is even.

Recall the signed permutation representation of $w_n$
from Theorem~\ref{dbad}:
\[
w_n= \left((-1)^{n/2},\underline{n},3,\underline{n-2}, 5,\dots,
  \underline{4},n-1,\underline{2}\right).
\]

\begin{lemma}\label{badtableaux2}
  Let $n\equiv 2\bmod 4$ and let $w_n\in W(D_n)$ be the unique longest
  bad element.  Then we have
  \begin{center}

  \end{center}
\end{lemma}

\begin{proof}
  We will build up $T_L(w_n)$ following the method used in
  Remark~\ref{build}.  It will be enough to show that after $2k$ steps
  we obtain
  \begin{center}
    %
  \end{center}
  Note that $k\leq \frac{n}{2}$, so $T^{2k}_L(w_n)$ is always a valid
  domino tableau. Then $T_L(w_n)$ will take $n$ steps to build, so if
  we set $k=\frac{n}{2}$ we obtain the desired result.

  We will proceed by induction on $k$.  For the base case, let $k=1$.
  Then since $w_n^{-1}(1)=-1$ and $w_n^{-1}(2)=-n$ we have
  \begin{center}
    %
  \end{center}
  and the hypothesis holds when $k=1$.

  Suppose that $k>1$ and we have created the first $2k$ steps of
  $T_L(w_n)$, $T^{2k}_L(w_n)$, using Remark~\ref{build}, resulting in
  \begin{center}
    %
  \end{center}

  Now since $w_n^{-1}(2k+1)=2k+1$ when $k>1$, we must next add a domino with label
  $2k+1$.  We first add the $2k+1$ domino, which simply gets added on
  to the right end of the top row:
  \begin{center}
    %
  \end{center}

  Then $w_n^{-1}(2k+2)=-(n-2k)$, so we next add a vertical domino with
  label $n-2k$. We first remove all dominoes with labels greater than
  $n-2k$.  Then the domino with label $n-2k$ is inserted in the place of the
  $n-2k+2$ domino, which has the effect of bumping the even dominoes
  to the right:
  \begin{center}
    %
  \end{center}
\end{proof}

\begin{lemma}\label{badtableaux0}
  Let $n\equiv 0\bmod 4$ and let $w_n\in W(D_n)$ be the unique longest
  bad element.  Then we have
  \begin{center}
    %
  \end{center}
\end{lemma}

\begin{proof}
  Now suppose that $n\equiv 0\bmod 4$.  As in
  Lemma~\ref{badtableaux2}, we will build up $T_L(w_n)$ following the
  method used in Remark~\ref{build}.  It will be enough to show that
  after $2k$ steps we obtain
  \begin{center}
    %
  \end{center}
  Note that $k\leq \frac{n}{2}$, so $T^{2k}_L(w_n)$ is always a valid
  domino tableau. Then $T_L(w_n)$ will take $n$ steps to build, so if
  we set $k=\frac{n}{2}$ we obtain the desired result.

  We will proceed by induction on $k$.  For the base case, let $k=1$.
  Then since $w_n^{-1}(1)=1$ and $w_n^{-1}(2)=-n$ we have
  \begin{center}
    %
  \end{center}
  so the hypothesis holds when $k=1$.

  Suppose that $k>1$ and we have created the first $2k$ steps of
  $T_L(w_n)$, $T^{2k}_L(w_n)$, using Remark~\ref{build}, resulting in
  \begin{center}
    %
  \end{center}

  Now since $w_n^{-1}(2k+1)=2k+1$ when $k>1$ we must next add a domino with label
  $2k+1$.  We first add the $2k+1$ domino, which simply gets added on
  to the right end of the top row:
  \begin{center}
    %
  \end{center}

  Then $w_n^{-1}(2k+2)=-(n-2k)$, so we next add a vertical domino with
  label $n-2k$. We first remove all dominoes with labels greater than
  $n-2k$.  Then the domino with label $n-2k$ is inserted in the place of the
  $n-2k+2$ domino, which has the effect of bumping the even dominoes
  to the right:
  \begin{center}
    %
  \end{center}
  
\end{proof}

\section{An element in the left cell of \texorpdfstring{$w_n$}{wn}}

Now that we have computed $T_L(w_n)$ we can use
Theorem~\ref{leftcells} to find an element, $v_n$, in the same left
cell as $w_n$.  We will move $T_L(w_n)$ through cycles to help us find
$v_n$.

\begin{lemma}\label{badcycles2}
  Let $n\equiv 2\bmod 4$. Each $T_L(w_n)$ contains an open cycle $C_1=\{1,2,3,5,7,9\dots, n-1\}$.
  If we move $T_L(w_n)$ through $C_1$ then we obtain
  \begin{center}
    %
  \end{center}
  where the shaded squares above correspond to fixed squares. We now
  label $r_k$ for $k\in C_1$ as in
  Definition~\ref{Pprime}.
  \begin{center}
    %
  \end{center}

  Now for $k\in\{1,3,5,\dots, n-1\}$ we have $0=r_k<k$, and $4=r_2>2$.
  Then if we replace $D(T_L(w_n),k)$ with $D'(T_L(w_n),k)$ for $k\in
  C_1$ in $T_L(w_n)$, we obtain

  \begin{center} 
    %
  \end{center}
  
  Observe that each intersection
  \begin{align*}
    P'(T_L(w_n),2) &\cap P(T_L(w_n),1)\\
    P'(T_L(w_n),1) &\cap P(T_L(w_n),3)\\
    &\hspace{.15cm}\vdots\\
    P'(T_L(w_n),n-3) &\cap P(T_L(w_n),n-1)
  \end{align*}
  is nonempty.  Furthermore, the variable square in $P(T_L(w_n),2)$
  does not lie in $(T_L(w_n))'$ and the variable square of
  $P'(T_L(w_n),n-1)$ does not lie in $T_L(w_n)$, so $C_1$ is an open
  cycle, and

  \begin{center}
    %
  \end{center}
\end{proof}

\begin{lemma}\label{badcycles0}
  Let $n\equiv 0\bmod 4$. Each $T_L(w_n)$ contains an open cycle $C_1=\{1,2,3,5,7,9\dots, n-1\}$.
  If we move $T_L(w_n)$ through $C_1$ then we obtain
  \begin{center}
    %
  \end{center}
  where the shaded squares above correspond to fixed squares. We now
  label $r_k$ for $k\in C_1$ as in
  Definition~\ref{Pprime}.
  \begin{center}
    %
  \end{center}

  Now for $k\in\{1,2\}$ we have $0=r_k<k$, and otherwise we have
  $r_k>k$.  Then if we replace $D(T_L(w_n),k)$ with $D'(T_L(w_n),k)$
  for $k\in C_1$ in $T_L(w_n)$, we obtain

  \begin{center}
    %
  \end{center}
  
  Observe that each intersection
  \begin{align*}
    P(T_L(w_n),2) &\cap P'(T_L(w_n),1)\\
    P(T_L(w_n),1) &\cap P'(T_L(w_n),3)\\
    &\hspace{.15cm}\vdots\\
    P(T_L(w_n),n-3) &\cap P'(T_L(w_n),n-1)
  \end{align*}
  is nonempty.  Furthermore, the variable square in $P'(T_L(w_n),2)$
  does not lie in $T_L(w_n)$ and the variable square of
  $P(T_L(w_n),n-1)$ does not lie in $(T_L(w_n))'$, so $C_1$ is an open
  cycle, and

  \begin{center}
    %
  \end{center}
\end{proof}

Now we can move $T_L(w_n)$ through $C_1$ and use
Theorem~\ref{leftcells} to find another element in the same left cell
as $w_n$.

\begin{lemma}\label{v_ntableau}
  Suppose that $n\geq 6$ is even and let $v_n=s_ns_{n-1}s_n w_{n-2}$.
  Then we have the following:
  \begin{enumerate}
  \item $v_n = \left((-1)^{(n-2)/2},\underline{n},3,\underline{n-4},5,
      \underline{n-6}, 7,\dots,n-3,\underline{2},n-1,n-2\right)$;
  \item $T^{n-3}_L(v_n)=T^{n-3}_L(w_{n-2})$.
  \end{enumerate}
\end{lemma}

\begin{proof}
  We will prove the lemma using domino tableaux.  We first find the
  signed permutation representation of $v_n$. Consider $w_{n-2}\in
  W(D_n)$. By Theorem~\ref{dbad} we have
  \[
  w_{n-2} = \left((-1)^{(n-2)/2},\underline{n-2},3,\underline{n-4}, 5,\dots,
    \underline{4},n-3,\underline{2}, n-1, n\right).
  \]
  Now we can use Proposition~\ref{smultiply} to find $v_n$.  We have
  \begin{align*}
    s_n w_{n-2} &=
    \left((-1)^{(n-2)/2},\underline{n-2},3,\underline{n-4}, 5,\dots,
      \underline{4},n-3,\underline{2}, n, n-1\right),\\
    s_{n-1}s_n w_{n-2} &=
    \left((-1)^{(n-2)/2},\underline{n-1},3,\underline{n-4}, 5,\dots,
      \underline{4},n-3,\underline{2}, n, n-2\right),\text{ and}\\
    s_ns_{n-1}s_n w_{n-2} &=
    \left((-1)^{(n-2)/2},\underline{n},3,\underline{n-4}, 5,\dots,
      \underline{4},n-3,\underline{2}, n-1, n-2\right),
  \end{align*}
  so
  \[
  v_n = \left((-1)^{(n-2)/2},\underline{n},3,\underline{n-4},5,
      \underline{n-6}, 7,\dots,n-3,\underline{2},n-1,n-2\right).
  \]

  Now since $w_{n-2}$ is an involution, we have
  $v^{-1}_n=w_{n-2}s_ns_{n-1}s_n$ so by Proposition~\ref{smultiply} we have
  \[
  \left(v_n^{-1}(1),v_n^{-1}(2),\dots,v_n^{-1}(n-3)\right) =
  \left(w_{n-2}^{-1}(1),w_{n-2}^{-1}(2),\dots,w_{n-2}^{-1}(n-3)\right),
  \]
  and $T^{n-3}_L(v_n)=T^{n-3}_L(w_{n-2})$.
\end{proof}

\begin{lemma}\label{v_n0}
  If $n\geq 6$ and $n\equiv 0\bmod 4$, then $w_n\sim_L v_n$.
\end{lemma}

\begin{proof}
  We see that $n-2\equiv 2\bmod 4$, so by Lemmas~\ref{badtableaux2} and~\ref{v_ntableau} we have
  \begin{center}
    \begin{tikzpicture}[node distance=0 cm,outer sep = 0pt]
      \node[bver] (1) at (  2,  3)   {1};
      \node[bver] (2) [below  = of 1]   {\begin{sideways}
          $4$\end{sideways}};
      \node[bhor] (3) at (3.5,3.5)   {3};
      \node[bver] (4) at (  3,  2)   {\begin{sideways}
          $6$\end{sideways}};
      \node[bhor] (5) [right  = of 3]   {5};
      \node[bver] (6) [right  = of 4] {\begin{sideways}
          $8$\end{sideways}};  
      \node[bhor] (n-3) at (10.5,3.5) {$n-3$};
      \node[bver] (n-2) at   (   8,  2)    {\begin{sideways}
          $n-2$\end{sideways}};
      \draw[dashed] (6.5, 4) -- (9.5, 4);
      \draw[dashed] (6.5, 3) -- (9.5, 3);
      \draw[dashed] (4.5, 1) -- (7.5, 1);
      \node at (0,2) {$T^{n-3}_L(v_n)=$};
      \node at (12,2) {.};
    \end{tikzpicture}
  \end{center}

  We have $v_n^{-1}(n-2)=n$, so to obtain $T^{n-2}_L(v_n)$ we see that we add a
  horizontal domino with label $n$ on the end of the first row to get
  
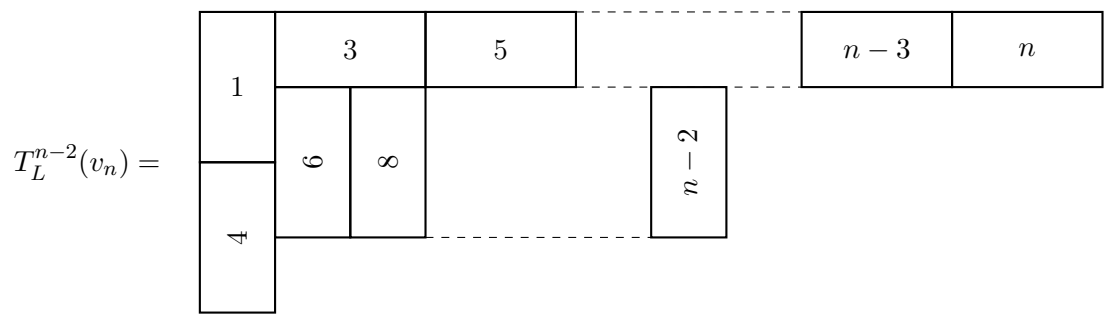
\begin{center}
    \begin{tikzpicture}[node distance=0 cm,outer sep = 0pt]
      \node[bver] (1) at (  2,  3)   {1};
      \node[bver] (2) [below  = of 1]   {\begin{sideways}
          $4$\end{sideways}};
      \node[bhor] (3) at (3.5,3.5)   {3};
      \node[bver] (4) at (  3,  2)   {\begin{sideways}
          $6$\end{sideways}};
      \node[bhor] (5) [right  = of 3]   {5};
      \node[bver] (6) [right  = of 4] {\begin{sideways}
          $8$\end{sideways}};  
      \node[bhor] (n-3) at (10.5,3.5) {$n-3$};
      \node[bver] (n-2) at   (   8,  2)    {\begin{sideways}
          $n-2$\end{sideways}};
      \node at (0,2) {$T^{n-2}_L(v_n)=$};
      \draw[dashed] (6.5, 4) -- (9.5, 4);
      \draw[dashed] (6.5, 3) -- (9.5, 3);
      \draw[dashed] (4.5, 1) -- (7.5, 1);
      \node[bhor] (n) [right  = of n-3]   {$n$};
      \node at (14,2) {.};
    \end{tikzpicture}
  \end{center}
  
  Next, $v_n^{-1}(n-1)=n-1$, so we must next add a horizontal domino
  with label $n-1$. To do this we must first remove all dominoes with
  labels greater than $n-1$, then add a horizontal domino with label
  $n-1$ to the end of the first row.
  
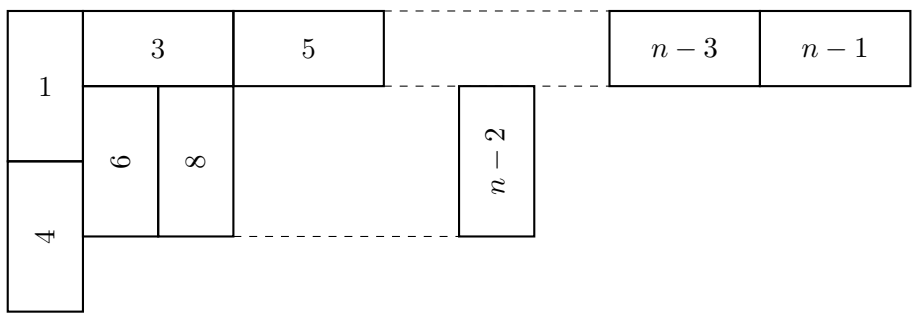
\begin{center}
    \begin{tikzpicture}[node distance=0 cm,outer sep = 0pt]
      \node[bver] (1) at (  2,  3)   {1};
      \node[bver] (2) [below  = of 1]   {\begin{sideways}
          $4$\end{sideways}};
      \node[bhor] (3) at (3.5,3.5)   {3};
      \node[bver] (4) at (  3,  2)   {\begin{sideways}
          $6$\end{sideways}};
      \node[bhor] (5) [right  = of 3]   {5};
      \node[bver] (6) [right  = of 4] {\begin{sideways}
          $8$\end{sideways}};  
      \node[bhor] (n-3) at (10.5,3.5) {$n-3$};
      \node[bver] (n-2) at   (   8,  2)    {\begin{sideways}
          $n-2$\end{sideways}};
      \draw[dashed] (6.5, 4) -- (9.5, 4);
      \draw[dashed] (6.5, 3) -- (9.5, 3);
      \draw[dashed] (4.5, 1) -- (7.5, 1);
      \node[bhor] (n-1) [right  = of n-3]   {$n-1$};
      \node at (14,2) {.};
    \end{tikzpicture}  
  \end{center}
  When we try to replace the domino with label $n$, we see that it
  fully overlaps with the domino with label $n-1$,
  
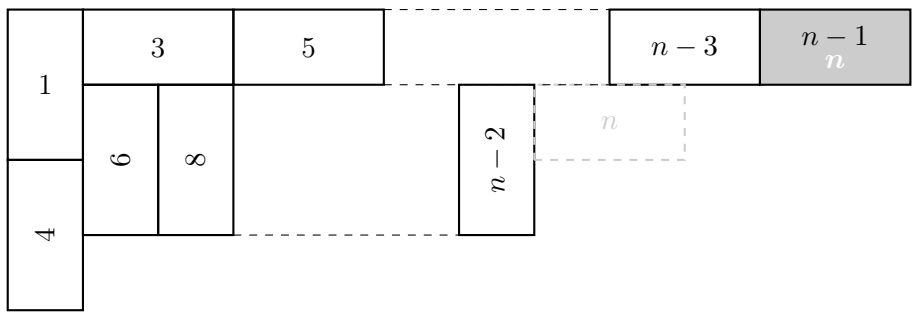
\begin{center}
    \begin{tikzpicture}[node distance=0 cm,outer sep = 0pt]
      \node[bver] (1) at (  2,  3)   {1};
      \node[bver] (2) [below  = of 1]   {\begin{sideways}
          $4$\end{sideways}};
      \node[bhor] (3) at (3.5,3.5)   {3};
      \node[bver] (4) at (  3,  2)   {\begin{sideways}
          $6$\end{sideways}};
      \node[bhor] (5) [right  = of 3]   {5};
      \node[bver] (6) [right  = of 4] {\begin{sideways}
          $8$\end{sideways}};  
      \node[bhor] (n-3) at (10.5,3.5) {$n-3$};
      \node[bver] (n-2) at   (   8,  2)    {\begin{sideways}
          $n-2$\end{sideways}};
      \node[bhor,color=shade,dashed] (n)  at    ( 9.5,2.5) {$n$};
      \draw[dashed] (6.5, 4) -- (9.5, 4);
      \draw[dashed] (6.5, 3) -- (9.5, 3);
      \draw[dashed] (4.5, 1) -- (7.5, 1);
      \node[bhor,fill=shade] (n-1) [right  = of n-3]
      {\shortstack{$n-1$\\\color{white}{$\boldsymbol n$}}};
      \node at (14,2) {,};
    \end{tikzpicture}
  \end{center}
  so we use case (2) of Definition~\ref{shuffle} to bump the domino
  with label $n$ to the end of the second row
  
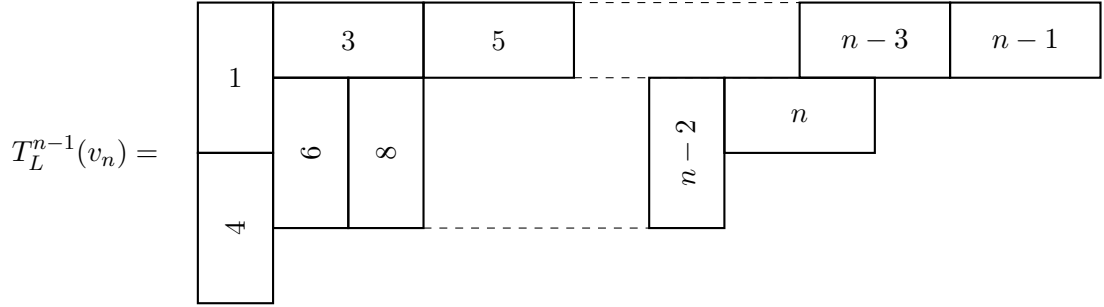
\begin{center}
    \begin{tikzpicture}[node distance=0 cm,outer sep = 0pt]
      \node[bver] (1) at (  2,  3)   {1};
      \node[bver] (2) [below  = of 1]   {\begin{sideways}
          $4$\end{sideways}};
      \node[bhor] (3) at (3.5,3.5)   {3};
      \node[bver] (4) at (  3,  2)   {\begin{sideways}
          $6$\end{sideways}};
      \node[bhor] (5) [right  = of 3]   {5};
      \node[bver] (6) [right  = of 4] {\begin{sideways}
          $8$\end{sideways}};  
      \node[bhor] (n-3) at (10.5,3.5) {$n-3$};
      \node[bver] (n-2) at   (   8,  2)    {\begin{sideways}
          $n-2$\end{sideways}};
      \node[bhor] (n)  at    ( 9.5,2.5) {$n$};
      \node at (0,2) {$T^{n-1}_L(v_n)=$};
      \draw[dashed] (6.5, 4) -- (9.5, 4);
      \draw[dashed] (6.5, 3) -- (9.5, 3);
      \draw[dashed] (4.5, 1) -- (7.5, 1);
      \node[bhor] (n-1) [right  = of n-3]   {$n-1$};
      \node at (14,2) {.};
   \end{tikzpicture}
  \end{center}

  Finally, $v_n^{-1}(n)=-2$, so to complete the calculation of
  $T_L(v_n)$ we must add a vertical domino with label 2.  As in
  Example~\ref{buildT}, this bumps all dominoes with even labels less
  than $n$ to the right.  However, when we try to replace the domino
  with label $n$ it partially overlaps with the domino with label
  $n-2$.
  
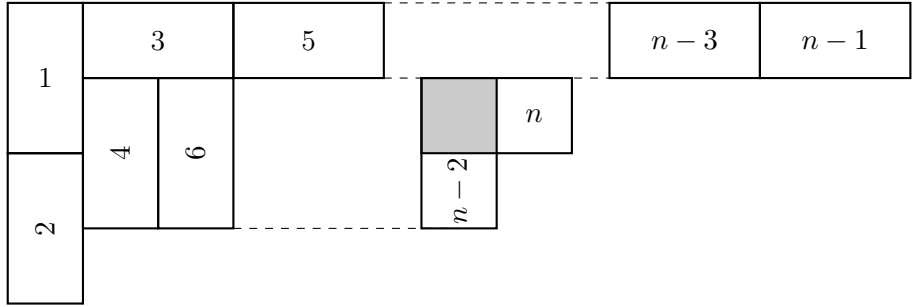
\begin{center}
    \begin{tikzpicture}[node distance=0 cm,outer sep = 0pt]
      \node[bver] (1) at (  2,  3)   {1};
      \node[bver] (2) [below  = of 1]   {\begin{sideways}
          $2$\end{sideways}};
      \node[bhor] (3) at (3.5,3.5)   {3};
      \node[bver] (4) at (  3,  2)   {\begin{sideways}
          $4$\end{sideways}};
      \node[bhor] (5) [right  = of 3]   {5};
      \node[bver] (6) [right  = of 4] {\begin{sideways}
          $6$\end{sideways}};  
      \node[bhor] (n-3) at (10.5,3.5) {$n-3$};
      \node[bfix] (over) at (  7.5,2.5) {}; 
      \node[bver] (n-2) at   ( 7.5,  2)    {\begin{sideways}
          \hspace{-.95cm}$n-2$\end{sideways}};
      \node[bhor] (n)  at    (   8, 2.5) {\hspace{1cm}$n$};
      \draw[dashed] (6.5, 4) -- (9.5, 4);
      \draw[dashed] (6.5, 3) -- (9.5, 3);
      \draw[dashed] (4.5, 1) -- (7.5, 1);
      \node[bhor] (n-1) [right  = of n-3]   {$n-1$};
      \node at (14,2) {.};
    \end{tikzpicture}
  \end{center}
  Now we use case (3) of Definition~\ref{shuffle} to shuffle the
  domino with label $n$ into position that allows it to fit into the
  tableau
  
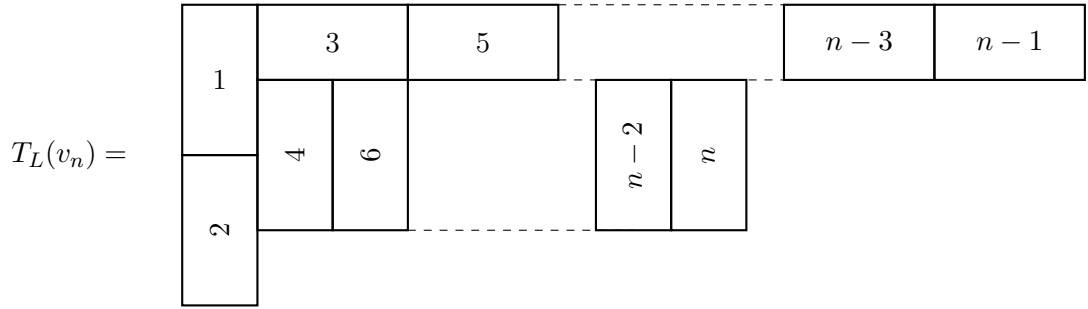
\begin{center}
    \begin{tikzpicture}[node distance=0 cm,outer sep = 0pt]
      \node[bver] (1) at (  2,  3)   {1};
      \node[bver] (2) [below  = of 1]   {\begin{sideways}
          $2$\end{sideways}};
      \node[bhor] (3) at (3.5,3.5)   {3};
      \node[bver] (4) at (  3,  2)   {\begin{sideways}
          $4$\end{sideways}};
      \node[bhor] (5) [right  = of 3]   {5};
      \node[bver] (6) [right  = of 4] {\begin{sideways}
          $6$\end{sideways}};  
      \node[bhor] (n-3) at (10.5,3.5) {$n-3$};
      \node[bver] (n-2) at   ( 7.5,  2)    {\begin{sideways}
          $n-2$\end{sideways}};
      \node[bver] (n)  at    ( 8.5,  2) {\begin{sideways}
          $n$\end{sideways}};
      \node at (0,2) {$T_L(v_n)=$};
      \draw[dashed] (6.5, 4) -- (9.5, 4);
      \draw[dashed] (6.5, 3) -- (9.5, 3);
      \draw[dashed] (4.5, 1) -- (7.5, 1);
      \node[bhor] (n-1) [right  = of n-3]   {$n-1$};
      \node at (14,2) {.};
    \end{tikzpicture}
  \end{center}
  Then by Lemma~\ref{badcycles2} we have $T_L(v_n)=E(T_L(w_n),C_1)$,
  so $T_L(v_n)\approx T_L(w_n)$, thus by Theorem~\ref{leftcells}, we
  have $w_n\sim_L v_n$.
\end{proof}

\begin{lemma}\label{v_n2}
  If $n\geq 6$ and $n\equiv 2\bmod 4$, then $w_n\sim_L v_n$.
\end{lemma}

\begin{proof}
  We see that $n-2\equiv 0\bmod 4$, so by Lemmas~\ref{badtableaux0}
  and~\ref{v_ntableau} we have
  \begin{center}
    \begin{tikzpicture}[node distance=0 cm,outer sep = 0pt]
      \node[bhor] (1) at (  2.5, 3.5) {1};
      \node[bver] (2) at (    2,   2) {\begin{sideways}
          $4$\end{sideways}};
      \node[bhor] (3) [right  = of 1] {3};
      \node[bver] (4) [right  = of 2] {\begin{sideways}
          $6$\end{sideways}};
      \node[bhor] (5) [right  = of 3] {5};
      \node[bver] (6) [right  = of 4] {\begin{sideways}
          $8$\end{sideways}};  
      \node[bhor] (n-3) at (11.5,3.5) {$n-3$};
      \node[bver] (n-2) at ( 8.5,  2) {\begin{sideways}
          $n-2$\end{sideways}};
      \node at (0,2) {$T^{n-3}_L(v_n)=$};
      \draw[dashed] (6.5, 4) -- (10.5, 4);
      \draw[dashed] (6.5, 3) -- (10.5, 3);
      \draw[dashed] (4.5, 1) -- ( 8.5, 1);
      \node at (13,2) {.};
    \end{tikzpicture}
  \end{center}

  We have $v_n^{-1}(n-2)=n$, so to obtain $T^{n-2}_L(v_n)$ we see that
  we add a horizontal domino with label $n$ on the end of the first
  row to get
  \begin{center}
    \begin{tikzpicture}[node distance=0 cm,outer sep = 0pt]
      \node[bhor] (1) at (  2.5, 3.5) {1};
      \node[bver] (2) at (    2,   2) {\begin{sideways}
          $4$\end{sideways}};
      \node[bhor] (3) [right  = of 1] {3};
      \node[bver] (4) [right  = of 2] {\begin{sideways}
          $6$\end{sideways}};
      \node[bhor] (5) [right  = of 3] {5};
      \node[bver] (6) [right  = of 4] {\begin{sideways}
          $8$\end{sideways}};  
      \node[bhor] (n-3) at (11.5,3.5) {$n-3$};
      \node[bhor] (n) [right  = of n-3] {$n$};
      \node[bver] (n-2) at ( 8.5,  2) {\begin{sideways}
          $n-2$\end{sideways}};
      \node at (0,2) {$T^{n-2}_L(v_n)=$};
      \draw[dashed] (6.5, 4) -- (10.5, 4);
      \draw[dashed] (6.5, 3) -- (10.5, 3);
      \draw[dashed] (4.5, 1) -- ( 8.5, 1);
      \node at (15,2) {.};
    \end{tikzpicture}
  \end{center}
  
  Next, $v_n^{-1}(n-1)=n-1$, so we must next add a horizontal domino
  with label $n-1$. To do this we must first remove all dominoes with
  labels greater than $n-1$, then add a horizontal domino with label
  $n-1$ to the end of the first row.
  \begin{center}
    \begin{tikzpicture}[node distance=0 cm,outer sep = 0pt]
      \node[bhor] (1) at (  2.5, 3.5) {1};
      \node[bver] (2) at (    2,   2) {\begin{sideways}
          $4$\end{sideways}};
      \node[bhor] (3) [right  = of 1] {3};
      \node[bver] (4) [right  = of 2] {\begin{sideways}
          $6$\end{sideways}};
      \node[bhor] (5) [right  = of 3] {5};
      \node[bver] (6) [right  = of 4] {\begin{sideways}
          $8$\end{sideways}};  
      \node[bhor] (n-3) at (11.5,3.5) {$n-3$};
      \node[bhor] (n) [right  = of n-3] {$n-1$};
      \node[bver] (n-2) at ( 8.5,  2) {\begin{sideways}
          $n-2$\end{sideways}};
      \draw[dashed] (6.5, 4) -- (10.5, 4);
      \draw[dashed] (6.5, 3) -- (10.5, 3);
      \draw[dashed] (4.5, 1) -- ( 8.5, 1);
      \node at (15,2) {.};
    \end{tikzpicture}
  \end{center}
  When we try to replace the domino with label $n$, we see that it
  fully overlaps with the domino with label $n-1$,
  \begin{center}
    \begin{tikzpicture}[node distance=0 cm,outer sep = 0pt]
      \node[bhor] (1) at (  2.5, 3.5) {1};
      \node[bver] (2) at (    2,   2) {\begin{sideways}
          $4$\end{sideways}};
      \node[bhor] (3) [right  = of 1] {3};
      \node[bver] (4) [right  = of 2] {\begin{sideways}
          $6$\end{sideways}};
      \node[bhor] (5) [right  = of 3] {5};
      \node[bver] (6) [right  = of 4] {\begin{sideways}
          $8$\end{sideways}};  
      \node[bhor] (n-3) at (11.5,3.5) {$n-3$};
      \node[bhor,fill=shade] (n-1) [right  = of n-3]
      {\shortstack{$n-1$\\\color{white}{$\boldsymbol n$}}};
      \node[bver] (n-2) at ( 8.5,  2) {\begin{sideways}
          $n-2$\end{sideways}};
      \node[bhor,color=shade,dashed] (n)  at    ( 10,2.5) {$n$};
      \draw[dashed] (6.5, 4) -- (10.5, 4);
      \draw[dashed] (6.5, 3) -- (10.5, 3);
      \draw[dashed] (4.5, 1) -- ( 8.5, 1);
      \node at (15,2) {,};
    \end{tikzpicture}
  \end{center}
  so we use case (2) of Definition~\ref{shuffle} to bump the domino
  with label $n$ to the end of the second row
  \begin{center}
    \begin{tikzpicture}[node distance=0 cm,outer sep = 0pt]
      \node[bhor] (1) at (  2.5, 3.5) {1};
      \node[bver] (2) at (    2,   2) {\begin{sideways}
          $4$\end{sideways}};
      \node[bhor] (3) [right  = of 1] {3};
      \node[bver] (4) [right  = of 2] {\begin{sideways}
          $6$\end{sideways}};
      \node[bhor] (5) [right  = of 3] {5};
      \node[bver] (6) [right  = of 4] {\begin{sideways}
          $8$\end{sideways}};  
      \node[bhor] (n-3) at (11.5,3.5) {$n-3$};
      \node[bhor] (n-1) [right  = of n-3] {$n-1$};
      \node[bver] (n-2) at ( 8.5,  2) {\begin{sideways}
          $n-2$\end{sideways}};
      \node[bhor] (n)  at    ( 10,2.5) {$n$};
      \node at (0,2) {$T^{n-1}_L(v_n)=$};
      \draw[dashed] (6.5, 4) -- (10.5, 4);
      \draw[dashed] (6.5, 3) -- (10.5, 3);
      \draw[dashed] (4.5, 1) -- ( 8.5, 1);
      \node at (15,2) {.};
    \end{tikzpicture}
  \end{center}

  Finally, $v_n^{-1}(n)=-2$, so to complete the calculation of
  $T_L(v_n)$ we must add a vertical domino with label 2.  As in
  Lemma~\ref{v_n0}, this bumps all dominoes with even labels less
  than $n$ to the right.  However, when we try to replace the domino
  with label $n$ it partially overlaps with the domino with label
  $n-2$.
  \begin{center}
    \begin{tikzpicture}[node distance=0 cm,outer sep = 0pt]
      \node[bhor] (1) at (  2.5, 3.5) {1};
      \node[bver] (2) at (    2,   2) {\begin{sideways}
          $2$\end{sideways}};
      \node[bhor] (3) [right  = of 1] {3};
      \node[bver] (4) [right  = of 2] {\begin{sideways}
          $4$\end{sideways}};
      \node[bhor] (5) [right  = of 3] {5};
      \node[bver] (6) [right  = of 4] {\begin{sideways}
          $6$\end{sideways}};  
      \node[bhor] (n-3) at (11.5,3.5) {$n-3$};
      \node[bhor] (n-1) [right  = of n-3] {$n-1$};
      \node[bfix] (over) at (8.5,2.5)  {};
      \node[bver] (n-2) at ( 8.5,  2) {\begin{sideways}
          \hspace{-.95cm}$n-2$\end{sideways}};
      \node[bhor] (n)  at    (  9,2.5) {\hspace{1cm}$n$};
      \draw[dashed] (6.5, 4) -- (10.5, 4);
      \draw[dashed] (6.5, 3) -- (10.5, 3);
      \draw[dashed] (4.5, 1) -- ( 8.5, 1);
      \node at (15,2) {.};
    \end{tikzpicture}
  \end{center}
  Now we use case (3) of Definition~\ref{shuffle} to shuffle the
  domino with label $n$ into position that allows it to fit into the
  tableau
  \begin{center}
    \begin{tikzpicture}[node distance=0 cm,outer sep = 0pt]
      \node[bhor] (1) at (  2.5, 3.5) {1};
      \node[bver] (2) at (    2,   2) {\begin{sideways}
          $2$\end{sideways}};
      \node[bhor] (3) [right  = of 1] {3};
      \node[bver] (4) [right  = of 2] {\begin{sideways}
          $4$\end{sideways}};
      \node[bhor] (5) [right  = of 3] {5};
      \node[bver] (6) [right  = of 4] {\begin{sideways}
          $6$\end{sideways}};  
      \node[bhor] (n-3) at (11.5,3.5) {$n-3$};
      \node[bhor] (n-1) [right  = of n-3] {$n-1$};
      \node[bver] (n-2) at ( 8.5,  2) {\begin{sideways}
          $n-2$\end{sideways}};
      \node[bver] (n) [right  = of n-2] {\begin{sideways}
            $n$\end{sideways}};
      \draw[dashed] (6.5, 4) -- (10.5, 4);
      \draw[dashed] (6.5, 3) -- (10.5, 3);
      \draw[dashed] (4.5, 1) -- ( 8.5, 1);
      \node at (0,2) {$T_L(v_n)=$};
      \node at (15,2) {.};
    \end{tikzpicture}
  \end{center}
  Then by Lemma~\ref{badcycles0} we have $T_L(v_n)=E(T_L(w_n),C_1)$,
  so $T_L(v_n)\approx T_L(w_n)$, thus by Theorem~\ref{leftcells}, we
  have $w_n\sim_L v_n$.
\end{proof}

Now we have an element, $v_n$, in the same left cell as $w_n$.
Unfortunately, $a(v_n)$ is still difficult to calculate.  However, we
can find an element in the same right cell as $v_n$ that will allow us
to calculate $a(w_n)$.  

\section{An element in the two-sided cell of \texorpdfstring{$w_n$}{wn}}

We will now find an element, $u_n$ in the same right cell as $v_n$,
and therefore in the same two-sided cell as $w_n$.  To do this we will
use the function $T_R$.  Recall that $T_R(w)=T_L(w^{-1})$, so to
construct $T_R(w)$ we follow the method outlined in
Remark~\ref{build}, but we switch the roles of $w$ and $w^{-1}$.  We
begin by constructing $T_R(v_n)$ and moving $T_R(v_n)$ through a
cycle.

\begin{lemma}\label{v_norder}
  Recall that $v_n = s_n s_{n-1} s_n w_{n-2}$. The domino tableau
  $T^{n-2}_R(v_n)$ is constructed in the same way as
  $T^{n-2}_R(w_{n-2})$.  That is, the dominoes in $T^{n-2}_R(v_n)$ lie
  in the same positions and relative order as those in
  $T^{n-2}_R(w_{n-2})$, but the labels correspond to the first $n-2$
  entries in the signed permutation of $v_n$. In particular we have
  
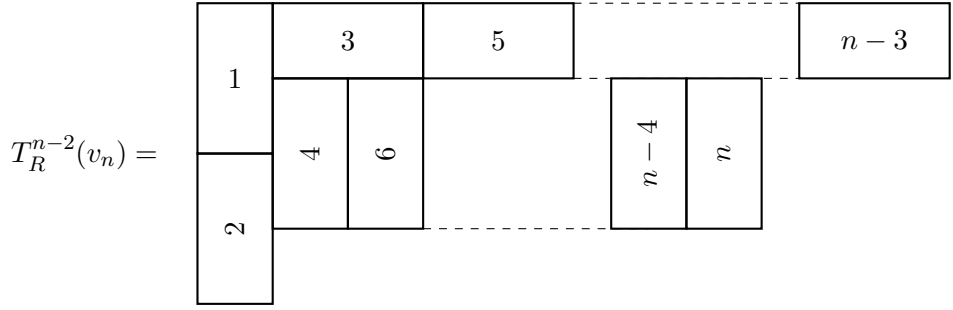
\begin{center}
    \begin{tikzpicture}[node distance=0 cm,outer sep = 0pt]
      \tikzstyle{ver}=[rectangle, draw, thick, minimum width=1cm,
      minimum height=2cm]         
      \tikzstyle{hor}=[rectangle, draw, thick, minimum width=2cm,
      minimum height=1cm]
      \node[ver] (1) at (  2,  3)   {$1$};
      \node[ver] (2) [below  = of 1]   {\begin{sideways}
          $2$\end{sideways}};
      \node[hor] (3) at (3.5,3.5)   {$3$};
      \node[ver] (4) at (  3,  2)   {\begin{sideways}
          $4$\end{sideways}};
      \node[hor] (5) [right  = of 3]   {$5$};
      \node[ver] (6) [right  = of 4] {\begin{sideways}
          $6$\end{sideways}};  
      \node[hor] (n-3) at (10.5,3.5) {$n-3$};
      \node[ver] (n-4) at   ( 7.5,  2)    {\begin{sideways}
          $n-4$\end{sideways}};
      \node at (0,2) {$T^{n-2}_R(v_n)=$};
      \draw[dashed] (6.5, 4) -- (9.5, 4);
      \draw[dashed] (6.5, 3) -- (9.5, 3);
      \draw[dashed] (4.5, 1) -- (7.5, 1);
      \node[ver] (n) [right  = of n-4] {\begin{sideways}
          $n$\end{sideways}};
    \end{tikzpicture}
  \end{center}
  when $n\equiv 0\bmod 4$ and 
  
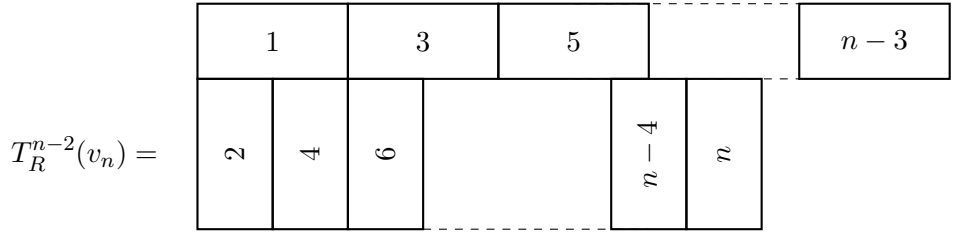
\begin{center}
    \begin{tikzpicture}[node distance=0 cm,outer sep = 0pt]
      \node[bhor] (1) at (2.5,3.5)   {$1$};
      \node[bver] (2) at (  2,  2)   {\begin{sideways}
          $2$\end{sideways}};
      \node[bhor] (3) [right = of 1]   {$3$};
      \node[bver] (4) [right = of 2]   {\begin{sideways}
          $4$\end{sideways}};
      \node[bhor] (5) [right  = of 3]   {$5$};
      \node[bver] (6) [right  = of 4] {\begin{sideways}
          $6$\end{sideways}};  
      \node[bhor] (n-3) at (10.5,3.5) {$n-3$};
      \node[bver] (n-4) at   ( 7.5,  2)    {\begin{sideways}
          $n-4$\end{sideways}};
      \node at (0,2) {$T^{n-2}_R(v_n)=$};
      \draw[dashed] (6.5, 4) -- (9.5, 4);
      \draw[dashed] (6.5, 3) -- (9.5, 3);
      \draw[dashed] (4.5, 1) -- (7.5, 1);
      \node[bver] (n) [right  = of n-4] {\begin{sideways}
          $n$\end{sideways}};
    \end{tikzpicture}
  \end{center}
  when $n\equiv 2\bmod 4$.
\end{lemma}

\begin{proof}
  In Theorem~\ref{dbad} and Lemma~\ref{v_ntableau} we found that
  \begin{align*}
  w_{n-2} &= \left((-1)^{(n-2)/2},\underline{n-2},3,\underline{n-4}, 5,\dots,
    \underline{4},n-3,\underline{2}, n-1, n\right),\text{ and}\\
  v_n &=  \left((-1)^{(n-2)/2},\underline{n},3,\underline{n-4},5,
      \underline{n-6}, 7,\dots,n-3,\underline{2},n-1,n-2\right),
  \end{align*}
  so
  \[
  \left(\left|v_n(1)\right|,\left|v_n(2)\right|,
    \dots,\left|v_n(n-2)\right|\right)
  \]
  are in the same relative order as
  \[
  \left(\left|w_{n-2}(1)\right|,\left|w_{n-2}(2)\right|,
    \dots,\left|w_{n-2}(n-2)\right|\right),
  \]
  and
  \[
  \left(\sign\left(v_n(1)\right),
    \dots,\sign\left(v_n(n-2)\right)\right) =
  \left(\sign\left(w_{n-2}(1)\right),
    \dots,\sign\left(w_{n-2}(n-2)\right)\right).
  \]
  Thus, $T^{n-2}_R(v_n)$ is constructed in the same way as
  $T^{n-2}_R(w_{n-2})$.
\end{proof}

\begin{lemma}\label{v_nright0}
  Let $n\equiv 0\bmod 4$.  Then $T_R(v_n)$ contains an open cycle
  \[
  C_2=\{1,2,3,5,7,9\dots,n-5,n-3,n-2\}.
  \]
  If we move $T_R(v_n)$ through $C_2$ we obtain
  \begin{center}

  \end{center}

  By Lemma~\ref{v_ntableau} we have $v_n(n-1)=n-1$, so to obtain
  $T^{n-1}_R(v_n)$ we see that we add $n-1$ as a horizontal domino on
  the end of the first row:
  \begin{center}
    %
 
  \end{center}

  Finally, $v_n(n)=n-2$, so we finish by adding a horizontal domino
  with label $n-2$.  We first remove all dominoes with labels greater
  than $n-2$, then add a domino with label $n-2$ to the end of the
  first row.
  \begin{center}
    %
  \end{center}
  Now when we try to replace the domino with label $n-1$, we see that
  it fully overlaps with the domino with label $n-2$:
  \begin{center}
    %
  \end{center}
  so we use case (2) of Definition~\ref{shuffle} to bump the domino
  with label $n-1$ to the end of the second row:
  \begin{center}
    %
  \end{center}
  We finally must replace the domino with label $n$, which we see
  partially overlaps with the domino with label $n-1$:
  \begin{center}
    %
  \end{center}
  so we use case (3) of Definition~\ref{shuffle} to shuffle the domino
  with label $n$ into position that allows it to fit into the tableau
  \begin{center}
    %
  \end{center}

  As in Lemma~\ref{badcycles0}, we can shade the fixed squares and
  compute the $r_k$ for \\$k\in\{1,2,3,5,7,9,\dots,n-5,n-3,n-2\}$ in
  order to help us find cycles.
  \begin{center}
    %
  \end{center}

  We see that $r_2 = 4 > 2$ and $r_k = 0 < k$ for $k= 1,3, 5, 7,
  \dots, n-3, n-2$, so we can use Remark~\ref{computer} to shuffle
  each of the corresponding dominoes. Now, as in
  Lemma~\ref{badcycles0}, $T_R(v_n)$ contains an open cycle
  $C_2=\{1,2,3,5,7,9,\dots,n-5,n-3,n-2\}$.  If we move $T_R(v_n)$
  through this open cycle we obtain
  \begin{center}
    %
  \end{center}
\end{proof}

\begin{lemma}\label{v_nright2}
  Let $n\equiv 2\bmod 4$.  Then $T_R(v_n)$ contains an open cycle
  \[
  C_2=\{1,2,3,5,7,9\dots,n-5,n-3,n-2\}.
  \]
  If we move $T_R(v_n)$
  through $C_2$ we obtain
  \begin{center}
    %
  \end{center}

  By Lemma~\ref{v_ntableau} we have $v_n(n-1)=n-1$, so to obtain
  $T^{n-1}_R(v_n)$ we see that we add $n-1$ as a horizontal domino on
  the end of the first row:
  \begin{center}
    %
  \end{center}

  Finally, $v_n(n)=n-2$, so we finish by adding a horizontal domino
  with label $n-2$.  The shuffling that occurs is the same as in the
  $0\bmod 4$ case, so we obtain
  \begin{center}
    %
  \end{center}

  As in Lemma~\ref{v_nright0}, we can shade the fixed squares and
  compute the $r_k$ for \\$k\in\{1,2,3,5,7,9,\dots,n-5,n-3,n-2\}$ in
  order to help us find cycles.
  \begin{center}
    %
  \end{center}

  We see that $r_1 = 0 < 1$, $r_2 = 0 < 2$ and $r_k > k$ for $k=3, 5, 7,
  \dots, n-3, n-2$, so we can use Remark~\ref{computer} to shuffle each
  of the corresponding dominoes. Now, as in Lemma~\ref{v_nright0},
  $T_R(v_n)$ contains an open cycle
  $C_2=\{1,2,3,5,7,9,\dots,n-5,n-3,n-2\}$.  If we move $T_R(v_n)$
  through this open cycle we obtain
  \begin{center}
    %
  \end{center}
\end{proof}

\begin{remark}\label{v_nright6}
  If $n=6$ then $T_R(v_6)$ is given by 
  \begin{center}
    %
  \end{center}
  and contains the open cycle $C_2=\{1,2,3,4\}$.
  If we move $T_R(v_6)$ through $C_2$ we obtain
  \begin{center}
    %
  \end{center}
\end{remark}

\begin{lemma}\label{u_norder}
  Let $n\geq 8$ be even and let $u_n=w_{n-4}s_ns_{n-1}s_n$. The domino
  tableau $T^{n-4}_R(u_n)$ is constructed in the same way as
  $T_R(v_{n-4})$.  That is, the dominoes in $T^{n-4}_R(u_n)$ lie in the
  same positions and relative order as those in $T_R(v_{n-4})$, but
  the labels correspond to the first $n-4$ entries in the signed
  permutation of $u_n$. Then we obtain
  \begin{center}
    %
  
  \end{center}
  if $n\equiv 2\bmod 4$.
\end{lemma}

\begin{proof}
  We will first find the signed permutation representation of
  $u_n$. By Theorem~\ref{dbad} we have
  \[
  w_{n-4} = \left((-1)^{(n-4)/2},\underline{n-4},3,\underline{n-6}, 5,\dots,
    \underline{4},n-5,\underline{2}, n-3, n-2, n-1, n\right).
  \]
  Now we can use Proposition~\ref{smultiply} to find $u_n$.  We have
  \begin{align*}
    w_{n-4}s_n &=
    \left((-1)^{(n-4)/2},\underline{n-4},3,\underline{n-6}, 5,\dots,
      \underline{4},n-5,\underline{2}, n-3, n-2, n, n-1\right),\\
    w_{n-4}s_n s_{n-1} &=
    \left((-1)^{(n-4)/2},\underline{n-4},3,\underline{n-6}, 5,\dots,
      \underline{4},n-5,\underline{2}, n-3, n, n-2, n-1\right),\text{ and}\\
    w_{n-4}s_ns_{n-1}s_n &=
    \left((-1)^{(n-4)/2},\underline{n-4},3,\underline{n-6}, 5,\dots,
      \underline{4},n-5,\underline{2}, n-3, n, n-1, n-2\right),
  \end{align*}
  so
  \[
  u_n =
  \begin{cases}
    \left(1,\underline{n-4},3,\underline{n-6},5,\dots,\underline{4},n-5,
      \underline{2},n-3,n,n-1,n-2\right)&\mbox{if }n\equiv 0\bmod 4;\\
    \left(\underline{1},\underline{n-4},3,\underline{n-6},5,\dots,\underline{4},n-5,
      \underline{2},n-3,n,n-1,n-2\right)&\mbox{if }n\equiv 2\bmod 4.\\
  \end{cases}
  \]
  
  We next compute $T_R(u_n)$.  We first notice that
  \[
  \left(\left|u_n(1)\right|,\left|u_n(2)\right|,
    \dots,\left|u_n(n-4)\right|\right)
  \]
  are in the same relative order as
  \[
  \left(\left|w_{n-4}(1)\right|,\left|w_{n-4}(2)\right|,
    \dots,\left|w_{n-4}(n-4)\right|\right),
  \]
  and that
  \[
  \left(\sign\left(u_n(1)\right),
    \dots,\sign\left(u_n(n-4)\right)\right) =
  \left(\sign\left(w_{n-4}(1)\right),
    \dots,\sign\left(w_{n-4}(n-4)\right)\right),
  \]
  so $T_R^{n-4}(u_n)$ is constructed in the same way as $T_R(w_{n-4})$.
\end{proof}

\begin{lemma}\label{u_n0}
  If $n\geq 8$ is such that $n\equiv 0\bmod 4$, then $v_n\sim_R
  u_n$.
\end{lemma}

\begin{proof}
  By Lemma~\ref{u_norder} we have
  \begin{center}
    \begin{tikzpicture}[node distance=0 cm,outer sep = 0pt]
      \node[bhor] (1) at (  2,  3.5)   {1};
      \node[bver] (2) at ( 1.5, 2)   {\begin{sideways}
          $2$\end{sideways}};
      \node[bhor] (3) [right = of 1]   {3};
      \node[bver] (4) [right = of 2]   {\begin{sideways}
          $4$\end{sideways}};
      \node[bhor] (5) [right  = of 3]   {5};
      \node[bver] (6) [right  = of 4] {\begin{sideways}
          $6$\end{sideways}};  
      \node[bhor] (n-5) at (9.5,3.5) {$n-5$};
      \node[bver] (n-4) at   ( 7,  2)    {\begin{sideways}
          $n-4$\end{sideways}};
      \node at (-.25,2.5) {$T^{n-4}_R(u_n)=$};
      \draw[dashed] (6.5, 4) -- (8.5, 4);
      \draw[dashed] (6.5, 3) -- (8.5, 3);
      \draw[dashed] (4, 1) -- (6.5, 1);
      \node at (11,2.5) {.};
    \end{tikzpicture}
  \end{center}
  
  Now $u_n(n-3)=n-3$, and since $n-3$ is larger than any of the labels
  of dominoes in $T_R^{n-4}(u_n)$ we simply add a horizontal domino
  with label $n-3$ to the end of the first row.  Similarly,
  $u_n(n-2)=n$, so we add a horizontal domino with label $n$ to the
  end of the first row, thus obtaining
  \begin{center}
    \begin{tikzpicture}[node distance=0 cm,outer sep = 0pt]
      \node[bhor] (1) at (  2,  3.5)   {1};
      \node[bver] (2) at ( 1.5, 2)   {\begin{sideways}
          $2$\end{sideways}};
      \node[bhor] (3) [right = of 1]   {3};
      \node[bver] (4) [right = of 2]   {\begin{sideways}
          $4$\end{sideways}};
      \node[bhor] (5) [right  = of 3]   {5};
      \node[bver] (6) [right  = of 4] {\begin{sideways}
          $6$\end{sideways}};  
      \node[bhor] (n-5) at (9.5,3.5) {$n-5$};
      \node[bhor] (n-3) [right = of n-5] {$n-3$};
      \node[bhor] (n) [right = of n-3] {$n$};
      \node[bver] (n-4) at   ( 7,  2)    {\begin{sideways}
          $n-4$\end{sideways}};
      \node at (-.25,2.5) {$T^{n-2}_R(u_n)=$};
      \draw[dashed] (6.5, 4) -- (8.5, 4);
      \draw[dashed] (6.5, 3) -- (8.5, 3);
      \draw[dashed] (4, 1) -- (6.5, 1);
      \node at (14.75,2.5) {.};
    \end{tikzpicture} 
  \end{center}

  Next, $u_n(n-1)=n-1$, so we must add a horizontal domino with label
  $n-1$.  We first remove then domino with label $n$ and place a
  horizontal domino with label $n-1$ at the end of the first row.
  When we replace the domino with label $n$ we see that it overlaps
  fully with the domino with label $n-1$, so we use case (2) of
  Definition~\ref{shuffle} to bump the domino with label $n$ to the
  end of the second row:
  \begin{center}
    \begin{tikzpicture}[node distance=0 cm,outer sep = 0pt]
      \node[bhor] (1) at (  2,  3.5)   {1};
      \node[bver] (2) at ( 1.5, 2)   {\begin{sideways}
          $2$\end{sideways}};
      \node[bhor] (3) [right = of 1]   {3};
      \node[bver] (4) [right = of 2]   {\begin{sideways}
          $4$\end{sideways}};
      \node[bhor] (5) [right  = of 3]   {5};
      \node[bver] (6) [right  = of 4] {\begin{sideways}
          $6$\end{sideways}};  
      \node[bhor] (n-5) at (9.5,3.5) {$n-5$};
      \node[bhor] (n-3) [right = of n-5] {$n-3$};
      \node[bhor] (n-1) [right = of n-3] {$n-1$};
      \node[bver] (n-4) at   ( 7,  2)    {\begin{sideways}
          $n-4$\end{sideways}};
      \node[bhor] (n) at (8.5, 2.5) {$n$};
      \node at (-.25,2.5) {$T^{n-1}_R(u_n)=$};
      \draw[dashed] (6.5, 4) -- (8.5, 4);
      \draw[dashed] (6.5, 3) -- (8.5, 3);
      \draw[dashed] (4, 1) -- (6.5, 1);
      \node at (14.75,2.5) {.};
    \end{tikzpicture} 
  \end{center}

  Finally, $u_n(n)=n-2$, so we conclude by adding a domino with label
  $n-2$.  We remove the dominoes with labels $n-1$ and $n$ and add a
  horizontal domino with label $n-2$ to the end of the first row.
  When we replace the domino with label $n-1$ it fully overlaps with
  the domino with label $n-2$, so we use case (2) of
  Definition~\ref{shuffle} to bump the domino with label $n-1$ to the
  end of the second row:
  \begin{center}
    \begin{tikzpicture}[node distance=0 cm,outer sep = 0pt]
      \node[bhor] (1) at (  2,  3.5)   {1};
      \node[bver] (2) at ( 1.5, 2)   {\begin{sideways}
          $2$\end{sideways}};
      \node[bhor] (3) [right = of 1]   {3};
      \node[bver] (4) [right = of 2]   {\begin{sideways}
          $4$\end{sideways}};
      \node[bhor] (5) [right  = of 3]   {5};
      \node[bver] (6) [right  = of 4] {\begin{sideways}
          $6$\end{sideways}};  
      \node[bhor] (n-5) at (9.5,3.5) {$n-5$};
      \node[bhor] (n-3) [right = of n-5] {$n-3$};
      \node[bhor] (n-2) [right = of n-3] {$n-2$};
      \node[bver] (n-4) at   ( 7,  2)    {\begin{sideways}
          $n-4$\end{sideways}};
      \node[bhor, fill=shade] (n-1) at (8.5, 2.5) {\shortstack{$n-1$\\\color{white}{$\boldsymbol n$}}};
      \draw[dashed] (6.5, 4) -- (8.5, 4);
      \draw[dashed] (6.5, 3) -- (8.5, 3);
      \draw[dashed] (4, 1) -- (6.5, 1);
      \node at (15,2) {.};
    \end{tikzpicture}
  \end{center}
  Then when the domino with label $n$ is replaces it fully overlaps
  with the domino with label $n-1$, so we again use case (2) of
  Definition~\ref{shuffle} to bump the domino with label $n$ to the
  end of the third row.  Then we obtain
  \begin{center}
    \begin{tikzpicture}[node distance=0 cm,outer sep = 0pt]
      \node[bhor] (1) at (  2,  3.5)   {1};
      \node[bver] (2) at ( 1.5, 2)   {\begin{sideways}
          $2$\end{sideways}};
      \node[bhor] (3) [right = of 1]   {3};
      \node[bver] (4) [right = of 2]   {\begin{sideways}
          $4$\end{sideways}};
      \node[bhor] (5) [right  = of 3]   {5};
      \node[bver] (6) [right  = of 4] {\begin{sideways}
          $6$\end{sideways}};  
      \node[bhor] (n-5) at (9.5,3.5) {$n-5$};
      \node[bhor] (n-3) [right = of n-5] {$n-3$};
      \node[bhor] (n-2) [right = of n-3] {$n-2$};
      \node[bver] (n-4) at   ( 7,  2)    {\begin{sideways}
          $n-4$\end{sideways}};
      \node at (-.25,2.5) {$T_R(u_n)=$};
      \node[bhor] (n-1) at (8.5, 2.5) {$n-1$};
      \node[bhor] (n) [below = of n-1] {$n$};
      \draw[dashed] (6.5, 4) -- (8.5, 4);
      \draw[dashed] (6.5, 3) -- (8.5, 3);
      \draw[dashed] (4, 1) -- (6.5, 1);
      \node at (15,2.5) {.};
    \end{tikzpicture}  
  \end{center}

  Now by Lemma~\ref{v_nright0} we see that $T_R(u_n)=E(T_R(v_n),C_2)$,
  so $T_R(u_n)\approx T_R(v_n)$, thus by Corollary~\ref{rightcells} we
  see that $u_n\sim_R v_n$.
\end{proof}

\begin{lemma}\label{u_n2}
  If $n\geq 8$ is such that $n\equiv 2\bmod 4$, then $v_n\sim_R
  u_n$.
\end{lemma}

\begin{proof}
  By Lemma~\ref{u_norder} we have
  \begin{center}
    \begin{tikzpicture}[node distance=0 cm,outer sep = 0pt]
      \node[bver] (1) at ( 1.5,  3)   {1};
      \node[bver] (2) [below = of 1]   {\begin{sideways}
          $2$\end{sideways}};
      \node[bhor] (3) at (   3,3.5)    {3};
      \node[bver] (4) at ( 2.5,  2)    {\begin{sideways}
          $4$\end{sideways}};
      \node[bhor] (5) [right  = of 3]   {5};
      \node[bver] (6) [right  = of 4] {\begin{sideways}
          $6$\end{sideways}};  
      \node[bhor] (n-5) at (9.5,3.5) {$n-5$};
      \node[bver] (n-4) at   ( 7,  2)    {\begin{sideways}
          $n-4$\end{sideways}};
      \node at (-.25,2) {$T^{n-4}_R(u_n)=$};
      \draw[dashed] (  6, 4) -- (8.5, 4);
      \draw[dashed] (  6, 3) -- (8.5, 3);
      \draw[dashed] (4, 1) -- (6.5, 1);
      \node at (11,2) {.};
    \end{tikzpicture}  
  \end{center}
  
  Now $u_n(n-3)=n-3$, and since $n-3$ is larger than any of the labels
  of dominoes in $T_R^{n-4}(u_n)$ we simply add a horizontal domino
  with label $n-3$ to the end of the first row.  Similarly,
  $u_n(n-2)=n$, so we add a horizontal domino with label $n$ to the
  end of the first row, thus obtaining
  
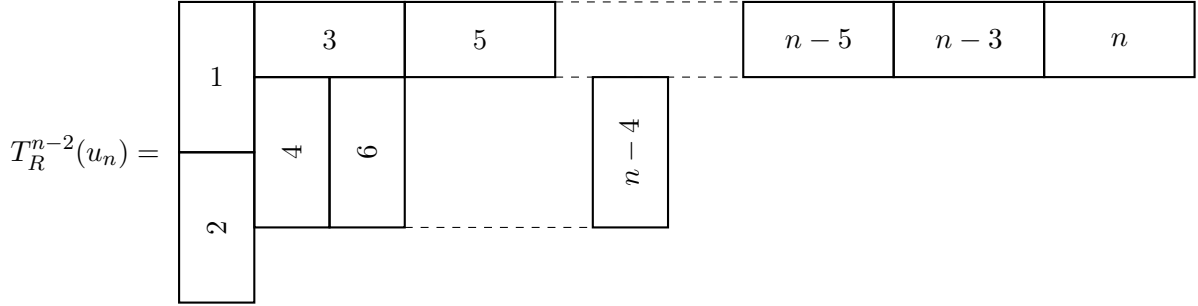
\begin{center}
    \begin{tikzpicture}[node distance=0 cm,outer sep = 0pt]
      \node[bver] (1) at ( 1.5,  3)   {1};
      \node[bver] (2) [below = of 1]   {\begin{sideways}
          $2$\end{sideways}};
      \node[bhor] (3) at (   3,3.5)    {3};
      \node[bver] (4) at ( 2.5,  2)    {\begin{sideways}
          $4$\end{sideways}};
      \node[bhor] (5) [right  = of 3]   {5};
      \node[bver] (6) [right  = of 4] {\begin{sideways}
          $6$\end{sideways}};  
      \node[bhor] (n-5) at (9.5,3.5) {$n-5$};
      \node[bhor] (n-3) [right = of n-5] {$n-3$};
      \node[bhor] (n) [right = of n-3] {$n$};
      \node[bver] (n-4) at   ( 7,  2)    {\begin{sideways}
          $n-4$\end{sideways}};
      \node at (-.25,2) {$T^{n-2}_R(u_n)=$};
      \draw[dashed] (  6, 4) -- (8.5, 4);
      \draw[dashed] (  6, 3) -- (8.5, 3);
      \draw[dashed] (4, 1) -- (6.5, 1);
      \node at (14.75,2) {.};
    \end{tikzpicture} 
  \end{center}

  Next, $u_n(n-1)=n-1$, so we must add a horizontal domino with label
  $n-1$.  We first remove then domino with label $n$ and place a
  horizontal domino with label $n-1$ at the end of the first row.
  When we replace the domino with label $n$ we see that it overlaps
  fully with the domino with label $n-1$, so we use case (2) of
  Definition~\ref{shuffle} to bump the domino with label $n$ to the
  end of the second row, obtaining 
  
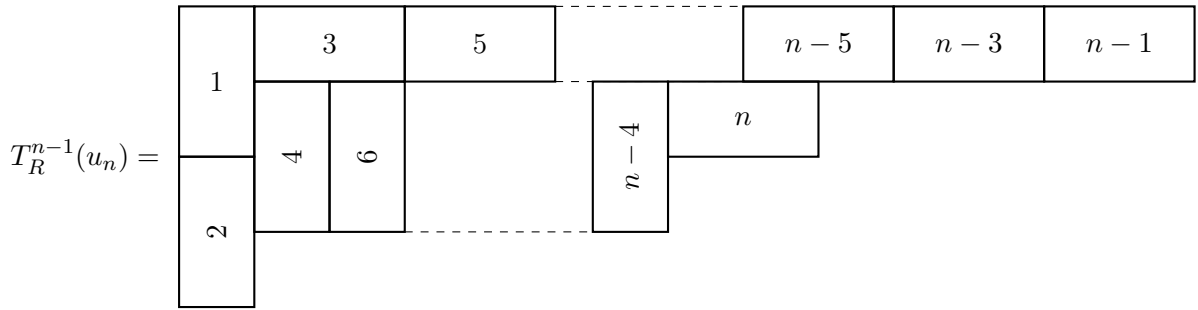
\begin{center}
    \begin{tikzpicture}[node distance=0 cm,outer sep = 0pt]
      \node[bver] (1) at ( 1.5,  3)   {1};
      \node[bver] (2) [below = of 1]   {\begin{sideways}
          $2$\end{sideways}};
      \node[bhor] (3) at (   3,3.5)    {3};
      \node[bver] (4) at ( 2.5,  2)    {\begin{sideways}
          $4$\end{sideways}};
      \node[bhor] (5) [right  = of 3]   {5};
      \node[bver] (6) [right  = of 4] {\begin{sideways}
          $6$\end{sideways}};  
      \node[bhor] (n-5) at (9.5,3.5) {$n-5$};
      \node[bhor] (n-3) [right = of n-5] {$n-3$};
      \node[bhor] (n-1) [right = of n-3] {$n-1$};
      \node[bver] (n-4) at   ( 7,  2)    {\begin{sideways}
          $n-4$\end{sideways}};
      \node[bhor] (n) at (8.5, 2.5) {$n$};
      \node at (-.25,2) {$T^{n-1}_R(u_n)=$};
      \draw[dashed] (  6, 4) -- (8.5, 4);
      \draw[dashed] (  6, 3) -- (8.5, 3);
      \draw[dashed] (4, 1) -- (6.5, 1);
      \node at (14.75,2) {.};
    \end{tikzpicture}
  \end{center}

  Finally, $u_n(n)=n-2$, so we conclude by adding a domino with label
  $n-2$.  We remove the dominoes with labels $n-1$ and $n$ and add a
  horizontal domino with label $n-2$ to the end of the first row.
  When we replace the domino with label $n-1$ it fully overlaps with
  the domino with label $n-2$, so we use case (2) of
  Definition~\ref{shuffle} to bump the domino with label $n-1$ to the
  end of the second row
  
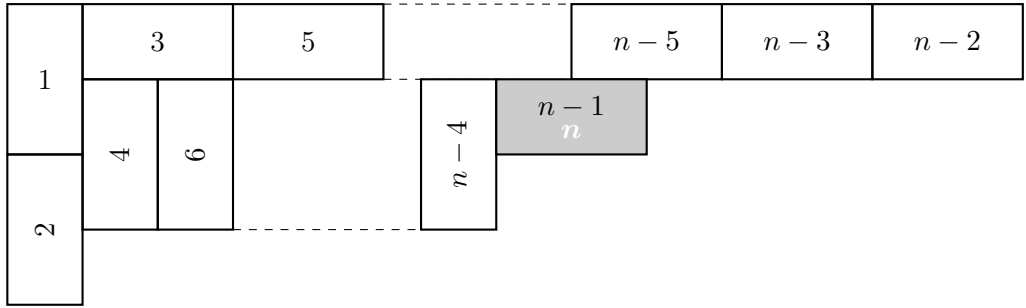
\begin{center}
    \begin{tikzpicture}[node distance=0 cm,outer sep = 0pt]
      \node[bver] (1) at ( 1.5,  3)   {1};
      \node[bver] (2) [below = of 1]   {\begin{sideways}
          $2$\end{sideways}};
      \node[bhor] (3) at (   3,3.5)    {3};
      \node[bver] (4) at ( 2.5,  2)    {\begin{sideways}
          $4$\end{sideways}};
      \node[bhor] (5) [right  = of 3]   {5};
      \node[bver] (6) [right  = of 4] {\begin{sideways}
          $6$\end{sideways}};  
      \node[bhor] (n-5) at (9.5,3.5) {$n-5$};
      \node[bhor] (n-3) [right = of n-5] {$n-3$};
      \node[bhor] (n-2) [right = of n-3] {$n-2$};
      \node[bver] (n-4) at   ( 7,  2)    {\begin{sideways}
          $n-4$\end{sideways}};
      \node[bhor, fill=shade] (n-1) at (8.5, 2.5) {\shortstack{$n-1$\\\color{white}{$\boldsymbol n$}}};
      \draw[dashed] (  6, 4) -- (8.5, 4);
      \draw[dashed] (  6, 3) -- (8.5, 3);
      \draw[dashed] (4, 1) -- (6.5, 1);
      \node at (15,2) {.};
    \end{tikzpicture}
  \end{center}

  Then when the domino with label $n$ is replaces it fully overlaps
  with the domino with label $n-1$, so we again use case (2) of
  Definition~\ref{shuffle} to bump the domino with label $n$ to the
  end of the third row. Then we obtain
  \begin{center}
    \begin{tikzpicture}[node distance=0 cm,outer sep = 0pt]
      \node[bver] (1) at ( 1.5,  3)   {1};
      \node[bver] (2) [below = of 1]   {\begin{sideways}
          $2$\end{sideways}};
      \node[bhor] (3) at (   3,3.5)    {3};
      \node[bver] (4) at ( 2.5,  2)    {\begin{sideways}
          $4$\end{sideways}};
      \node[bhor] (5) [right  = of 3]   {5};
      \node[bver] (6) [right  = of 4] {\begin{sideways}
          $6$\end{sideways}};  
      \node[bhor] (n-5) at (9.5,3.5) {$n-5$};
      \node[bhor] (n-3) [right = of n-5] {$n-3$};
      \node[bhor] (n-2) [right = of n-3] {$n-2$};
      \node[bver] (n-4) at   ( 7,  2)    {\begin{sideways}
          $n-4$\end{sideways}};
      \node[bhor] (n-1) at (8.5, 2.5) {$n-1$};
      \node[bhor] (n) [below = of n-1] {$n$};
      \node at (-.25,2) {$T_R(u_n)=$};
      \draw[dashed] (  6, 4) -- (8.5, 4);
      \draw[dashed] (  6, 3) -- (8.5, 3);
      \draw[dashed] (4, 1) -- (6.5, 1);
      \node at (15,2) {.};
    \end{tikzpicture}
  \end{center}

  Now by Lemma~\ref{v_nright2} we see that $T_R(u_n)=E(T_R(v_n),C_2)$,
  so $T_R(u_n)\approx T_R(v_n)$, thus by Corollary~\ref{rightcells} we
  see that $u_n\sim_R v_n$.
\end{proof}

\begin{proposition}\label{lrcell}
  If $n\geq 8$ is even, then $w_n\sim_{LR} u_n$.
\end{proposition}

\begin{proof}
  It follows from Lemmas \ref{v_n0} and \ref{v_n2} that $w_n\sim_L
  v_n$, and it follows from Lemmas \ref{u_n0} and \ref{u_n2} that $v_n\sim_R u_n$,
  thus we have $w_n\sim_{LR} u_n$.
\end{proof}

\section{Two-sided cells of \texorpdfstring{$w_n$}{wn} for small values of \texorpdfstring{$n$}{n}}

We have now found elements, $u_n$, in the same Kazhdan--Lusztig cell
as $w_n$ for $n\geq 8$. As we will see in Section~\ref{avalues} these
will allow us to calculate $a(w_n)$ for $n\geq 8$. We are left to
calculate simpler elements in the same two-sided cell as $w_4$ and
$w_6$.

\begin{proposition}\label{w4cell}
  We have $w_4\sim_{LR} s_1s_2s_4$.
\end{proposition}

\begin{proof}
  We know that
  \[
  w_4 = (1,\underline{4},3,\underline{2}),
  \]
  so using Remark~\ref{computer} we can calculate $T_L(w_4)$.
  \begin{center}
    \begin{tikzpicture}[node distance=4cm,>=latex']
      \node (b)
      {
        \begin{tikzpicture}[node distance=0 cm,outer sep = 0pt]
          \node[smhor] (1) at ( .5,1.5)   {\rm{1}};
          \node[rectangle, thick, minimum width=.5cm,
          minimum height=1cm] [below = of 1] {}; 
        \end{tikzpicture}
      };
      \node[below of=b,node distance=1.5cm] {$T^1_L(w_4)$};
      \node[right of=b] (c)
      {
        \begin{tikzpicture}[node distance=0 cm,outer sep = 0pt]
          \node[smhor] (1) at ( .5, 1.5)   {\rm{1}};
          \node[smver] (4) at (.25, .75)   {\rm{4}};
        \end{tikzpicture}
      };
      \node[below of=c,node distance=1.5cm] {$T^2_L(w_4)$};
      \node[right of=c] (d)
      {
        \begin{tikzpicture}[node distance=0 cm,outer sep = 0pt]
          \node[smhor] (1) at ( .5, 1.5)   {\rm{1}};
          \node[smver] (4) at (.25, .75)   {\rm{4}};
          \node[smhor] (3) [right = of 1]   {\rm{3}};
        \end{tikzpicture}
      };
      \node[below of=d,node distance=1.5cm] {$T^3_L(w_4)$};
      \node[right of=d] (e)
      {
        \begin{tikzpicture}[node distance=0 cm,outer sep = 0pt]
          \node[smhor] (1) at ( .5, 1.5)   {\rm{1}};
          \node[smver] (2) at (.25, .75)   {\rm{2}};
          \node[smhor] (3) [right = of 1]   {\rm{3}};
          \node[smver] (4) [right = of 2]   {\rm{4}};
        \end{tikzpicture}
      };
      \node[below of=e,node distance=1.5cm] {$T_L(w_4)$};
      \draw[shorten >=0.5cm,shorten <=0.5cm,->,thick] (b)--(c);
      \draw[shorten >=0.5cm,shorten <=0.5cm,->,thick] (c)--(d);
      \draw[shorten >=0.5cm,shorten <=0.5cm,->,thick] (d)--(e);
    \end{tikzpicture}
  \end{center}

  Now we shade the fixed squares and label the $r$-values
  \begin{center}
    \begin{tikzpicture}[node distance=0 cm,outer sep = 0pt]
      \node[fix] at  (  .5, 1.5) {};
      \node[fix] at  ( 1.5, 1.5) {};
      \node[fix] at  (  .5,  .5) {};
      \node[fix] at  (   1,   1) {};
      \node[smhor] (1) at (.75, 1.5)   {\rm{1}};
      \node[smver] (2) at ( .5, .75)   {\rm{2}};
      \node[smhor] (3) [right = of 1]   {\rm{3}};
      \node[smver] (4) [right = of 2]   {\rm{4}};
      \node[redr] at (   0,   1) {\hspace{-.02in}$r_1$};
      \node[redr] at (   1,   1) {\hspace{-.02in}$r_3$};
      \node[redr] at (   0,   0) {\hspace{-.02in}$r_2$};
      \node[redr] at ( 1.5, 1.5) {\hspace{-.02in}$r_4$};
      \node at (2.5,.75) {.};
    \end{tikzpicture}
  \end{center}
  so using Remark~\ref{computer} we can calculate 
  \begin{center}
    \begin{tikzpicture}[node distance=0 cm,outer sep = 0pt]
      \node[smver] (1) at ( 1, 1.5)   {\rm{1}};
      \node[smver] (2) [below = of 1]   {\rm{2}};
      \node[smhor] (3) at (1.75, 1.75)   {\rm{3}};
      \node[smhor] (4) [below  = of 3]   {\rm{4}};
      \node at (-.75,1) {$(T_L(w_4))' = $};
      \node at (2.5,1) {.};
    \end{tikzpicture}
  \end{center}

  Then $C_1=\{1,2,3\}$ and $C_2=\{4\}$ are open cycles, and
  \begin{center}
    \begin{tikzpicture}[node distance=0 cm,outer sep = 0pt]
      \tikzstyle{ver}=[rectangle, draw, thick, minimum width=.5cm,
      minimum height=1cm]         
      \tikzstyle{hor}=[rectangle, draw, thick, minimum width=1cm,
      minimum height=.5cm]
      \node[ver] (1) at ( 1, 1.5)   {\rm{1}};
      \node[ver] (2) [below = of 1]   {\rm{2}};
      \node[hor] (3) at (1.75, 1.75)   {\rm{3}};
      \node[hor] (4) [below  = of 3]   {\rm{4}};
      \node at (-1.25,1) {$E(T_L(w_4),C_1,C_2) = $};
      \node at (2.5,1) {.};
    \end{tikzpicture}
  \end{center}

  Now as a signed permutation we have
  \[
  s_1s_2s_4=(\underline{1},\underline{2},4,3),
  \]
  so we can calculate $T_L(s_1s_2s_4)$.
  \begin{center}
    \begin{tikzpicture}[node distance=4cm,>=latex']
      \node (b)
      {
        \begin{tikzpicture}[node distance=0 cm,outer sep = 0pt]
          \node[smver] (1) at (  .5, 1.5)   {\rm{1}};
          \node[rectangle, thick, minimum width=.5cm,
          minimum height=1cm] [below = of 1] {}; 
        \end{tikzpicture}
      };
      \node[below of=b,node distance=1.5cm] {$T^1_L(s_1s_2s_4)$};
      \node[right of=b] (c)
      {
        \begin{tikzpicture}[node distance=0 cm,outer sep = 0pt]
          \node[smver] (1) at (  .5, 1.5)   {\rm{1}};
          \node[smver] (2) [below = of 1]  {\rm{2}};
        \end{tikzpicture}
      };
      \node[below of=c,node distance=1.5cm] {$T^2_L(s_1s_2s_4)$};
      \node[right of=c] (d)
      {
        \begin{tikzpicture}[node distance=0 cm,outer sep = 0pt]
          \node[smver] (1) at (  .5, 1.5)   {\rm{1}};
          \node[smver] (2) [below = of 1]   {\rm{2}};
          \node[smhor] (4) at (1.25, 1.75)  {\rm{4}};
        \end{tikzpicture}
      };
      \node[below of=d,node distance=1.5cm] {$T^3_L(s_1s_2s_4)$};
      \node[right of=d] (e)
      {
        \begin{tikzpicture}[node distance=0 cm,outer sep = 0pt]
          \node[smver] (1) at (  .5, 1.5)   {\rm{1}};
          \node[smver] (2) [below = of 1]   {\rm{2}};
          \node[smhor] (3) at (1.25, 1.75)  {\rm{3}};
          \node[smhor] (4) [below = of 3]   {\rm{4}};
        \end{tikzpicture}
      };
      \node[below of=e,node distance=1.5cm] {$T_L(s_1s_2s_4)$};
      \draw[shorten >=0.5cm,shorten <=0.5cm,->,thick] (b)--(c);
      \draw[shorten >=0.5cm,shorten <=0.5cm,->,thick] (c)--(d);
      \draw[shorten >=0.5cm,shorten <=0.5cm,->,thick] (d)--(e);
    \end{tikzpicture}
  \end{center}

  Then $T_L(s_1s_2s_3)=E(T_L(w_4),C_1,C_2)$, thus by
  Theorem~\ref{leftcells} $w_4\sim_{L}s_1s_2s_4$, so
  $w_4\sim_{LR}s_1s_2s_4$.
\end{proof}

\begin{lemma}\label{w6cell}
  We have $w_6\sim_{LR}s_1s_2s_6s_5s_6$.
\end{lemma}

\begin{proof}
  We found that
  \begin{center}
    \begin{tikzpicture}[node distance=0 cm,outer sep = 0pt]
      \tikzstyle{ver}=[rectangle, draw, thick, minimum width=.5cm,
      minimum height=1cm]         
      \tikzstyle{hor}=[rectangle, draw, thick, minimum width=1cm,
      minimum height=.5cm]
      \node[ver] (1) at (   1,  1.5)   {\rm{1}};
      \node[ver] (2) [below  = of 1]   {\rm{2}};
      \node[hor] (3) at (1.75, 1.75)   {\rm{3}};
      \node[ver] (4) at ( 1.5,    1)   {\rm{4}};
      \node[hor] (5) [right  = of 3]   {\rm{5}};
      \node[ver] (6) [right  = of 4]   {\rm{6}};
      \node at (-.25,1) {$T_L(w_6) = $};
    \end{tikzpicture}
  \end{center}
  in Example~\ref{buildT}.  Now $v_6=s_6s_5s_6 w_4$ and by
  Lemma~\ref{v_n2} we have $w_6\sim_L v_6$.  By Lemma~\ref{v_nright2}
  and Remark~\ref{v_nright6} we have
  \begin{center}
    \begin{tikzpicture}[node distance=0 cm,outer sep = 0pt]
      \tikzstyle{ver}=[rectangle, draw, thick, minimum width=.5cm,
      minimum height=1cm]         
      \tikzstyle{hor}=[rectangle, draw, thick, minimum width=1cm,
      minimum height=.5cm]
      \node[hor] (1) at ( .75, 1.75)   {\rm{1}};
      \node[ver] (2) at (  .5,    1)   {\rm{2}};
      \node[hor] (3) at (1.75, 1.75)   {\rm{3}};
      \node[hor] (5) at (1.25, 1.25)   {\rm{5}};
      \node[hor] (4) [right  = of 3]   {\rm{4}};
      \node[hor] (6) [below  = of 5]   {\rm{6}};
      \node at (-.75, 1.25) {$T_R(v_6)=$};
      \node at (3.5, 1.25) {,};
    \end{tikzpicture}
  \end{center}
  and $C_2=\{1,2,3,4\}$ is an open cycle in $T_R(v_6)$.  If
  we move $T_R(v_6)$ through $C_2$ we get
  \begin{center}
    \begin{tikzpicture}[node distance=0 cm,outer sep = 0pt]
      \node[smver] (1) at (   1,  1.5)   {\rm{1}};
      \node[smver] (2) [below  = of 1]   {\rm{2}};
      \node[smhor] (3) at (1.75, 1.75)   {\rm{3}};
      \node[smhor] (5) [below  = of 3]   {\rm{5}};
      \node[smhor] (4) [right  = of 3]   {\rm{4}};
      \node[smhor] (6) [below  = of 5]   {\rm{6}};
      \node at ( -1,1) {$E(T_R(v_6), C_2) = $};
      \node at (3.5, 1) {.};
    \end{tikzpicture}
  \end{center}

  Now we will calculate $T_R(s_1s_2s_6s_5s_6)$.  First we can
  calculate that as a signed permutation we have
  \[
  s_1s_2s_6s_5s_6 = (s_1s_2s_6s_5s_6)^{-1} =
  (\underline{1},\underline{2},3,6,5,4).
  \]
  Then we have
  
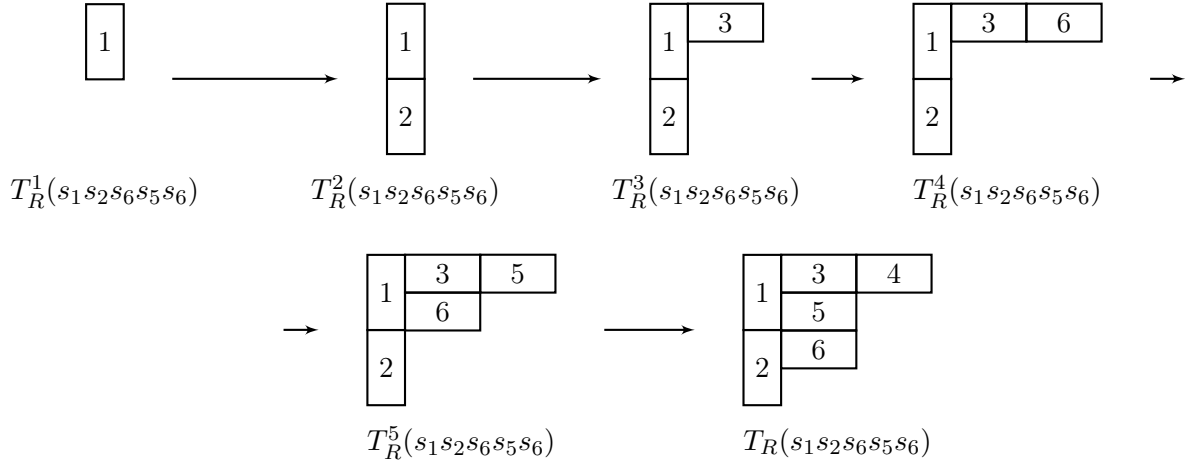
\begin{center}
    \begin{tikzpicture}[node distance=4cm,>=latex']
      \node (b)
      {
        \begin{tikzpicture}[node distance=0 cm,outer sep = 0pt]
          \node[smver] (1) at (   1,  1.5)   {\rm{1}};
          \node[rectangle, thick, minimum width=.5cm,
          minimum height=1cm] [below = of 1] {}; 
        \end{tikzpicture}
      };
      \node[below of=b,node distance=1.5cm] {$T^1_R(s_1s_2s_6s_5s_6)$};
      \node[right of=b] (c)
      {
        \begin{tikzpicture}[node distance=0 cm,outer sep = 0pt]
          \node[smver] (1) at (   1,  1.5)   {\rm{1}};
          \node[smver] (2) [below  = of 1]   {\rm{2}};
        \end{tikzpicture}
      };
      \node[below of=c,node distance=1.5cm] {$T^2_R(s_1s_2s_6s_5s_6)$};
      \node[right of=c] (d)
      {
        \begin{tikzpicture}[node distance=0 cm,outer sep = 0pt]
          \node[smver] (1) at (   1,  1.5)   {\rm{1}};
          \node[smver] (2) [below  = of 1]   {\rm{2}};
          \node[smhor] (3) at (1.75, 1.75)   {\rm{3}};
        \end{tikzpicture}
      };
      \node[below of=d,node distance=1.5cm] {$T^3_R(s_1s_2s_6s_5s_6)$};
      \node[right of=d] (e)
      {
        \begin{tikzpicture}[node distance=0 cm,outer sep = 0pt]
          \node[smver] (1) at (   1,  1.5)   {\rm{1}};
          \node[smver] (2) [below  = of 1]   {\rm{2}};
          \node[smhor] (3) at (1.75, 1.75)   {\rm{3}};
          \node[smhor] (6) [right  = of 3]   {\rm{6}};
        \end{tikzpicture}
      };
      \node[below of=e,node distance=1.5cm] {$T^4_R(s_1s_2s_6s_5s_6)$};
      \node[right of=e, node distance=2.5cm] (f) {};
      \draw[shorten >=0.5cm,shorten <=0.5cm,->,thick] (b)--(c);
      \draw[shorten >=0.5cm,shorten <=0.5cm,->,thick] (c)--(d);
      \draw[shorten >=0.5cm,shorten <=0.5cm,->,thick] (d)--(e);
      \draw[shorten <=0.5cm,->,thick] (e)--(f);
    \end{tikzpicture}
  \end{center}

  \begin{center}
    \begin{tikzpicture}[node distance=4cm,>=latex']
      \node (e) {};
      \node[right of=e, node distance=2.5cm] (f)
      {
        \begin{tikzpicture}[node distance=0 cm,outer sep = 0pt]
          \node[smver] (1) at (   1,  1.5)   {\rm{1}};
          \node[smver] (2) [below  = of 1]   {\rm{2}};
          \node[smhor] (3) at (1.75, 1.75)   {\rm{3}};
          \node[smhor] (6) [below  = of 3]   {\rm{6}};
          \node[smhor] (5) [right  = of 3]   {\rm{5}};
        \end{tikzpicture}
      };
      \node[below of=f,node distance=1.5cm] {$T^5_R(s_1s_2s_6s_5s_6)$};
      \node[right of=f, node distance=5cm] (g)
      {
        \begin{tikzpicture}[node distance=0 cm,outer sep = 0pt]
          \node[smver] (1) at (   1,  1.5)   {\rm{1}};
          \node[smver] (2) [below  = of 1]   {\rm{2}};
          \node[smhor] (3) at (1.75, 1.75)   {\rm{3}};
          \node[smhor] (5) [below  = of 3]   {\rm{5}};
          \node[smhor] (4) [right  = of 3]   {\rm{4}};
          \node[smhor] (6) [below  = of 5]   {\rm{6}};
        \end{tikzpicture}
      };
      \node[below of=g,node distance=1.5cm] {$T_R(s_1s_2s_6s_5s_6)$};
      \draw[shorten >=0.5cm,->,thick] (e)--(f);
      \draw[shorten >=0.5cm,shorten <=0.5cm,->,thick] (f)--(g);
    \end{tikzpicture}
  \end{center}
  so $E(T_R(v_6),C_2)=T_R(s_1s_2s_6s_5s_6)$.  Then $v_6\sim_R
  s_1s_2s_6s_5s_6$ by Corollary~\ref{rightcells}, thus $w_6\sim_L
  v_6\sim_R s_1s_2s_6s_5s_6$, so $w_6\sim_{LR} s_1s_2s_6s_5s_6$.
\end{proof}

\section{Proof of main result}\label{avalues}

Recall Lusztig's $a$-function, from Definition~\ref{afunction}. We now
have sufficient information to calculate Lusztig's $a$-function on bad
elements. This will allow us to bound the degree of key
Kazhdan--Lusztig polynomials and to eventually prove our main result
in Theorem~\ref{mainresult}:
\begin{theorem*}
  Let $x,w\in W(D_n)$ be such that $x$ is fully commutative. Then
  $\mu(x,w)\in\{0,1\}$.
\end{theorem*}

\begin{proposition}\label{avaluecalc}
  If $n\in\mathbb{N}$ is even, then
  \[
  a(w_n)=
  \begin{dcases}
    \frac{3n}{4} &\mbox{if }n\equiv 0\bmod 4;\\
    \frac{3n+2}{4} &\mbox{if }n\equiv 2\bmod 4.\\
  \end{dcases}
  \]
\end{proposition}

\begin{proof}
  By Proposition~\ref{w4cell}, $a(w_4)=a(s_1s_2s_4)=3$ and by
  Lemmas~\ref{aparabolic} and~\ref{w6cell}
  $a(w_6)=a(s_1s_2s_6s_5s_6)=5$.

  By Proposition~\ref{lrcell} and Lemma~\ref{avaluecells} we know that
  $a(w_n)=a(u_n)$ when $n\geq 8$.  By Lemma~\ref{u_norder}, $u_n =
  w_{n-4}s_{n}s_{n-1}s_{n}$, we can use Lemmas~\ref{alongest},
  ~\ref{aparabolic}, and~\ref{aadditive} to see that
  \[
  a(w_n)=a(w_{n-4}s_{n}s_{n-1}s_{n})=a(w_{n-4})+3.
  \]

  Then we have
  \[
  a(w_n)=
  \begin{dcases}
    3+\dfrac{3}{4}(n-4) = \frac{3n}{4} &\mbox{if }n\equiv 0\bmod 4;\\
    5+\dfrac{3}{4}(n-4) = \frac{3n+2}{4} &\mbox{if }n\equiv 2\bmod 4.\\
  \end{dcases}
  \]
\end{proof}

\begin{lemma}\label{mu0mod4}
  If $n> 8$ is such that $n\equiv 0\bmod 4$, then $\mu(x_n,w_n) = 0$.
\end{lemma}

\begin{proof}
  By Proposition~\ref{avaluecalc} and Lemma~\ref{badlength} we have
  \begin{align*}
    a(w_n) &= \frac{3n}{4}\\
    \ell(w_n) &= \frac{3n^2}{8}+\frac{n}{4}.
  \end{align*}
  Then by Proposition~\ref{abound}
  \[
  \deg(P_{e,w_n})\leq \frac{1}{2}\left(\frac{3n^2}{8} -
    \frac{n}{2}\right) = \frac{3n^2}{16} - \frac{n}{4}.
  \]
  We see that $\ell(x_n)=\frac{n}{2}+1$, so
  \[
  \frac{1}{2}(\ell(w_n)-\ell(x_n)-1) = \frac{3n^2}{16}-\frac{n}{8}-1.
  \]
  Then as long as $n>8$ we have $\deg(P_{e,w_n}) <
  \frac{1}{2}(\ell(w_n)-\ell(x_n)-1)$, so by Lemma~\ref{pew} we have
  $\mu(x_n,w_n) = 0$.
\end{proof}

\begin{lemma}\label{mu2mod4}
  If $n > 8$ is such that $n\equiv 2\bmod 4$, then $\mu(x_n,w_n) = 0$.
\end{lemma}

\begin{proof}
  By Proposition~\ref{avaluecalc} and Lemma~\ref{badlength} we have
  \begin{align*}
    a(w_n) &= \frac{3n+2}{4}\\
    \ell(w_n) &= \frac{3n^2}{8}+\frac{n}{4}.
  \end{align*}
  Then by Proposition~\ref{abound}
  \[
  \deg(P_{e,w_n})\leq \frac{1}{2}\left(\frac{3n^2}{8} -
    \frac{n}{2} - \frac{1}{2}\right) = \frac{3n^2}{16} - \frac{n}{4}-\frac{1}{4}.
  \]
  As in Lemma~\ref{mu0mod4}, $\ell(x_n)=\frac{n}{2}+1$, so
  \[
  \frac{1}{2}(\ell(w_n)-\ell(x_n)-1) = \frac{3n^2}{16}-\frac{n}{8}-1.
  \]
  Then as long as $n>6$ we have $\deg(P_{e,w_n}) <
  \frac{1}{2}(\ell(w_n)-\ell(x_n)-1)$, so by Lemma~\ref{pew} we have
  $\mu(x_n,w_n) = 0$.
\end{proof}

\begin{lemma}\label{mu4}
  We have $\mu(x_4,w_4) = 0$.
\end{lemma}

\begin{proof}
  We see that $\ell(w_4)=7$ and $\ell(x_4)=3$, so
  $\ell(w_4)\equiv\ell(x_4)\bmod 2$. Proposition~\ref{muproperties}
  shows that $\mu(x_4,w_4) = 0$.
\end{proof}

\begin{lemma}\label{mu68}
  We have $\mu(x_6,w_6) = 1$ and $\mu(x_8,w_8) = 0$.
\end{lemma}

\begin{proof}
  These values were calculated using a program called \verb,coxeter,
  developed by du Cloux \cite{du2002computing}.
\end{proof}

\begin{lemma}\label{muodd}
  Let $n\geq 5$ be odd.  Then $\mu(x_n, w_n) = \mu(x_{n-1}, w_{n-1})$.
\end{lemma}

\begin{proof}
  We know that $x_n$ and $x_{n-1}$ have the same reduced expressions,
  and by Lemma~\ref{equalbad}, $w_n$ and $w_{n-1}$ have the same
  reduced expressions, so $\mu(x_n, w_n) = \mu(x_{n-1}, w_{n-1})$.
\end{proof}

\begin{lemma}\label{muwn01}
  Let $n\geq 4$.  Then $\mu(x_n,w_n)\in\{0,1\}$.
\end{lemma}

\begin{proof}
  We have shown this for all possibilities of $n$ in Lemmas
  \ref{mu0mod4}, \ref{mu2mod4}, \ref{mu4}, \ref{mu68}, and
  \ref{muodd}.
\end{proof}

\begin{lemma}\label{xnfc}
  Let $n$ be even and let $x\in W_c$ be such that
  \[
  \mathcal{L}(x) = \mathcal{R}(x) = \mathcal{L}(x_n) =
  \mathcal{R}(x_n) = \{s_{1},s_{2},s_{4},s_{6},s_{8},\dots, s_{n}\}.
  \]
  Then $x = x_n$.
\end{lemma}

\begin{proof}
  Let $x\in W_c$ be such that $\mathcal{L}(x) = \mathcal{R}(x) =
  \{s_{1},s_{2},s_{4},s_{6},s_{8},\dots, s_{n}\}$. We see that each
  generator in $S\setminus\mathcal{L}(w) = S\setminus\mathcal{R}(w) =
  \{s_3, s_5, s_7, \dots, s_{n-1}\}$ fails to commute with at least two of the
  generators in $\mathcal{L}(w) = \mathcal{R}(w)$. Since $x\in W_c$ it
  follows that $x$ has no reduced expressions beginning or ending in
  two noncommuting generators, thus $x$ is either a product of
  commuting generators or bad. By Corollary~\ref{badnotfc}, bad elements
  cannot be fully commutative, so $x$ must be a product of commuting
  generators. Then we have $x = x_n$.
\end{proof}

\begin{lemma}\label{mubad}
  Let $W=W(D_n)$ and let $x,w\in W$ be such that $x\in W_c$ and $w$ is
  bad.  Then $\mu(x,w)\in\{0,1\}$.
\end{lemma}

\begin{proof}
  Let $x,w\in W$ be such that $x$ is fully commutative and $w$ is bad.
  If there exists $s\in\mathcal{L}(w)\setminus\mathcal{L}(x)$, or
  $s\in\mathcal{R}(w)\setminus\mathcal{R}(x)$ then we are done by
  Proposition~\ref{muproperties}.

  Now suppose $\mathcal{L}(w)\subseteq \mathcal{L}(x)$ and
  $\mathcal{R}(w)\subseteq \mathcal{R}(x)$.  By Theorem~\ref{longbad}
  we can write $w=w_k\cdot u$, where $k\leq n$ and $u$ is a product of
  commuting generators not in $\supp(w_k)$.  Suppose
  $s_i\in\mathcal{L}(x)\setminus\mathcal{L}(w)$.  Since
  $s_k\in\mathcal{L}(w)$ we see that either $1<i<k$ and $i$ is odd, or
  $i>k$.  In the first case we must have $s_{i+1}\in \mathcal{L}(x)$
  since $s_{i+1}\in\mathcal{L}(w)$, so $x$ is not fully commutative by
  Corollary~\ref{fccomm}, contradicting our assumption.  If $i>k$ then
  $x\not\leq w$ and thus $\mu(x,w) = 0$.  Then
  $s_i\in\mathcal{L}(x)\setminus\mathcal{L}(w)$ implies that
  $\mu(x,w)\in\{0,1\}$, so we may assume that
  $\mathcal{L}(x)=\mathcal{L}(w)$.  We can use a similar argument to
  show we may assume that $\mathcal{R}(x)=\mathcal{R}(w)$.

  Now since $x$ is fully commutative with
  $\mathcal{L}(x)=\mathcal{L}(w)$ and $\mathcal{R}(x)=\mathcal{R}(w)$
  we must have $x=x_k\cdot u$ by Lemma~\ref{xnfc}.  By Lemma~\ref{wnu}
  we have $\mu(x_k\cdot u, w_k\cdot u)=\mu(x_k,w_k)$, and by
  Lemma~\ref{muwn01} we have $\mu(x,w) = \mu(x_k,w_k)\in\{0,1\}$.
\end{proof}

\begin{corollary}\label{mukey}
  Let $x,w\in W$ be such that $x\in W_c$ and such that $w$ has one of the
  following properties:
  \begin{enumerate}
  \item $w$ is a product of commmuting generators;
  \item $w$ is bad; or
  \item either $\mathcal{L}(w)$ or $\mathcal{R}(w)$ is not
    commutative.
  \end{enumerate}
  Then $\mu(x,w)\in\{0,1\}$.
\end{corollary}

\begin{proof}
  If either $w$ is a product of commmuting generators or $w$ is bad
  then we are done by Corollary~\ref{mucomm} or Lemma~\ref{mubad},
  respectively. If either $\mathcal{L}(w)$ or $\mathcal{R}(w)$ is not
  commutative, then there is some generator $s$ in
  $\mathcal{L}(w)\setminus\mathcal{L}(x)$ or
  $\mathcal{R}(w)\setminus\mathcal{R}(x)$, so we are done by
  Proposition~\ref{muproperties}.
\end{proof}

We are now ready to prove our main result.

\begin{theorem}\label{mainresult}
  Let $x,w\in W(D_n)$ be such that $x$ is fully commutative. Then
  $\mu(x,w)\in\{0,1\}$.
\end{theorem}

\begin{proof}
  Let $w\in W$. If either $\mathcal{L}(w)$ or $\mathcal{R}(w)$ is not
  commutative, then the result follows from
  Proposition~\ref{muproperties}.

  By Corollary~\ref{starred} we see that $w$ is star
  reducible to an element $y\in W$ such that $y$ has one of the
  following properties:
  \begin{enumerate}
  \item $y$ is a product of commmuting generators;
  \item $y$ is bad; or
  \item either $\mathcal{L}(y)$ or $\mathcal{R}(y)$ is not
    commutative.
  \end{enumerate}
  Then we can write a sequence
  \[
  w = w_{(0)}, w_{(1)}, \dots, w_{(k-1)}, w_{(k)} = y
  \]
  such that $w_{(i)}$ is left or right star reducible to $w_{(i+1)}$.
  We will complete the proof by induction on $k$. If $k=0$ then we are
  done by Corollary~\ref{mukey}.

  Suppose that $w_{(1)} = \expl{w}{*}$ and let $s,t$ be the pair of
  generators such that $w_{(1)} = \expl{w}{*}$, $s \in
  \mathcal{L}(w)$, and $t \not\in \mathcal{L}(w)$. If
  $s\not\in\mathcal{L}(x)$ then we are done by
  Proposition~\ref{muproperties}, so suppose that
  $s\in\mathcal{L}(x)$. Then $t\not\in\mathcal{L}(x)$ by
  Corollary~\ref{fccomm} since $x$ is fully commutative, so $x\in
  D_{\mathcal{L}}(s,t)$, and thus $\expl{x}{*}$ is defined. By
  Proposition~\ref{muproperties}, $\mu(x,w) =
  \mu(\expl{x}{*},\expl{w}{*})$, and by Proposition~\ref{starfc} we
  have $\expl{x}{*}\in W_c$. Then $\expl{w}{*}$ is star reducible to
  $\expl{x}{*}$ using a sequence of length less than $k$, so $\mu(x,w)
  = \mu(\expl{x}{*},\expl{w}{*})\in \{0,1\}$ by induction.

  If $w_{(1)} = w^*$ we can use a symmetric argument to show $\mu(x,w)
  \in \{0,1\}$.
  
\end{proof}

\bibliographystyle{plain}
\bibliography{refs}	

\end{document}